\documentclass[12pt]{article}
\usepackage[utf8]{inputenc}
\usepackage{tikz}
\usepackage{amsthm}
\usepackage{amsmath}
\usepackage{amssymb}
\usepackage{amsthm}
\usepackage{amscd}
\usepackage{txfonts}
\usepackage{leftindex}
\usepackage{tikz-cd}
\usepackage[all]{xy}
\usepackage[margin=1in]{geometry}
\usepackage{comment}
\usepackage{pdfpages}
\usepackage{url}
\usepackage[T1]{fontenc}
\usepackage{amsmath,amsfonts,amssymb,enumerate}
\usepackage{amscd}
\usepackage{graphicx}
\usepackage[T1]{fontenc}
\usepackage{amsmath,amsfonts,amssymb,enumerate}
\usepackage{amsthm} 
\usepackage[english]{babel}
\newcommand{\mf}{\mathfrak}
\newcommand{\mb}{\mathbb}

\newtheorem{thm}{Theorem}[section]
\newtheorem{lem}[thm]{Lemma}
\newtheorem{prop}[thm]{Proposition}
\newtheorem{cor}[thm]{Corollary}

\newtheorem{conjecture}[thm]{Conjecture}

\theoremstyle{definition}

\newtheorem{definition}[thm]{Definition}
\newtheorem{remark}[thm]{Remark}
\newtheorem{example}[thm]{Example}

\newtheorem{algorithm}[thm]{Algorithm}
\numberwithin{equation}{section}
\newcommand{\al}{\boldsymbol{\alpha}}

\newcommand{\FF}{\mathbb{F}}

\newcommand{\s}{\mathcal{S}}
\newcommand{\m}{\mathfrak{m}}

\def\min{\mbox{\rm min}}
\def\max{\mbox{\rm max}}

\def\deg{\mbox{\rm deg}}

\def\gen{\mbox{\rm gen}}
\def\GL{\mbox{\rm GL}}

\usepackage{array}

\newcommand{\as}{\boldsymbol{a}}
\def\min{\mbox{\rm min}}
\def\max{\mbox{\rm max}}

\def\deg{\mbox{\rm deg}}
\begin{document}

	\title{ Generic models of licci ideals parametrized by Schur functors}

	\author{Lorenzo Guerrieri \thanks{ Jagiellonian University, Instytut Matematyki, Krak\'{o}w \textbf{Email address:} lorenzo.guerrieri@uj.edu.pl 
		} 
		\and Xianglong Ni
		\thanks{UC Berkeley, Department of Mathematics \textbf{Email address:} xlni@berkeley.edu } 
		\and Jerzy Weyman
		\thanks{Jagiellonian University, Instytut Matematyki, Krak\'{o}w \textbf{Email address:} jerzy.weyman@uj.edu.pl }} 
	\maketitle
	
	\maketitle

	\begin{abstract}
		\noindent 
		Let $R$ be a commutative Noetherian ring. Licci ideals are the ideals of $R$ that can be linked in a finite number of steps to a complete intersection. Each licci ideal admits a rigid deformation, and two licci ideals are in the same Herzog class if they have a common deformation.
        In this work, we show how all the Herzog classes of licci ideals of codimension 3 can be parametrized in terms of pairs of partitions associated to Schur functors. This fact allows to describe in a purely combinatorial way the (infinite) graph whose vertices correspond to Herzog classes and edges represent direct links between representatives of such classes. As applications, we obtain results on the classification of multiplications in Tor Algebras, and new structure theorems for families of licci ideals. In the final section, we extend many of these results to arbitrary codimension, but under some conjectural assumptions.
		
		\medskip
		
	\end{abstract}

	\section{Introduction}
	
	% check title and abstract \ec
	
	Linkage of ideals in commutative rings arises as an algebraic counterpart of the theory of liaison of algebraic varieties.
	%Properties related to linkage have been considered from beginning of 20th century in several works in a geometric context and
	This has been formalized in algebraic terms by Peskine and Szpiro in \cite{pes-szp}. Successively, algebraic linkage has been deeply studied in many articles, including the celebrated one of Huneke and Ulrich \cite{hu-ul1}. %but was already used before in several works in a geometric context. 
	The theory of linkage allows one to understand properties of ideals in terms of those of their linked ideals, and can be used as a strong tool in the classification of perfect ideals. In particular, a lot of studies have been focused on the ideals that can be linked to a complete intersection (i.e. to an ideal generated by a regular sequence) in a finite number of steps (see \cite{hu-ul1}, \cite{Ul2}, \cite{hu-ul4}, and related references).
	Such ideals are called \it licci\rm.
	In codimension 2, licci ideals coincide with perfect ideals, but non-licci perfect ideals abound in codimension 3.
	
	Recently, several new results have been obtained on the class of licci ideals of codimension 3, using methods coming from representation theory. 
	Let $I$ be a perfect ideal of codimension 3 in a regular local (or graded) ring $R$. The free resolution of $R/I$ has the form 
	$$  \FF: 0 \longrightarrow F_3 \buildrel{d_3}\over\longrightarrow  F_2 \buildrel{d_2}\over\longrightarrow F_1 \buildrel{d_1}\over\longrightarrow R \longrightarrow \frac{R}{I}. $$
	We denote the rank of $F_i$ by $r_i$ and say that the complex $\FF$ has format $(1,r_1,r_2,r_3)$. 
	
	In \cite{W18}, it is demonstrated how the structure theory of such free resolutions is related to a Kac-Moody Lie algebra, whose formal definition depends only on the format of the resolution. For most formats, the associated Kac-Moody Lie algebra is not finite dimensional. %and there are infinitely many higher structure maps. There are only finitely many of such maps if and only if the 
	%Lie algebra 
	%It is finite dimensional, that is: the corresponding diagram is a Dynkin diagram (this condition is equivalent to the generic ring $\widehat{R}_{gen}$ being Noetherian).
	%This happens only for sufficiently small or well-behaved formats, 
	When the algebra is finite dimensional, the format in question is called a \it Dynkin format\rm, %It is therefore of great interest to consider the properties of the ideals having minimal free resolution of such particular formats and classify them. 
	and the ideals having minimal free resolutions of such formats are called \it ideals of Dynkin type. \rm %In these cases, it is possible to describe explicitly the structure of the free resolution, generalizing the structures studied by Buchsbaum and Eisenbud in \cite{BE}, \cite{BE74} and related papers. 
	%In ... it is proved that a perfect ideal of codimension 3 is licci if and only if some of the higher structure maps corresponding to one of the fundamental representations mentioned above does not vanish modulo the maximal ideal 
	%In the context of codimension 3, 
	
	It is proved in \cite{GNW2}, \cite{GNW3} that perfect ideals of Dynkin type are all licci and they all appear as specializations of finitely many models, each of which can be described explicitly in terms of Schubert varieties. This generalizes the well-known structure theorems of Hilbert-Burch and Buchsbaum-Eisenbud (\cite{Eis},\cite{BE77}, \cite{BE74}) to other formats of resolutions. In \cite{GNW3}, it was also shown that the structure theorem relating licci ideals with specializations of Schubert varieties extends to all licci ideals of codimension 3 beyond Dynkin types. However, it was already known from \cite{CVWdynkin} that for each non-Dynkin format of length 3, there exist perfect ideals with resolution of that format that are not licci.
	
	Results in the same flavour are also obtained for Gorenstein ideals of codimension 4 \cite{CGNW}, \cite{weyman-gorenstein}.
	Other recent valuable studies on licci ideals are devoted to describe properties related with their graded Betti numbers % cite Huneke Polini Ulrich ?? \ec
	or to relate them with special type of points of the Hilbert scheme \cite{JRS}.
	
	The aim of this paper is to establish a combinatorial method to classify the models of licci ideals of codimension 3, and conjecturally also of arbitrary codimension. Our starting points are results of Buchweitz \cite{buchweitz} and Herzog \cite{Herzog}.
	
	%The perfect ideal is licci (linkage class of a complete intersection) if it can be transformed to a complete intersection by a sequence of links.
Buchweitz proved that licci ideals are strongly unobstructed. %which means that they can be deformed to ideals defining rigid algebras. 
	Herzog showed that every strongly unobstructed ideal (in particular any licci ideal) is a specialization of a unique (up to trivial deformations) ideal defining a rigid algebra. Any deformation of a licci ideal is still a licci ideal and two ideals sharing a common deformation are said to be part of the same \it Herzog class \rm (see Section \ref{sec:background-herzog} for precise definitions of these notions).
	
	As a consequence of the results of the papers \cite{GNW2}, \cite{GNW3}, we classified the Herzog classes of licci ideals of codimension $3$. The representatives of these classes are the affine patches of certain opposite Schubert varieties in certain homogeneous space of the Kac-Moody group associated to the graph $T_{2, d+1, t+1}$ (here $d$ denotes the deviation of a given ideal, and $t$ its Cohen-Macaulay type).
	
	Similar examples of licci ideals can be found in an analogous homogeneous space related to the graph $T_{c-1,d+1,t+1}$ for arbitrary codimension bigger than  $3$. These examples are conjectured to provide a family of representative of all Herzog classes of licci ideals of arbitrary codimension.
	
	For each codimension $c$, we may form a graph $\mathrm{Licci}_c$ whose vertices index the Herzog classes of licci ideals of codimension $c$, and whose edges represent direct links. The goal of the present paper is to first present a combinatorial description of this graph for $c=3$, and then to extrapolate conjecturally for $c > 3$. The vertices of the graph may be indexed by certain pairs of partitions $(\boldsymbol \lambda ,\boldsymbol\mu)$ where $|\boldsymbol\lambda |=(c-1)|\boldsymbol\mu|+1$. We will give both qualitative and algorithmic descriptions of the possible pairs of partitions that can occur. When describing vertices of $\mathrm{Licci}_c$, we will often use the notation $S_{\boldsymbol \lambda} F_1\otimes S_{\boldsymbol \mu} F_c^*$ in place of the pair $(\boldsymbol\lambda ,\boldsymbol\mu)$, where $F_1$, $F_c$ are free modules in a minimal free resolution $\FF_\bullet$ of $R/I$ and $I$ is a licci ideal in the Herzog class associated to the given vertex. The reason for this is briefly explained in Section %\ref{sec:grade-3-classification}.
\ref{sec:background-herzog}.
	% still need to revisit the following ``outline'' \ec
	
The paper is organized as follows. In Section \ref{sec:preliminaries}, we provide all the needed preliminaries from Representation Theory (i.e. Kac-Moody Lie algebras, weights, representations) and from Commutative Algebra (licci ideals, Herzog classes, higher structure maps). Then, we recall the statement of one of the main result from \cite{GNW3} (see Theorem \ref{licci1}). This result explains specifically the relation between Herzog classes and pairs of partitions.

In Section \ref{sec:liccigraph}, we first define the graph $\text{Licci}_3$. Then, in Theorem \ref{thm:graph-edges-partitions}, we explain how the pairs of partitions get modified when moving through the vertices of this graph. Indeed, given a pair of partitions $(\boldsymbol{\lambda}, \boldsymbol \mu)$ associated to a vertex of $\text{Licci}_3$, we show how to obtain the pairs of partitions associated to the adjacent vertices. Subsequently, in Theorem \ref{thm:licci-graph}, we prove that the edge between two vertices of the graph exists if and only if there exist two licci ideals in the two Herzog classes corresponding to the vertices that are directly linked. 

Section \ref{sec:comblinks} is devoted to establish the notation that will be used to study the graph $\text{Licci}_3$ with combinatorial methods and to relate the elements of the partition $\boldsymbol \lambda$ with special systems of generators of ideals. The relation between partitions and free resolutions of ideals is further explored in Section \ref{sec:freeres}. It is presented an algorithm (Algorithm \ref{algo:licci-examples}) such that, combining the use of partitions with classical linkage theory, allows to construct inductively some representative for each Herzog class. Stronger results follow from Theorem \ref{genericcomplexthm} and Lemma \ref{lem:special-graded-free-res}. It is shown that, in a graded setting (possibly non-standard), each Herzog class has some representative whose graded Betti numbers can be written explicitly in terms of the two associated partitions. If such partitions are $\boldsymbol \lambda = (\lambda_1, \ldots, \lambda_{d+3})$, $\boldsymbol \mu = (\mu_1, \ldots, \mu_t)$, and $k=\sum_{i=1}^t \mu_i$, then there exists a licci ideal $I$ in the corresponding Herzog class having minimal graded free resolutions of the form:
$$ \FF: 0 \longrightarrow \bigoplus_{i=1}^t R(-k-1 - \mu_i) \buildrel{d_3}\over\longrightarrow  R^{d+t+2}( -k -1) \buildrel{d_2}\over\longrightarrow \bigoplus_{i=1}^{d+3} R(-k-1 + \lambda_i) \buildrel{d_1}\over\longrightarrow R. $$
Conversely, Lemma \ref{lem:special-graded-free-res} shows that, under some mild assumptions, a licci ideal having such a graded minimal free resolution must be in the Herzog class associated to the pair $(\boldsymbol \lambda, \boldsymbol \mu)$. The proofs of the latter two results are based on the representation theoretic methods developed in the previous papers. For a deeper understanding of these methods, the reader will be referred to \cite[$\S3$-$\S5$]{GNW3}.

In Section \ref{sec:decorations}, we list many examples of pairs of partitions corresponding to well-known Herzog classes of licci ideals of codimension 3. We cover explicitly all Gorenstein, almost complete intersections, hyperplane sections, and ideals of Dynkin type. For each Dynkin format there exist only finitely many Herzog classes of ideals with minimal free resolution of that format.
The smallest non-Dynkin format of free resolution is $(1,6,8,3)$. In Section \ref{sec:1683}, we exhibit an infinite family of pairs of partitions (hence of distinct Herzog classes) associated to ideals with resolution of format $(1,6,8,3)$.

Section \ref{sec:admissible} is devoted to an inverse problem: understanding for which arbitrary pairs $(\boldsymbol \lambda, \boldsymbol \mu)$ there exists an associated Herzog class. Any single instance of this problem has a clear algorithmic solution, which unfortunately may need long time for large values of $\lambda_i, \mu_j$. The results in this section illustrate several simple restrictions that the elements $\lambda_i, \mu_j$ of partitions associated to Herzog classes must satisfy. The most interesting of them is the quadratic equation:
$\sum \lambda_i^2 + \sum \mu_j^2 = (k+1)^2$.

In Section \ref{sec:toralgebra}, we illustrate how pairs of partitions can be used to understand the structure of the multiplication in the Tor Algebra Tor$(R/I,R/\m)$ of licci ideals of codimension 3 (see \cite{Weyman89}, \cite{AKM88}, \cite{CVWlinkage}, \cite{CV} for a detailed background on this topic). 
This structure is an invariant of the Herzog class, and the presence of units in the multiplication maps $\bigwedge^2 F_1 \to F_2$, $F_1 \otimes F_2 \to F_3$ has an equivalent formulation in terms of $(\boldsymbol \lambda, \boldsymbol \mu)$, as pointed out in Theorem \ref{multiprelation}. Sections \ref{sec:monomials} and \ref{sec:classG} deal respectively with the partitions associated to ideals with Koszul relations among minimal syzygies, and with ideals of Tor Algebra class $G(r)$. For the latter family of ideals, we answer positively to a conjecture stated in \cite{CVWlinkage}, showing that $r$ cannot be too large in terms of the number of generators of $I$, unless $I$ is Gorenstein.

Section \ref{sec:closetoGor} is dedicated to the study of the perfect ideals of codimension 3 which are the closest ones to the Gorenstein, both in the sense of partitions, and also in the sense that their free resolutions have the largest possible perfect pairing after the one of Gorenstein ideals. 
In terms of Tor Algebra classes, we prove that this family is precisely identified by the class $G(b-3)$ where $b$ is the number of generators. We also describe the pairs of partitions associated to the ideals in this family, and provide a conjectural explicit structure theorem for the minimal free resolutions of the rigid algebras in such Herzog classes. 

Finally, in Section \ref{sec:arbitrary}, we pass to discuss pairs of partitions in arbitrary codimension $c$.
The graph $\text{Licci}_c$ and the combinatoric formula computing the partitions corresponding to adjacent vertices can be generalized exactly following the same methods used for $c=3$. However, the absence of a well-defined theory of generic rings and higher structure maps in arbitrary codimension does not allow to obtain a formal correspondence between pairs of partitions and Herzog classes of licci ideals. We propose this analogous correspondence in the form of a conjecture (see Conjecture \ref{conj:general-c}).
After that, we start exploring the theory of pairs of partitions proving several properties in higher codimension.
We conclude the paper by listing examples of several families of pairs of partitions, connecting them to known classes of licci ideals and to ideal-theoretic constructions, such as hyperplane sections and doublings.

	\section{Preliminaries}\label{sec:preliminaries}
	
	In this section we recall preliminary notions that will be needed for the rest of the paper. If not otherwise specified, our setting is the following.
	By $R$ we denote a commutative Noetherian ring. For simplicity, we assume $R$ to be regular and local or graded, with maximal ideal $\m$ and infinite residue field $K$ of characteristic zero.
    However, we point out that our results work also without assuming $R$ is regular, if the notion of codimension is replaced by grade.
    We will mainly work with perfect ideals of $R$. For a perfect ideal $I \subseteq R$ we will denote by $c,d,t$ the invariants codimension, deviation (the difference between codimension and number of generators) and Cohen-Macaulay type.
	
	We start by summarizing a few facts regarding Kac-Moody Lie algebras and their representations. These topics, in much greater detail, can be found in the first few chapters of \cite{KumarBook}.
	
	\subsection{Kac-Moody Lie Algebras and weights}
	Fix integers $d, t \geq 1$ and let $T$ denote the graph
	\[\begin{tikzcd}[column sep = small, row sep = small]
		x_1 \ar[r,dash] & u \ar[r,dash]\ar[d,dash] & y_1 \ar[r,dash] & \cdots \ar[r,dash] & y_{d} \\
		& z_1 \ar[d,dash]\\
		& \vdots \ar[d,dash]\\
		& z_{t}
	\end{tikzcd}\]
	To this graph we have an associated Kac-Moody Lie algebra $\mf{g}$ with Weyl group $W$. For a subset $t \subset T$ of vertices, let $W_t$ denote the subgroup of $W$ generated by all simple reflections $\{s_i\}_{i \notin t}$. Let $W^t$ be the set of minimal length representatives of $W/W_t$, and if $t' \subset T$ is another subset, let $\leftindex^{t'}W^t$ denote the set of minimal length representatives of $W_{t'} \backslash W / W_t$.
	
	\medskip 
	%\subsection{Weights}
	
	The Weyl group $W$ acts linearly on weights of $\mf{g}$. Explicitly,
	\begin{equation}\label{eq:W-action-on-weights}
		s_i \omega_j = \begin{cases}
			\omega_j & \text{if $i \neq j$}\\
			-\omega_j + \displaystyle\sum_{\text{$k$ adjacent to $j$ on $T$}} \omega_k & \text{if $i = j$}
		\end{cases}
	\end{equation}
	Given a weight $\omega = \sum_{i \in T} a_i \omega_i$, let $t = \{i \in T : a_i \neq 0\}$ be the support of $\omega$. The stabilizer of $\omega$ is exactly $W_t$.
	
	As such, we may identify $W^{x_1}$ or $W/W_{x_1}$ with the $W$-orbit of the fundamental weight $\omega_{x_1}$. Explicitly, $\sigma \in W$ is an element of $W^{x_1}$ if $\ell(\sigma') \geq \ell(\sigma)$ for all $\sigma'$ satisfying $\sigma' \omega_{x_1} = \sigma \omega_{x_1}$.
	
	Suppose that $\sigma \in W^{x_1}$ and $\sigma \omega_{x_1} = \sum a_i \omega_i$. If $a_j > 0$, then $s_j \sigma \in W^{x_1}$ also, with $\ell(s_j\sigma) = \ell(\sigma) + 1$. 
	
	Let $t \subset T$. An element $\sigma \in W^{x_1}$ belongs to $\leftindex^t W^{x_1}$ if $\sigma \omega_{x_1} = \sum a_i \omega_i$ with $a_i \geq 0$ for all $i \notin t$. In particular, given $\sigma \in W^{x_1}$ with $\sigma \omega_{x_1} = \sum a_i \omega_i$, we may repeatedly apply simple reflections $s_i$ for which $a_i < 0$ and $i \notin t$. Once we reach a point where all $a_i \geq 0$ for $i \notin t$, this weight corresponds to the minimal length representative of $[\sigma] \in W_t \backslash W / W_{x_1}$.
	
	\subsection{The limiting group and graph}
	In the above, the definition of $W$ and the various derivative constructions depend on the choice of parameters $d$ and $t$, which have been suppressed from the notation. In this part we will consider adjusting these parameters, so we will write $T(d,t)$ and $W(d,t)$ for what were previously denoted $T$ and $W$. Suppose that $d' \geq d$ and $t' \geq t$. Then the group $W(d,t)$ is naturally a subgroup of $W(d',t')$. We define $\overline{W} = \lim_{d,t\to\infty} W(d,t)$, and $\overline{T}$ to be the limiting graph with infinite right and bottom arms.
	
	We analogously define $\overline{W}_t$, $\overline{W}^t$, and $\leftindex^{t'}{\overline{W}}^t$. The discussion about weights remains applicable; e.g. $\overline{W}^{x_1}$ can be identified with the $\overline{W}$-orbit of $\omega_{x_1}$, with the action of $\overline{W}$ on the infinite-dimensional vector space spanned by $\{\omega_i\}_{i \in \overline{T}}$ given by the same formula \eqref{eq:W-action-on-weights}.
	
	An element $\sigma \in \overline{W}$ belongs to $W(d,t)$ if $\sigma$ is expressible as a product of simple reflections $s_i$ where $i \in T(d,t)$. Suppose that $\sigma \in \overline{W}^{x_1}$ with $\sigma\omega_{x_1} = \sum a_i \omega_i$. If $N$ is the largest integer for which $s_{y_N}$ appears in a reduced word for $\sigma$, then $a_{y_{N+1}} > 0$ and $a_{y_j} = 0$ for $j \geq N+2$. In particular, we have that $\sigma \in W(d,t)$ if $a_{y_i} = 0$ for $i \geq d + 2$ and $a_{z_i} = 0$ for $i \geq t + 2$.
	
	\subsection{Representations}\label{sec:background-reps}
	Fix $d,t$ and consider the Kac-Moody Lie algebra associated to the graph $T(d,t)$. Let $L(\lambda)$ be its irreducible representation with highest weight $\lambda$. The representation $L(\omega_{x_1})$ plays an important role in the theory of higher structure maps arising from the study of generic free resolutions of length three. It has a grading by the root lattice, which may be coarsened to a $\mb{Z}$-grading by considering the coefficient of the simple root $\alpha_{z_1}$; we call this the $z_1$-grading. Each graded component is a representation of $\mf{gl}(F_1) \times \mf{gl}(F_3)$ where $F_1 = \mb{C}^{3+d}$ and $F_3 = \mb{C}^t$. When discussing higher structure maps, we more commonly look at the (restricted) dual $L(\omega_{x_1})^\vee$, i.e. the irreducible representation with lowest weight $-\omega_{x_1}$. Its bottom three components, dual to the top three components of $L(\omega_{x_1})$, are displayed below:
	\[
	L(\omega_{x_1})^\vee=  F_1 \oplus \bigwedge^3 F_1 \otimes F_3^* \oplus \begin{matrix}
		S_{2,1^3} F_1 \otimes \bigwedge^2 F_3^*\\
		\bigwedge^5 F_1 \otimes \bigwedge^2 F_3^*\\
		\bigwedge^5 F_1 \otimes S_2 F_3^*
	\end{matrix} \oplus \cdots
	\]
	where e.g. $S_{2,1^3} F_1 = S_{2,1,1,1} F_1$ denotes a Schur functor.
	
	If $\sigma \in W^{x_1}$, then $\sigma \omega_{x_1}$ is an extremal weight in $L(\omega_{x_1})$. If moreover $\sigma \in \leftindex^{z_1} W^{x_1}$, then $\sigma \omega_{x_1}$ is the highest weight of an extremal $\mf{gl}(F_1) \times \mf{gl}(F_3)$-representation inside of $L(\omega_{x_1})$. This representation has the form $S_{\boldsymbol \lambda} F_1^* \otimes S_{\boldsymbol \mu} F_3$ for some pair of partitions $(\boldsymbol\lambda,\boldsymbol\mu)$, which are deduced from $\sigma$ as follows.
	
	View $\sigma$ as an element of $\overline{W}$ and consider its action on $\omega_{x_1}$:
	\[
	\sigma \omega_{x_1} = \sum_{i \in \overline{T}} a_i \omega_i
	\]
	Note that $a_i$ may be nonzero for $i = y_{d+1}, z_{t+1}$ even though these are not vertices in the original graph $T(d,t)$. As $\sigma$ is a minimal length representative of $W_{z_1} \backslash W / W_{x_1}$, we have that $a_i \geq 0$ for all vertices $i \neq z_1$. The partitions $\boldsymbol \lambda$ and $\boldsymbol \mu$ are read off from the coefficients of the sequences of vertices $l_{\boldsymbol\lambda} = (x_1,u,y_1,y_2,\ldots)$ and $l_{\boldsymbol\mu} = (z_2,z_3,\ldots)$ in the typical way for type A weights, i.e.
	\begin{equation}\label{eq:weights-to-partitions}
		\lambda_k = \sum_{j=k}^\infty a_{(l_\lambda)_j}, \quad \mu_k = \sum_{j=k}^\infty a_{(l_\mu)_j}
	\end{equation}
	where e.g. $(l_{\boldsymbol\lambda})_j$ denotes the $j$th vertex in the sequence $l_{\boldsymbol\lambda}$, indexing starting from $j=1$. The weight $\sigma \omega_{x_1}$ is then the highest weight of the representation $S_{\boldsymbol\lambda} F_1^* \otimes S_{\boldsymbol\mu} F_3 \subset L(\omega_{x_1})$. We will often use the associated pair of partitions $(\boldsymbol\lambda,\boldsymbol\mu)$ when discussing elements of $\sigma \in \leftindex^{z_1}{\overline{W}}^{x_1}$; notice that if $\sigma \in W(d,t)$, the pair $(\boldsymbol\lambda,\boldsymbol\mu)$ does not depend on the values of $d$ and $t$ (although the actual representation $S_{\boldsymbol\lambda} F_1 \otimes S_{\boldsymbol\mu}F_3^*$ of course does).
	
	\subsection{Licci ideals, Herzog classes, and higher structure maps}
	\label{sec:background-herzog}
	
	An ideal $I$ of $R$ is \it perfect \rm if the quotient ring $R/I$ is Cohen-Macaulay and has finite free resolution as $R$-module (equivalently if the projective dimension of $R/I$ over $R$ coincides with the codimension of $I$).
	
	Two ideals $I,J$ of $R$ are said to be linked if there exists a regular sequence $\alpha \subseteq I \cap J$ such that $I=(\alpha):J = \lbrace x \in R \, | \, xJ \subseteq (\alpha) \rbrace$ and $J=(\alpha):I$. If two ideals are linked, then they have the same height and one of the two is perfect if and only if the other one is perfect. Working over a regular ring, if $I$ is a perfect ideal of height $c$ and $\alpha$ is a regular sequence of length $c$ contained in $I$, then $(\alpha):I$ is linked to $I$. An ideal is \it licci \rm if it can be linked to a complete intersection in finitely many steps. Licci ideals are clearly perfect.
	
	Given two pairs of ring-ideal $(R,I)$ and $(S,J)$, we say that $(R,I)$ is a \it specialization \rm of $(S,J)$ if $(R,I) =  \left( \frac{S}{(\as)}, \frac{J+(\as)}{(\as)} \right)$ where $\as$ is a regular sequence in $S$ whose image in $\frac{S}{J}$ is also regular.  
	At the same time we say that  $(S,J)$ is a \it deformation \rm of $(R,I)$. 
	
	A  trivial deformation  of a pair $(R,I)$ is a deformation of the form $$(R[z_1, \ldots, z_n], IR[z_1, \ldots, z_n]),$$ or a similar local version defined with formal power series.
	In classical deformation theory, ideals defining rigid algebras are defined in terms of certain functors $T_i$. For the purpose of this project, we say that an ideal  $I$ is \it rigid \rm if $(R,I)$ admits only trivial deformations. 
	
	By results of Herzog \cite{Herzog} and Buchweitz \cite{buchweitz}, licci ideals always admit rigid deformations. We say that two licci ideals that have a common rigid deformation are in the same \it Herzog class\rm.
	As a consequence of Hilbert-Burch Theorem, in codimension $2$, the rigid ideals are exactly the ideals of maximal minors of generic $n \times (n-1)$ matrices and there is exactly one Herzog class for each $n \geq 2$.
	
	We can associate to each Herzog class in codimension $3$ a pair of partitions associated to certain Schur functors. We give a brief explanation of this method, referring \cite{GNW3}, \cite{xianglong}, for further details.
	
	%Let $I \subseteq R$ be a perfect ideal of height 3 and consider the minimal free resolution
	%\begin{equation}
	%\label{complexF}
	%\FF: 0 \longrightarrow F_3 \buildrel{d_3}\over\longrightarrow  F_2 \buildrel{d_2}\over\longrightarrow F_1 \buildrel{d_1}\over\longrightarrow R \longrightarrow \frac{R}{I}.
	%\end{equation}
	%As said in the introduction, we denote by $r_i$ the rank of $F_i$ and say that the complex $\FF$ has format $(1,r_1,r_2,r_3)$. 
	%The bases of $ F_1,F_2,F_3$ will be respectively denoted by $\lbrace e_1, \ldots, e_{r_1}\rbrace$, $\lbrace f_1, \ldots, f_{r_2} \rbrace$, $\lbrace g_1, \ldots, g_{r_3}\rbrace$.
	
	Let $I$ be a licci ideal of codimension 3 in the ring $R$. The minimal free resolution of $R/I$ has the form 
	$$  \FF: 0 \longrightarrow F_3 \buildrel{d_3}\over\longrightarrow  F_2 \buildrel{d_2}\over\longrightarrow F_1 \buildrel{d_1}\over\longrightarrow R \longrightarrow \frac{R}{I}. $$
	As done in the introduction we denote by $r_i$ the rank of $F_i$ and say that the complex $\FF$ has format $(1,r_1,r_2,r_3)$, where $r_1 = 3+d$ and $r_3 = t$. 
	
	Let us consider the Kac-Moody Lie algebra $\mf{g}$ corresponding to the graph $T(d,t)$. Some of the main tools used to develop the theory in \cite{GNW2}, \cite{GNW3} and related papers, are three fundamental representations of this Lie algebra, corresponding to the three extremal vertices of the diagram. These three fundamental representations are $\mathbb{N}$-graded and appear in the generic ring $\widehat{R}_{gen}$ associated to the format of the free resolution (see \cite{W18}). Therefore, %fixed a format and given a free resolution $\FF$ of that format defined over a commutative ring $R$, 
	each direct summand of their graded components can be mapped to $R$. This induces a sequence of $R$-linear maps involving Schur functors in the modules $F_i$ appearing in $\FF$. These maps are called \it higher structure maps \rm  (see also \cite{Gue-Wey}, \cite{GNW1}, \cite{Gue-SAF} for explicit definitions and computations of these maps). %More details on their structure will be given in Section 2.1. 
	%These three fundamental representations are $\mathbb{N}$-graded, and each direct summand corresponds to a linear map which can be computed over the ring $R$. 
	The linear maps associated to zero graded components of these representations correspond naturally to the three differentials $d_1, d_2, d_3$ of the complex $\FF$. For this reason the three fundamental representations are denoted respectively by $W(d_1), W(d_2), W(d_3)$.
	
	In \cite{BE77}, it is shown that any free resolution $\FF$ of length 3 carries an associative multiplicative structure, described by three standard maps $\bigwedge^2F_1 \to F_2, F_1 \otimes F_2 \to F_3, \bigwedge^3F_1 \to F_3$. Such maps are computed by lifting the maps obtained from the comparison between the Koszul complex defined on the generators of $I$ and the complex $\FF$. The first graded components of the three fundamental representations correspond exactly to these three linear maps defining the multiplicative structure.
	The successive graded components encode further structure which was unknown before the work done in \cite{W18}. See \cite{Lee-Weyman} for the description of all the components of these representation in the case of Dynkin formats.
	
	The representation $W(d_1)$ has the form
	$$ W(d_1)= F_1  \oplus  \bigwedge^3 F_1 \otimes F_3^*  \oplus  \bigwedge^5 F_1 \otimes S_2F_3^* \oplus  \bigwedge^4 F_1 \otimes F_1 \otimes \bigwedge^2F_3^*  \oplus \ldots  \oplus S_{\boldsymbol \lambda}F_1 \otimes S_{\boldsymbol \mu}F_3^*  \oplus \ldots, $$
	such that in the $k$-th graded component the Schur functors appearing are of the form $  S_{\boldsymbol \lambda}F_1 \otimes S_{\boldsymbol \mu}F_3^*$ where $\boldsymbol \mu $ is a partition of $k$ and $\boldsymbol \lambda $ is a partition of $2k+1$.
	Each of these pairs of Schur functors is associated to a linear map $$ w^{(1)}_{k,j}= w^{(1)}_{k,j}(I):  S_{\boldsymbol \lambda}F_1 \to S_{\boldsymbol \mu}F_3, $$ which can be computed over the ring $R$ in terms of data from the complex $\FF$ (the index $k$ denotes the graded component from which the linear map comes from, while the index $j$ is used to denote possible different maps coming from the same component). 
	We have that
	$ w^{(1)}_{0}= d_1: F_1 \to R$ and $ w^{(1)}_{1}: \bigwedge^3 F_1 \to F_3$ are respectively the first differential of $\FF$ and the multiplication map obtained from the comparison with the Koszul complex.
	%In the second graded component we have maps $ w^{(1)}_{2,1}: \bigwedge^5 F_1 \to S_2F_3$, \\ $ w^{(1)}_{2,2}: \bigwedge^4 F_1 \otimes F_1 \to \bigwedge^2F_3$. 
	Also some of the successive maps can be computed explicitly by lifting cycles defined in terms of previous higher structure maps (coming also from the other fundamental representations) over some exact complex. %\cite{Gue-Wey}, \cite{GNW1}.
	In particular higher structure maps are not unique.
	
	%Denote by $w^{(1)}_j$ the higher structure maps corresponding to the critical representation $W(d_1)$ associated to the differential $d_1$. We write these maps in terms of Schur functors in the form $w^{(1)}_j: S_{\boldsymbol{\lambda}}F_1 \otimes S_{\boldsymbol{\mu}}F_3^* \longrightarrow R,$ where $\boldsymbol{\lambda}$ denotes a partition of $2j+1$ and $\boldsymbol{\mu}$ denotes a partition of $j$.
	
	A key fact used in \cite{GNW3} is that the higher structure maps transform elegantly under linkage. Each vertex of the diagram $T$ induces a grading on the Lie algebra $\mf{g}$ and its representations. Let $\begin{bmatrix}
		\alpha_1 & \alpha_2 & \alpha_3
	\end{bmatrix}$ be the restriction of $w^{(1)}$ to its 3-dimensional bottom $(y_1,z_1)$-bigraded component. If the $\alpha_i$ form a regular sequence, then the restriction of $w^{(1)}$ to its bottom $y_1$-graded component yields the ideal $I' = (\alpha_1,\alpha_2,\alpha_3):I$. Furthermore, $R/I'$ admits a resolution $\mb{F}'$ of format $(1,3+t,2+d+t,d)$ and a choice of $w'$ specializing the generic free resolution to $\mb{F}'$ such that $w'^{(1)} = w^{(1)}$ if we identify $T(d,t)$ and $T(t,d)$ in the evident way.
	
	We may paraphrase one of the main results of \cite{GNW3} as follows. 
	
	\begin{thm}
		\label{licci1}
		Let $\mf{m}$ be the maximal ideal of $R$ and $\Bbbk = R/\mf{m}$ the residue field. Let $I \subset R$ be a perfect ideal of height 3, and suppose we have computed $w^{(1)}$ for a resolution of $R/I$. The ideal $I$ is licci exactly when $w^{(1)} \otimes \Bbbk \neq 0$. Assuming this to be the case, $I$ is classified up to deformation by the lowest (in Bruhat order) extremal $\mf{gl}(F_1) \times \mf{gl}(F_3)$-representation $S_{\boldsymbol\lambda} F_1 \otimes S_{\boldsymbol\mu} F_3^*$ on which $w^{(1)} \otimes \Bbbk$ is nonzero, i.e. it is classified by the map
		\[
			w^{(1)}_{k,j}= w^{(1)}_{k,j}(I):  S_{\boldsymbol \lambda}F_1 \to S_{\boldsymbol \mu}F_3
		\]
		with minimal $k = \kappa(I)$ that is nonzero mod $\mf{m}$.
		
		In fact, there exists a choice of $\mb{F}$ and $w$ such that $w^{(1)} \otimes \Bbbk$ is nonzero only on the lowest weight space of $S_{\boldsymbol \lambda}F_1 \to S_{\boldsymbol \mu}F_3$, which has weight $-\sigma \omega_{x_1}$ for some $\sigma \in \leftindex^{z_1}W(d,t)^{x_1}$. We refer to the combinatorial data of either $\sigma$ or equivalently $(\boldsymbol\lambda,\boldsymbol\mu)$ as the Herzog class of $I$. Every $\sigma \in \leftindex^{z_1}W(d,t)^{x_1}$ is the Herzog class of some codimension 3 licci ideal with deviation $\leq d$ and type $\leq t$, with the exception of $\sigma = \mathrm{id}$ which corresponds to the unit ideal.
	\end{thm}
	
	We can therefore canonically associate a pair of Schur functors to each height 3 licci ideal $I$: \begin{equation}
		\label{schurfunctor}
		\mathcal{S}_I= S_{\boldsymbol{\lambda}}F_1 \otimes S_{\boldsymbol{\mu}}F_3^*,
	\end{equation}   where $\boldsymbol{\lambda}$ is a partition of $2k+1$ and $\boldsymbol{\mu}$ is a partition of $k$. For example, complete intersections are characterized by the fact that the multiplication $\bigwedge^3 \operatorname{Tor}_1(R/I,\Bbbk) \to \operatorname{Tor}_3(R/I,\Bbbk)$ is nonzero. This multiplication appears inside of $w^{(1)} \otimes \Bbbk$:
	\[
	\left(\bigwedge^3 F_1 \otimes F_3^*\right) \otimes \Bbbk \subset L(\omega_{x_1})^\vee \otimes \Bbbk \to \Bbbk.
	\]
	Thus, for a complete intersection $I$ we have $\mathcal{S}_I= \bigwedge^3 F_1 \otimes F_3^*$ i.e. $\boldsymbol{\lambda} = (1,1,1)$ and $\boldsymbol{\mu} = (1)$.%This Schur functor corresponds to the source of the higher structure map $w^{(1)}_k$ not vanishing modulo $\m$.
	
	% Each ideal licci $I$ is a specialization of a generic (rigid?) ideal $H$, described as defining ideal of a Schubert variety.  The ideal $H$ is associated to the same Schur functor $\mathcal{S}_I$. Many linkage properties of $I$ can be studied only in terms of this Schur functor. \ec

	%In \cite{GNW3}, it is proved that a perfect ideal of height 3 is licci if and only if there exists some choice of images of higher structure maps such that exactly one map $  w^{(1)}_{k,j}(I): S_{\boldsymbol \lambda}F_1 \to S_{\boldsymbol \mu}F_3 $ is nonzero modulo $\m$. Moreover, the partitions $\boldsymbol \lambda,\boldsymbol \mu$ do not depend on this choice.
	
	%It follows that we can associate univocally a tensor product of two Schur functors $  \s_I= S_{\boldsymbol \lambda}F_1 \otimes S_{\boldsymbol \mu}F_3^*$ and hence two partitions $\boldsymbol \lambda, \boldsymbol \mu$ to each licci ideal $I$.
	
	Furthermore, higher structure maps behave well with respect to deformations and specializations. Units of a ring can deform only to units and non-units can deform only to non-units.
	Hence, for any licci ideal $I$ of codimension 3, the non-vanishing of the map $ w^{(1)}_{k,j}(I) $ modulo $\m$ identifies univocally a rigid ideal $\tilde{I}$ such that $\s_I = \s_{\tilde{I}}$.
	
	\section{Licci ideals and partitions in codimension 3}
    \label{sec:liccigraph}
	
	\subsection{Definition of $\mathrm{Licci}_3$}
	We now define $\mathrm{Licci}_3$, a graph that encodes the links between different families of codimension 3 licci ideals (Theorem~\ref{thm:licci-graph}). We first define a bipartite graph $\mathrm{Licci}^{\mathrm{bi}}_3$. Its vertex set is given by
	\[
	\leftindex^{z_1}{\overline{W}}^{x_1} \amalg \leftindex^{y_1}{\overline{W}}^{x_1}
	\]
	and there is an edge between $\sigma \in \leftindex^{z_1}{\overline{W}}^{x_1}$ and $\sigma' \in \leftindex^{y_1}{\overline{W}}^{x_1}$ exactly when their respective double cosets overlap, i.e. $(\sigma,\sigma')$ is in the image of the ``diagonal'' map
	\[
	{\overline{W}} \to {\overline{W}}_{z_1} \backslash {\overline{W}} / {\overline{W}}_{x_1} \times {\overline{W}}_{y_1} \backslash {\overline{W}} / {\overline{W}}_{x_1} \cong \leftindex^{z_1}{\overline{W}}^{x_1} \times \leftindex^{y_1}{\overline{W}}^{x_1}
	\]
	sending an element of ${\overline{W}}$ to its double cosets. Note that we may restrict this map to $\leftindex^{y_1,z_1} {\overline{W}}^{x_1} \subset {\overline{W}}$ without changing its image.
	
	There is a natural correspondence between the two sets of vertices $\leftindex^{z_1}{\overline{W}}^{x_1}\xrightarrow{\chi}\leftindex^{y_1}{\overline{W}}^{x_1}$ induced by exchanging the $y$ and $z$ arms of $\overline{T}$. The graph $\mathrm{Licci}_3$ is the quotient of $\mathrm{Licci}^{\mathrm{bi}}_3$ by this identification; explicitly it has vertex set $\leftindex^{z_1}{\overline{W}}^{x_1}$ and edge set
	\[
	\{ (\sigma, \chi^{-1} \sigma') : (\sigma,\sigma') \text{ is an edge in } \mathrm{Licci}^{\mathrm{bi}}_3\}.
	\]
	Note that this is an undirected graph with loops.
	
	One can analogously define $\mathrm{Licci}^{\mathrm{bi}}_3(d,t)$ to be the subgraph of $\mathrm{Licci}^{\mathrm{bi}}_3$ obtained by using $W(d,t)$ in place of $\overline{W}$ everywhere in the construction. There is an edge connecting $\sigma,\sigma' \in \mathrm{Licci}^{\mathrm{bi}}_3(d,t)$ if and only if there is one connecting them in $\mathrm{Licci}^{\mathrm{bi}}_3$.
	
	\subsection{Main properties of $\mathrm{Licci}_3$}
	In the language of partitions $(\boldsymbol\lambda,\boldsymbol\mu)$ (see Sections \ref{sec:background-reps}-\ref{sec:background-herzog}), there is another equivalent characterization of edges in $\mathrm{Licci}_3$.
	
	\begin{thm}\label{thm:graph-edges-partitions}
		Let $\sigma \in \leftindex^{z_1}{\overline{W}}^{x_1}$ correspond to the pair of partitions $(\boldsymbol\lambda,\boldsymbol\mu)$, where $\boldsymbol\lambda = (\lambda_1,\lambda_2,\ldots)$ and $\boldsymbol\mu = (\mu_1,\mu_2,\ldots)$ are viewed as infinite lists with trailing zeros. Let $\lambda'_1 \geq \lambda'_2 \geq \lambda'_3$ be parts of $\boldsymbol\lambda$ (possibly equal to zero) and let $\boldsymbol\mu^\mathrm{link}$ be the partition obtained from $\boldsymbol\lambda$ after removing these parts. Define
		\[
		\boldsymbol\lambda^\mathrm{link} = \operatorname{rsort}(\lambda'_1 + p, \lambda'_2 + p, \lambda'_3 + p,\mu_1,\mu_2,\mu_3,\ldots).
		\]
		where $\operatorname{rsort}$ means to sort in non-increasing order and
		\[
		p = \sum_{j=1}^\infty \mu^\mathrm{link}_j - \sum_{j=1}^\infty \mu_j.
		\]
		Then the pair $(\boldsymbol\lambda^\mathrm{link},\boldsymbol\mu^\mathrm{link})$ corresponds to some $\sigma^\mathrm{link} \in \leftindex^{z_1}{\overline{W}}^{x_1}$ adjacent to $\sigma$ on $\mathrm{Licci}_3$, and all neighbors of $\sigma$ are obtained in this fashion for some choice of parts $\lambda'_1 \geq \lambda'_2 \geq \lambda'_3$.
	\end{thm}
	\begin{proof}
		This is mainly an exercise in translating weights to partitions. Let $\tilde{\sigma} \in \leftindex^{y_1,z_1}{\overline{W}}^{x_1}$ represent an edge of $\mathrm{Licci}^{\mathrm{bi}}_3$ between $\sigma \in \leftindex^{z_1}{\overline{W}}^{x_1}$ and $\chi(\sigma^\mathrm{link})\in \leftindex^{y_1}{\overline{W}}^{x_1}$. Let us write
		\begin{align*}
			\tilde{\sigma}\omega_{x_1} &= \sum_{i\in \overline{T}} \tilde{a}_i \omega_i\\
			\sigma\omega_{x_1} &= \sum_{i\in \overline{T}} a_i \omega_i\\
			\chi(\sigma^\mathrm{link})\omega_{x_1} &= \sum_{i\in \overline{T}} a'_i \omega_i
		\end{align*}
		where $\tilde{a}_i \geq 0$ for all $i \notin \{y_1,z_1\}$, $a_i \geq 0$ for all $i \neq z_1$, and $a'_i \geq 0$ for all $i \neq y_1$. Our goal is to relate the coefficients $a'_i$ to the coefficients $a_i$, and then translate these to partitions via \eqref{eq:weights-to-partitions}.

		We have that $\tilde{\sigma} = \tau\sigma$ where $\tau$ is a product of simple reflections $s_i$ where $i \in \{x_1,u,y_1,y_2,\ldots\}$ and similarly $\tilde{\sigma} = \tau' \chi(\sigma^\mathrm{link})$ where $\tau'$ is a product of simple reflections $s_i$ where $i \in \{x_1,u,z_1,z_2,\ldots\}$, In particular, $\tilde{a}_{z_j} = a_{z_j}$ for $j \geq 2$ and $\tilde{a}_{y_j} = a'_{y_j}$ for $j \geq 2$.
		
		Let $\boldsymbol\lambda$ be as in \eqref{eq:weights-to-partitions}, and let $\tilde{\boldsymbol\lambda}$ be the sequence defined by $\tilde{\lambda}_k = \sum_{j=k}^\infty \tilde{a}_{(l_\lambda)_j}$. Then $\tilde{\boldsymbol\lambda}$ is just a reordering of $\boldsymbol\lambda$, satisfying $\tilde{\lambda}_1 \geq \tilde{\lambda}_2 \geq \tilde{\lambda}_3$ and $\tilde{\lambda}_4 \geq \tilde{\lambda}_5 \geq \cdots$. Furthermore, fixing $\sigma$ (and thus $\boldsymbol\lambda$), every such reordering is realized by some $\tilde{\sigma} = \tau\sigma$. Using the notation from the statement of the theorem, we can write $\tilde{\lambda}_j$ as $\lambda'_j$ for $j \in \{1,2,3\}$. Since $\tilde{a}_{y_j} = a'_{y_j}$ for $j \geq 2$, we have that $\boldsymbol\mu^\mathrm{link}$ is the second partition in the pair corresponding to $\sigma^\mathrm{link}$, as claimed.
	%Let $\boldsymbol\lambda$ be as in \eqref{eq:weights-to-partitions}, and let $\tilde{\boldsymbol\lambda}$ be the sequence defined by $\tilde{\lambda}_k = \sum_{j=k}^\infty \tilde{a}_{(l_\lambda)_j}$. Then $\tilde{\boldsymbol\lambda}$ is just a reordering of $\boldsymbol\lambda$, satisfying $\tilde{\lambda}_1 \geq \tilde{\lambda}_2 \geq \tilde{\lambda}_3$ and $\tilde{\lambda}_4 \geq \tilde{\lambda}_5 \geq \cdots$. Furthermore, fixing $\sigma$ (and thus $\boldsymbol\lambda$), every such reordering is realized by some $\tilde{\sigma} = \tau\sigma$. Using the notation from the statement of the theorem, we can write $\tilde{\lambda}_j$ as $\lambda'_j$ for $j \in \{1,2,3\}$. Since $\tilde{a}_{y_j} = a'_{y_j}$ for $j \geq 2$, we have that $\boldsymbol\mu^\mathrm{link}$ is the second partition in the pair corresponding to $\sigma^\mathrm{link}$, as claimed.
		
		Let $l'_{\boldsymbol\lambda}$ be the list of vertices $(x_1,u,z_1,z_2,\ldots)$. Define $\tilde{\boldsymbol\lambda}^\mathrm{link}$ to be the sequence
		\[
		\tilde{\lambda}^\mathrm{link}_k = \sum_{j=k}^\infty \tilde{a}_{(l'_{\boldsymbol\lambda})_j}.
		\]
		Since $\tilde{a}_{z_j} = a_{z_j}$ for $j \geq 2$, we deduce that the tail of $\tilde{\boldsymbol\lambda}^\mathrm{link}$ is given by $\mu$:
		\[
		\tilde{\boldsymbol\lambda}^\mathrm{link} = (\tilde{\lambda}^\mathrm{link}_1,\tilde{\lambda}^\mathrm{link}_2,\tilde{\lambda}^\mathrm{link}_3,\mu_1,\mu_2,\mu_3,\ldots).
		\]
		Furthermore, for $i \in \{1,2,3\}$,
		\[
		\tilde{\lambda}^\mathrm{link}_i - \tilde{\lambda}_i = \sum_{j=3}^\infty \tilde{a}_{(l'_\lambda)_j} - \sum_{j=3}^\infty \tilde{a}_{(l_\lambda)_j} = \sum_{j=1}^\infty \tilde{a}_{z_j} - \sum_{j=1}^\infty \tilde{a}_{y_j}.
		\]
		Now we observe that if $w \omega_{x_1} = \sum_{i \in T} c_i \omega_i$ where $w \in \overline{W}$ is arbitrary, we have the identity
		\[
		\sum_{j=1}^\infty j c_{y_j} = \sum_{j=1}^\infty j c_{z_j}.
		\]
		Indeed, this holds for $w = \mathrm{id}$, and it is easy to check that it is preserved by all simple reflections.
		
		We also have that
		\[
		\sum_{j=2}^\infty (j-1)\tilde{a}_{y_j} = \sum_{j=1}^\infty \mu'_j, \quad \sum_{j=2}^\infty (j-1)\tilde{a}_{z_j} = \sum_{j=1}^\infty \mu_j.
		\]
		Hence:
		\begin{align*}
			\tilde{\lambda}^\mathrm{link}_i - \tilde{\lambda}_i &= \left(\sum_{j=1}^\infty j \tilde{a}_{z_j} - \sum_{j=2}^\infty (j-1)\tilde{a}_{z_j}\right) - \left(\sum_{j=1}^\infty j \tilde{a}_{y_j} - \sum_{j=2}^\infty (j-1)\tilde{a}_{y_j}\right)\\
			&= \sum_{j=1}^\infty \mu^\mathrm{link}_j - \sum_{j=1}^\infty \mu_j,
		\end{align*}
		which was called $p$ in the theorem statement. Since $\tilde{\boldsymbol\lambda}^\mathrm{link}$ is just a reordering of $\boldsymbol\lambda^\mathrm{link}$ (analogously to how $\tilde{\boldsymbol\lambda}$ was a reordering of $\boldsymbol\lambda$), we are done.
	\end{proof}

	\begin{thm}\label{thm:licci-graph}
		Let $I \subset R$ be a codimension 3 licci ideal. Let $\sigma \in \leftindex^{z_1}{\overline{W}}^{x_1}$ be the Herzog class of $I$, and let $\sigma' \in \leftindex^{z_1}{\overline{W}}^{x_1}$ be arbitrary.
		
		Then there is an edge connecting $\sigma$ and $\sigma'$ on $\mathrm{Licci}_3$ if and only if there exists a codimension 3 licci ideal $I' \subset R$ with Herzog class $\sigma'$ satisfying $I \sim I'$.
	\end{thm}
	\begin{proof}
		We may equivalently work with the graph $\mathrm{Licci}^\mathrm{bi}_3$ replacing $\sigma'$ with $\chi(\sigma') \in \leftindex^{y_1}{\overline{W}}^{x_1}$. Take a minimal free resolution of $R/I$, with format $(1,3+d,2+d+t,t)$. Add a split exact part to this resolution so that it has format $(1,6+d,5+d+t,t)$ and call this resulting non-minimal resolution $\mb{F}$. Fix a homomorphism $w\colon \widehat{R}_\mathrm{gen} \to R$ specializing the generic free resolution of this larger format to $\mb{F}$. By adjusting $\mb{F}$ and $w$ as necessary, we may assume that $w^{(1)} \otimes \Bbbk$ is nonzero only on the lowest weight space of a single extremal $\mf{gl}(F_1)\times\mf{gl}(F_3)$-representation. The lowest weight of this representation is $-\sigma \omega_{x_1}$ where $\sigma$ is the Herzog class of $I$.
		
		For the ``if'' implication, suppose that we have an ideal $I'$ directly linked to $I$ by the regular sequence $\alpha_1,\alpha_2,\alpha_3 \in I$. We precompose $w^{(1)}$ with the action of an appropriate element of $\operatorname{GL}(F_1 \otimes R) = \operatorname{GL}_{6+d}(R)$ so that the restriction of $w^{(1)}$ to the bottom $(y_1,z_1)$-bigraded component of $L(\omega_{x_1})^\vee \otimes R$ is $\begin{bmatrix}
			\alpha_1 & \alpha_2 & \alpha_3
		\end{bmatrix}$. Note that the reason we enlarged $F_1$ in the beginning is precisely to accommodate non-minimal links here.
		
		After applying this action, $w^{(1)}$ describes higher structure maps for the linked ideal $R/I'$ when decomposed with respect to the $y_1$-grading. Let $-\tilde{\sigma}\omega_{x_1}$ be the weight of the lowest extremal weight space on which $w^{(1)} \otimes \Bbbk$ is nonzero, where $\tilde{\sigma} \in W^{x_1}$. This weight is in the same $W_{z_1}$-orbit as $-\sigma\omega_{x_1}$, and ($\chi$ applied to) the Herzog class $\sigma'$ of $I'$ is the minimal length representative of $[\tilde{\sigma}] \in W_{z_1} \backslash W / W_{x_1}$.
		
		For the ``only if'' implication, suppose that $\sigma'$ is connected to $\sigma$ on $\mathrm{Licci}_3$. Then there exists a $\tilde{\sigma} \in \leftindex^{y_1,z_1}{\overline{W}}^{x_1}$ representing both $[\sigma] \in \overline{W}_{z_1} \backslash \overline{W} / \overline{W}_{x_1}$ and $[\chi(\sigma')] \in \overline{W}_{y_1} \backslash \overline{W} / \overline{W}_{x_1}$. In particular, we have $\tilde{\sigma} = \tau \sigma$ where $\tau$ is a product of simple reflections $s_i$ where $i \in \{x_1,u,y_1,y_2,\ldots\}$. Since $\sigma \in W(d,t)$, we must have that $\tilde{\sigma} \in W(d+3,t)$ (see the proof of Theorem~\ref{thm:graph-edges-partitions}, where this is recast in terms of reordering partitions). The split part we added to our resolution accommodates this, so that $\tau$ can be represented by a permutation matrix $g\in \operatorname{GL}(F_1 \otimes R)$. If we act by $g + \epsilon$ on $w^{(1)}$, where $\epsilon$ is a general matrix whose entries are all in $\mf{m}$, the restriction of the resulting map to the bottom $(y_1,z_1)$-graded component of $L(\omega_{x_1})^\vee$ will be a regular sequence $\begin{bmatrix}
			\alpha_1 & \alpha_2 & \alpha_3
		\end{bmatrix}$. This is because $I$ has codimension 3, so 3 general linear combinations of its generators will form a regular sequence. Also, $w^{(1)}$ will then have the property that $w^{(1)} \otimes \Bbbk$ is nonzero only on the extremal weight space with weight $-\tilde{\sigma}\omega_{x_1}$. Restricting to the bottom $y_1$-graded component of $L(\omega_{x_1})^\vee$, we obtain an ideal $I'$ linked to $I$ by $(\alpha_1,\alpha_2,\alpha_3)$, whose Herzog class is $\sigma'$ by construction.
	\end{proof}
	
	\begin{definition}\label{def:SGS}
		Let $I\subset R$ be a codimension 3 licci ideal. Choose $w$ as in the proof of Theorem~\ref{thm:licci-graph} so that $w^{(1)} \otimes \Bbbk$ is nonzero only on the lowest weight space of a single extremal $\mf{gl}(F_1)\times\mf{gl}(F_3)$-representation. The bottom graded component of $w^{(1)}$ is a map $F_1 \to R$ whose image is $I$; say its entries are
		\[
			w^{(1)}_0 = \begin{bmatrix}
				h_1 & \cdots & h_b
			\end{bmatrix}
		\]
		where we use the increasing weight basis of $F_1$ (so that e.g. $h_1$ corresponds to the lowest weight space).
		We then say that the ordered list $(h_1,\ldots,h_b)$ is a \emph{special generating system} (SGS) for the ideal $I$.
	\end{definition}
	The significance of SGSs will be explained shortly, in Remark~\ref{re2}. Since higher structure maps specialize, if $\phi\colon R \to S$ is a local homomorphism and $\phi(I)S$ has codimension 3, then $\phi(h_1),\ldots,\phi(h_b)$ is an SGS for $\phi(I)S$.
	\begin{example}
		If $I\subset R$ is a codimension 3 Gorenstein ideal and $\mb{F}$ is a minimal free resolution of $R/I$, then it turns out that the condition on $w^{(1)}$ is automatically satisfied for any choice of higher structure maps. Thus any minimal generating list of $I$ is an SGS.
	\end{example}
	
	\subsection{Combinatorics of links of pairs of partitions}
    \label{sec:comblinks}
	
	%\subsection{aaa}
	
	Theorems \ref{thm:graph-edges-partitions} and \ref{thm:licci-graph} have many interesting consequences. In particular they allow us to explore Herzog classes of codimension 3 licci ideals with the use of simple combinatorial methods.
	
	%A Schur functor associated to a family of licci ideals is 
	If a pair of partitions $(\boldsymbol \lambda, \boldsymbol \mu) $ yields a valid Herzog class in the sense of Theorem \ref{licci1}, we say it is a \it decoration \rm of a vertex of $\mathrm{Licci}_3$. Equivalently, we say also that $\s_I$ is a decoration if we need to identify a specific representative of the Herzog class.
	In this case we have $\sum \mu_i = k = \kappa(I)$ and $\sum \lambda_i=2k+1$.
	This quantity $\kappa(I)$ corresponds to the graded component of $W(d_1)$ where the higher structure maps associated to $I$ sits. Roughly speaking, it is a measure of the complexity of the licci ideal $I$. We say that $k=\kappa(I)$ is the \it level \rm of $I$.
	We say that two decorations are linked if the corresponding vertices are adjacent in $\mathrm{Licci}_3$. %Equivalently, we say also that two pairs of Schur functors $\s_I$ and $\s_J$ are linked..
	
	We now give several remarks obtained as consequences of the results in the previous subsection.
	
	\begin{remark}
		\label{re1}
		Let $I $ be a licci ideal of codimension 3 minimally generated by $b=d+3$ elements and having Cohen-Macaulay type $t$.
		Suppose that the Herzog class of $I$ is associated to the decoration $(\boldsymbol \lambda, \boldsymbol \mu) $. 
		Then $\lambda_i \neq 0$ if and only if $i \leq b$ and $\mu_i \neq 0$ if and only if $i \leq t$.
	\end{remark}
	
	\begin{remark}
		\label{re2}
		Let $I $ be a licci ideal of codimension 3 with decoration  $$(\boldsymbol \lambda, \boldsymbol \mu)= ((\lambda_1, \ldots, \lambda_b),(\mu_1, \ldots, \mu_t)).$$
		We fix an SGS $x_1, \ldots, x_{b}$ of $I$ as in Definition~\ref{def:SGS}. Let $\boldsymbol{\alpha}$ be a regular sequence of length 3 in $I$. 
		Since we work over a local or graded ring with infinite residue field, by changing the basis of the ideal generated by $\boldsymbol{\alpha}$, we can set $\boldsymbol{\alpha} = \lbrace a_1, a_2, a_3 \rbrace$ such that for $j=1,2,3$:
		\begin{itemize}
			\item $a_j = \sum_{i=1}^{d+3} u_{ij}x_i$ with $u_{ij} \in R$.
			\item If $u_{ij}$ is a unit, then for every $l\neq j$, $u_{il} = 0$ and for every $r \neq i$, $u_{rj} \in \m$.
		\end{itemize}  %$a_j = \sum_{i=1}^{d+1} u_{ij}a_i$ with $u_{ij} \in R$ and, for each fixed $j$, only at most one of the $u_{ij}$ is a unit in $R$.
		%Let $\boldsymbol 0$ be a sequence of zeros of arbitrary length.
		% Any regular sequence of length $3$ in $I$ can be associated to a partition $\boldsymbol \lambda' \subseteq (\boldsymbol \lambda ,  \boldsymbol 0)$
		%$S_{\lambda_1', \lambda_2', \lambda_3'}F_1$ in the following way. 
		For $j=1,2,3$, set $\lambda_j' = \lambda_i$ if $u_{ij}$ is a unit and $\lambda_j' = 0 $ if $a_j \in \m I$ (none of the coefficients $u_{ij}$ is a unit). By convention, we permute the generators in such a way that $ \lambda_1' \geq \lambda_2' \geq \lambda_3'. $ %We set also for $i=1, \ldots, d+1$, $\lambda_i'' = \lambda_i$ if $\lambda_i \not \in \lbrace \lambda_1', \lambda_2', \lambda_3' \rbrace$ and equal to zero otherwise. 
		We set then $\boldsymbol \lambda'' :=  \boldsymbol \lambda \setminus \boldsymbol \lambda'.$
		Let $J= (\boldsymbol \alpha) : I$ and denote by $G_1$ and $G_3$ the first and last module in the minimal free resolution of $R/J$. The proofs of Theorems \ref{thm:graph-edges-partitions} and \ref{thm:licci-graph} show that the decoration of $J$ is \begin{equation}
			\label{linkageformulaeq}
			S_{\lambda'_1+p, \lambda'_2+p, \lambda'_3+p, \mu_1, \ldots, \mu_t}G_1 \otimes S_{\lambda_1'', \ldots, \lambda''_{b}}G_3^*,
		\end{equation}  where $p=\kappa(J) - \kappa(I)$ and the first partition has to be reordered in non-increasing order. 
        
        The key point (and the reason why an SGS is necessary, as opposed to an arbitrary generating set) is that we may precompose $w^{(1)}$ with the action of an appropriate element of $\GL(F_1)$ that moves our regular sequence to the bottom $(y_1,z_1)$-bigraded component while having $w^{(1)} \otimes \Bbbk$ be nonzero only on the extremal weight space with weight $-\tilde{\sigma}\omega_{x_1}$, where $\tilde{\sigma}$ is as in the proof of Theorem \ref{thm:graph-edges-partitions}. This element of $\GL(F_1)$ may be taken to be the product of a lower-triangular matrix and a permutation matrix.
	\end{remark}
	
	\begin{remark}
		\label{re3}
		The integer $p$ represents how much $\kappa(J)$ increase or decrease with respect to $k= \kappa(I)$. Recalling that $\sum \lambda_i = 2k+1$, we obtain the equation
		\begin{equation}
			\label{p}
			p= \sum_{i=1}^{b} \lambda''_i - k = k + 1 - \lambda_1'-  \lambda_2'- \lambda_3'.
		\end{equation} 
	\end{remark}
	
	Classically, a link from an ideal $I$ to a linked ideal $J=(\al):I$ is called \it minimal \rm if the elements of $\al$ are in $I \setminus \m I$.

		For this reason, we define a link from $\s_I$ to $\s_J$ to be \it minimal \rm if $\lambda_1', \lambda_2', \lambda_3' \neq 0$.
		%no $\gamma$ appear in the representation of $\sigma$, that is if $\lambda'_1, \lambda'_2, \lambda'_3 \neq 0.$ 
		A link $\s_J$ of a decoration $\s_I$ is the \it smallest minimal \rm if it is obtained choosing $\lambda'_j = \lambda_j$ for $j=1,2,3$; in this case $p$ is the minimal possible. %and $\s_J = \s_{J'}$.
		The largest possible $p$ is achieved if $\lambda_j'=0$ for $j=1,2,3$. This case corresponds to what is classically denoted as a generic link.
		
		More specifically, we identify four relevant examples of links of decorations that will be useful later in the paper to study several applications of this theory. 
		
		\begin{example} 
			\label{minimalminimal} (Smallest minimal linkage) \\
			Here $\lambda'_i = \lambda_i$ for $i=1,2,3$, $\lambda''_i= \lambda_{i+3}$ for $i=1, \ldots, d$ and 
			$ p= \sum_{i \geq 4} \lambda_i - k. $ Hence,
			$$ \s_J= \min\min(\s)=  S_{\lambda_1+p, \lambda_2+p, \lambda_3+p, \mu_1, \ldots, \mu_t} G_1 \otimes S_{\lambda_4, \ldots, \lambda_{d+3}} G_3^*. $$
		\end{example}
		
		\begin{example}
			\label{genericlink} (Generic linkage) \\
			In this case %$\lambda'_i = 0$ for $i=1,2,3$ and $\lambda''_i= \lambda_{i}$ for $i=1, \ldots, d+3$.
			%all $s_i$ have been replaced by $\gamma$'s. 
			%We have 
			$\lambda'_i = 0$ for $i=1,2,3$, $\lambda''_i= \lambda_{i}$ for $i=1, \ldots, d+3$ and 
			$ p= \sum_{i \geq 1} \lambda_i - k = k+1 $ is the maximal possible. Hence,
			$$ \s_J= \gen(\s)= S_{k+1, k+1, k+1, \mu_1, \ldots, \mu_t} G_1 \otimes S_{\lambda_1, \ldots, \lambda_{d+3}} G_3^*. $$
		\end{example}

		\begin{example} 
			\label{maximalminimal} (Largest minimal linkage) \\
			This link is the minimal link (in the classical sense) for which $p$ is the maximal possible. It often allows to link in such a way to increase the level $k$ without increasing the total Betti number of the ideal. 
			Here $\lambda'_i = \lambda_{d+i}$ for $i=1,2,3$, $\lambda''_i= \lambda_{i}$ for $i=1, \ldots, d$ and 
			$ p= \sum_{i = 1}^d \lambda_i - k. $ Hence,
			$$ \s_J=\max\min(\s)=  S_{\lambda_{d+1}+p, \lambda_{d+2}+p, \lambda_{d+3}+p, \mu_1, \ldots, \mu_t} G_1 \otimes S_{\lambda_1, \ldots, \lambda_{d}} G_3^*. $$
		\end{example}
		
	\begin{example} 
 \label{tightdouble} (Minimal tight double link) \\
 This is a canonical double link which can be used to prove statements depending inductively on $k$. The idea is based on the tight double linkage, classically considered in \cite{KMtight}, \cite{Ul2} and several more papers.
 Here, one first takes the smallest minimal link $\s_J$ and then define a second link with respect to the three elements $\lambda_1, \lambda_2, \mu_1$. The obtained decoration is (up to reordering)
 $$ \s_{I'}=  S_{\lambda_{1}+p+q, \lambda_{2}+p+q, \mu_1+q, \lambda_4, \ldots, \lambda_{d+3}} G_1 \otimes S_{\lambda_3 +p, \mu_2, \ldots, \mu_t} G_3^*, $$ where $ p=  k+1-\lambda_1-\lambda_2-\lambda_3 $ and $q = k+p+1-\lambda_1-p-\lambda_2-p-\mu_1=\lambda_3 -\mu_1$ (thus $\mu_1+q=\lambda_3$). In Proposition \ref{restrictions}(2), we will prove that the inequality $ \lambda_3+ p < \mu_1 $ is always satisfied and therefore $\kappa(I')<\kappa(I)$. Furthermore, notice that the type of $I'$ is either $t$ or $t-1$.
 %Here $\lambda'_i = \lambda_{d+i}$ for $i=1,2,3$, $\lambda''_i= \lambda_{i}$ for $i=1, \ldots, d$ and 
 %$ p= \sum_{i = 1}^d \lambda_i - k. $ Hence,
 \end{example}	
		
		%\begin{example}
		%\label{p=0link} \\
		%If $\lambda_1 + \lambda_2 = k+1$ we can define a canonical link with $p=0$.
		%Set $\lambda_1'= \lambda_1$, $\lambda_2'= \lambda_1$
		%In this case all $s_i$ have been replaced by $\gamma$'s. We have $\lambda'_i = 0$ for $i=1,2,3$, $\lambda''_i= \lambda_{i}$ for $i=1, \ldots, d+3$ and 
		%$ p= \sum_{i \geq 1} \lambda_i - k = k+1 $ is the maximal possible. 
		%Hence,
		%$$ \s_J=  S_{k+1, k+1, k+1, \mu_1, \ldots, \mu_t} G_1 \otimes S_{\lambda_1, \ldots, \lambda_{d+3}} G_3^*. $$
		%\end{example}

		\subsection{Constructing ideals in different Herzog classes}
		\label{sec:freeres}
		
		Remark~\ref{re2} demonstrates how, if we have an SGS for a codimension 3 licci ideal $I$, we can easily determine the Herzog class of ideals $J$ directly linked to $I$. However, we do not know how to explicitly compute an SGS for $J$ given the SGS for $I$, which poses an obstacle to iterating this procedure.
		
		The goal of this section is to circumvent this and provide an algorithm for exhibiting representatives of all vertices on $\mathrm{Licci}_3$. The main point is to show that for homogeneous licci ideals $I$ whose graded free resolutions have a particular form, it is easy to read off $\s_I$ as well as an SGS for $I$. Furthermore, this property of the graded free resolution is preserved under homogeneous linkage (with some mild restrictions).
		
		We first note that, for a given $\s= S_{\boldsymbol{\lambda}} F_1 \otimes S_{\boldsymbol{\mu}} F_3^*$, there exist ideals $I$ with $\s_I = \s$ whose free resolutions have a specific form:
		\begin{thm}
			\label{genericcomplexthm}
			%Let $R$ be a polynomial ring, graded with respect to  describe this precisely \ec
			%Suppose that $I$ is a licci ideal of codimension 3 defining a rigid algebra.
			Let $\s= S_{\boldsymbol{\lambda}} F_1 \otimes S_{\boldsymbol{\mu}} F_3^*$ be a decoration of some licci ideal of codimension 3. Set
			$\boldsymbol{\lambda}= (\lambda_1, \ldots, \lambda_{d+3})$, $\boldsymbol{\mu} = (\mu_1, \ldots, \mu_t)$ and $k=\sum^{t}_{i=1} \mu_i$. Then there exists a homogeneous ideal $I$ in a polynomial ring $R$ with ideal of variables $\mf{m}$ such that $I_\mf{m}$ is licci, with $\s=\s_{I_\mf{m}}$, and the quotient ring $R/I$ has minimal graded free resolution of the form 
			\begin{equation}
				\label{genericcomplex}
				\FF: 0 \longrightarrow \bigoplus_{i=1}^t R(-k-1 - \mu_i) \buildrel{d_3}\over\longrightarrow  R^{d+t+2}( -k -1) \buildrel{d_2}\over\longrightarrow \bigoplus_{i=1}^{d+3} R(-k-1 + \lambda_i) \buildrel{d_1}\over\longrightarrow R.
			\end{equation}
		\end{thm}
		
		\begin{proof}
			We may use the generic example of the Herzog class $\s$, whose free resolution is given in \cite[\S3]{GNW3}. Indeed, this example is obtained by localizing an ideal $I$ in a polynomial ring $R$ at its ideal of variables. The free resolution of $R/I$ is $\mb{Z}^{d+t+2}$-graded, but there exists a particular coarsening to a $\mb{Z}$-grading, described in \cite[Example~3.7]{GNW3}, that yields the graded free resolution indicated above.
			
			In more detail: the ring $R$ is the coordinate ring of a Schubert cell $C_\sigma$, and the grading group $\mb{Z}^{d+t+2}$ is the root lattice, corresponding geometrically to the torus action on $C_\sigma$. We omit the details here, because they would require a significant digression into representation theory. 
		\end{proof}
        An alternative proof of this theorem is given by Algorithm~\ref{algo:licci-examples}.
        
		The resolution of $R/I$ mentioned in the preceding proof admits a particularly nice choice of higher structure maps; see \cite[\S5]{GNW3}. In particular, for the map
		\[
		w^{(1)} \colon L(\omega_{x_1})^\vee \otimes R =  [F_1 \oplus \bigwedge^3 F_1 \otimes F_3^* \oplus \cdots ]\otimes R \to R,
		\]
		the only degree zero part of the free module $L(\omega_{x_1})^\vee \otimes R$ is the lowest weight space of the representation $\s = S_{\boldsymbol{\lambda}} F_1 \otimes S_{\boldsymbol{\mu}} F_3^*$. The map $w^{(1)}$ is homogeneous of degree zero, hence the condition of Definition~\ref{def:SGS} is automatically satisfied and we conclude that the entries of $d_1$ form an SGS for $I_\mf{m}$. The ordered basis of $F_1$ is important; here it is ordered in non-decreasing degree.
		
		The key observation is that this fact is purely a consequence of the grading on the free resolution. Using this, we see that certain codimension 3 licci ideals can be classified just using their graded Betti numbers:
		\begin{lem}\label{lem:special-graded-free-res}
			Let $\s = S_{\boldsymbol{\lambda}} F_1 \otimes S_{\boldsymbol{\mu}} F_3$ be a decoration of some vertex in $\mathrm{Licci}_3$. Let
			$\boldsymbol{\lambda}= (\lambda_1, \ldots, \lambda_{d+3})$, $\boldsymbol{\mu} = (\mu_1, \ldots, \mu_t)$ and $k=\sum^{t}_{i=1} \mu_i$.
			
			Suppose $R$ is graded and $I\subset R$ is a homogeneous codimension 3 perfect ideal such that $R/I$ has a graded free resolution of the form
			\[
			\mb{F} \colon 0 \to \bigoplus_{i=1}^t R(-k-1-\mu_i) \to F_2 \to \bigoplus_{i=1}^{d+3} R(-k-1+\lambda_i) \to R.
			\]
			Let $\mf{p} \subset R$ be a prime ideal such that $R_i \subset \mf{p}$ for all $i \neq 0$. If $I_\mf{p}$ is a licci ideal in $R_\mf{p}$, then:
			\begin{enumerate}
				\item $\s_{I} = \s$.
				\item Any minimal generating set for $I$, ordered so that the degrees are non-decreasing, forms an SGS for $I_\mf{p}$ after localization.
				\item $F_2 = R^{d+t+2}(-k-1)$.
			\end{enumerate}
		\end{lem}
		\begin{remark}
			If $I \subset R$ is a homogeneous codimension 3 perfect ideal with graded free resolution as in the preceding lemma, we expect that $I_\mf{p}$ is necessarily licci in $R_\mf{p}$ (i.e. that condition is superfluous in the above statement).
		\end{remark}
		\begin{proof}
			Let $w \colon \widehat{R}_\mathrm{gen} \to R$ be a homomorphism specializing the generic free resolution of format $(1,3+d,2+d+t,t)$ to $\mb{F}$. The grading on $\mb{F}$ induces one on $\widehat{R}_\mathrm{gen}$, and the map $w$ may be chosen to be homogeneous of degree zero. By assumption, the graded modules $F_1$ and $F_3$ in the free resolution are the same as in Theorem~\ref{genericcomplexthm}. As discussed above, this results in $L(\omega_{x_1})^\vee \otimes R$ having only a single generator in degree zero: the lowest weight space of $S_{\boldsymbol{\lambda}} F_1 \otimes S_{\boldsymbol{\mu}} F_3^*$. Since $I_\mf{p}$ is licci, the entries of $w^{(1)} \otimes R_\mf{p}$ must contain a unit, and this must correspond to the aforementioned lowest weight space. The conclusions (1) and (2) then follow. 
			
			The location of the unit in $w^{(1)}$ implies that $I$ is a specialization of the generic example as referenced in the proof of Theorem~\ref{genericcomplexthm}, prior to localization. Since the degrees coincide on $w^{(1)}$, this specialization is degree-preserving. In particular, the graded free resolution for the generic example specializes to one for our particular ideal with the same shifts, establishing (3).
		\end{proof}

		\begin{lem}\label{lem:special-graded-free-res-link}
			Let $\s$, $\boldsymbol{\lambda}$, $\boldsymbol{\mu}$, $k$, and $I \subset R$ be as in the Lemma \ref{lem:special-graded-free-res}. Let each of $\lambda'_1 \geq \lambda'_2 \geq \lambda'_3$ be either a part of $\boldsymbol{\lambda}$ or equal to zero, and denote \eqref{linkageformulaeq} by $\s^\mathrm{link}$.
			
			Let $x_1,\ldots,x_{d+3}$ be any minimal generating list for $I$ with $\deg(x_i) = k+1 - \lambda_i$.	For $i=1,2,3$, take $\alpha_i$ as follows:
			\begin{itemize}
				\item If $\lambda'_i = \lambda_j$, let $\alpha_i = x_j$.
				\item If $\lambda'_i = 0$, let $\alpha_i$ be any element of $\mf{m}I$ of degree $k+1$.
			\end{itemize}
			
			Suppose that $\alpha_1,\alpha_2,\alpha_3$ is a regular sequence. Then, letting $J = (\alpha_1,\alpha_2,\alpha_3):I$, the graded free resolution of $R/J$ has the form given in Lemma~\ref{lem:special-graded-free-res} with $\s^\mathrm{link}$ in place of $\s$.
		\end{lem}
		\begin{proof}
		Lemma~\ref{lem:mapping-cone-decorations} gives a resolution of $R/J$ which is almost the desired form (see Remark~\ref{rem:mapping-cone-decorations}). For the case $c=3$, the middle module of the produced resolution for $R/J$ is generated entirely in degree $k+p+1$. Since we know that $\s_J = \s^\mathrm{link}$ from point (2) of Lemma~\ref{lem:special-graded-free-res}, and the only possible non-minimality in the resolution occurs in degree $k+p+1$, it follows that the first module in the minimal free resolution does not have any generators in that degree. Thus we reach the desired form of the resolution.
		\end{proof}
		
		Using this, we can outline an algorithm for producing codimension 3 licci ideals in each Herzog class, which has been implemented in Macaulay2 \cite{M2examples}.
        \begin{algorithm}\label{algo:licci-examples}
            Let $R$ be a polynomial ring with at least three variables, and let $\mf{m}$ be the ideal of variables. Let $\s$ be a decoration of a vertex on $\mathrm{Licci}_3$. Compute a path on $\mathrm{Licci}_3$ from $\s$ to $\s_0 = \bigwedge^3 F_1 \otimes F_3^*$ (corresponding to a complete intersection) by repeatedly applying Example~\ref{minimalminimal} to $\s$:
		\[
			\s = \s_N \sim \cdots \sim \s_1 \sim \s_0.
		\]
		We then inductively construct $I_j$ so that $\s_{(I_j)_\mf{m}} = \s_j$, starting by taking $I_0$ to be a complete intersection generated by three elements of degree one. To construct $I_{j+1}$ from $I_j$, we use Lemma~\ref{lem:special-graded-free-res-link} with appropriate $\lambda'_1, \lambda'_2,\lambda'_3$ (c.f. Theorem~\ref{thm:graph-edges-partitions}). It remains to show that $\alpha_1,\alpha_2,\alpha_3$ may be adjusted into a regular sequence, so that the conclusion of Lemma~\ref{lem:special-graded-free-res-link} holds.
        
        The crucial point is that $I_j$ has a regular sequence of minimal degree among its minimal generators. If $j = 0$, then this is obvious. If $j > 0$, the elements of the regular sequence linking $I_{j-1} \sim I_j$ will always have lowest possible degree in $I_j$ by construction, since $\s_{j-1}$ was constructed from $\s_j$ via Example~\ref{minimalminimal}.
        
        Since we are working over an infinite field, adding a general homogeneous combination of lower degree generators of $I_j$ to each of $\alpha_1,\alpha_2,\alpha_3$ will ensure that it is a regular sequence.
        \end{algorithm}

		%We also assume $\frac{1}{2} \in R$. 
		%For a matrix $A$ with entries in $R$ we always denote by $I_d(A)$ the ideal generated by its $d\times d$ minors.

		\section{Decorations and links of relevant families}
        \label{sec:decorations}
		
		In this section, first we write down the tables of all the Herzog classes linked to relevant families of licci ideals of codimension 3, then we list all the Herzog classes of Dynkin formats and then we provide an explicit infinite family of Herzog classes of ideals with minimal free resolution of format $(1,6,8,3)$, that is the smallest non-Dynkin format.
		
		\subsection{Tables of examples}
		\label{sec:tables}
		
		For all the following Herzog classes of licci ideals of codimension 3, represented by an ideal $I$, we write down a table, listing all the pairs of partitions $\boldsymbol \lambda; \boldsymbol \mu$ of the linked Herzog classes, represented by an ideal $J$, together with the partition $\lambda'_1, \lambda_2', \lambda_3'$ obtained to define the link, the invariants $\kappa(J)$ and $p$, and the format of the minimal free resolution of $R/J$.
		
		\medskip
		
		\bf Complete intersection: \rm the pair of partition associated to $\s_I$ is $\boldsymbol \lambda=(1,1,1)$, $\boldsymbol \mu=(1)$. Format of the resolution: $(1,3,3,1)$. \\
		
		%\label{gorenstein}
		\begin{tabular} {|l|c|c|c|c|}
			\hline
			$\lambda_1', \lambda_2', \lambda_3' $ & $p$ & pair of part. $\s_J$ & $\kappa(J)$ & Format res. of $J$ \\
			\hline
			$1,1,1 $ & $-1$ & $(1);(0) $ & $0$ & $J=R$ \\
			\hline
			$1, 1, 0 $ & $0$ & $(1,1,1);(1)$ & $1$ & $(1,3,3,1)$ \\
			\hline
			$1, 0, 0 $ & $1$ & $(2,1,1,1);(1,1)$ & $2$ & $(1,4,5,2)$ \\
			\hline
			$0, 0, 0 $ & $2$ & $(2,2,2,1);(1,1,1)$ & $3$ & $(1,4,6,3)$ \\
			\hline 
		\end{tabular}
		
		\medskip
		
		This correspond to the well-known fact that a complete intersection of codimension 3 can be linked only to the unit ideal, to other complete intersections, or to almost complete intersections of type $2$ and  $3$.
		
		\bigskip
		
		\bf Gorenstein ideals: \rm Gorenstein ideals of codimension 3 are licci by a result of Watanabe \cite{Watanabe}. The pair of partitions associated to $\s_I$ is $\boldsymbol \lambda=(1^{2k+1})$, $\boldsymbol \mu=(k)$, for $k\geq 2$. Format of the resolution: $(1,2k+1,2k+1,1)$. \\
		%$\s_I = \bigwedge^{2k+1} F_1 \otimes S_k F_3^*$. Format $(1,n,n,1)$ with $n= 2k+1 \geq 5$. \\
		
		%\label{gorenstein}
		\begin{tabular} {|l|c|c|c|r|}
			\hline
			$\lambda_1', \lambda_2', \lambda_3' $ & $p$ & pair of part. $\s_J$ & $\kappa(J)$ & Format res. of $J$ \\
			\hline
			$1,1,1 $ & $k-2$ & $(k, k-1,k-1,k-1); (1^{2k-2})$ & $2k-2$ & $(1,4,2k+1,2k-2)$ \\
			\hline
			$1,1, 0 $ & $k-1$ & $(k, k,k,k-1); (1^{2k-1})$ & $2k-1$ & $(1,4,2k+2,2k-1)$ \\
			\hline
			$1, 0, 0 $ & $k$ & $(k+1, k,k,k); (1^{2k})$ & $2k$ & $(1,4,2k+3,2k)$ \\
			\hline
			$0,0,0 $ & $k+1$ & 
			$(k+1, k+1,k+1,k); (1^{2k+1})$ & $2k+1$ & $(1,4,2k+4,2k+1)$ \\
			\hline
		\end{tabular}
		
		\medskip
		
		Also in this case we recover the well-known fact that a Gorenstein ideal of codimension 3, generated by $n=2k+1$ generators can be linked only to almost complete intersections of type $t$ such that $n-3 \leq t \leq n$.
		
		\bigskip
		
		\bf Almost complete intersections of even type: \rm all the almost complete intersections of codimension 3 are licci since they are directly linked to Gorenstein ideals.
		The pair of partition associated to $\s_I$ is $\boldsymbol \lambda=(j+1,j,j,j)$, $\boldsymbol \mu=(1^{2j})$, for $j\geq 1$. Format of the resolution: $(1,4,2j+3,2j)$. \\
		
		%$\s_I = S_{l+1, l, l , l} F_1 \otimes \bigwedge^{2l} F_3^*$.  \\ Format $(1,4,2l+3,2l)$ with $l\geq 1$. $\kappa(I)=2l$.  \\
		
		%\label{gorenstein}
		\begin{tabular} {|l|c|c|c|r|}
			\hline
			$\lambda_1', \lambda_2', \lambda_3' $ & $p$ & pair of part. $\s_J$ & $\kappa(J)$ & Format res. of $J$ \\
			\hline
			$j+1,j,j $ & $-j$ & $(1^{2j+1});(j)$
			& $j$ & $(1,2j+1,2j+1,1)$ \\
			\hline
			$j,j,j $ & $-j+1$ & $(1^{2j+3});(j+1)$
			& $j+1$ & $(1,2j+3,2j+3,1)$ \\
			\hline
			$j+1,j, 0 $ & $0$ & $(j+1,j,1^{2j});(j,j)$
			& $2j$ & $(1,2j+2,2j+3,2)$ \\
			\hline
			$j,j, 0 $ & $1$ & $(j+1,j+1,1^{2j+1});(j+1,j)$
			& $2j+1$ & $(1,2j+3,2j+4,2)$ \\
			\hline
			$j+1,0,0 $ & $j$ & $(2j+1,j,j,1^{2j});(j^3)$
			& $3j$ & $(1,2j+3,2j+5,3)$ \\
			\hline
			$j,0,0 $ & $j+1$ & $(2j+1,j+1,j+1,1^{2j});(j+1,j,j)$
			& $3j+1$ & $(1,2j+3,2j+5,3)$ \\
			\hline
			$0,0,0 $ & $2j+1$ & $(2j+1,2j+1,2j+1,1^{2j});(j+1,j^3)$ & $4j+1$ & $(1,2j+3,2j+6,4)$ \\
			\hline
		\end{tabular}
		
		\bigskip
		
		\bf Almost complete intersections of odd type: \rm 
		the pair of partition associated to $\s_I$ is $\boldsymbol \lambda=(j,j,j,j-1)$, $\boldsymbol \mu=(1^{2j-1})$, for $j\geq 2$. Format of the resolution: $(1,4,2j+2,2j-1)$. \\
		
		%$\s_I = S_{l, l, l , l-1} F_1 \otimes \bigwedge^{2l-1} F_3^*$.  \\ Format $(1,4,2l+2,2l-1)$ with $l\geq 2$. $\kappa(I)=2l-1$.  \\
		
		%\label{gorenstein}
		\begin{tabular} {|l|c|c|c|r|}
			\hline
			$\lambda_1', \lambda_2', \lambda_3' $ & $p$ & pair of part. $\s_J$ & $\kappa(J)$ & Format res. of $J$ \\
			\hline
			$j,j,j $ & $-j$ & $(1^{2j-1});(j-1)$
			& $j-1$ & $(1,2j-1,2j-1,1)$ \\
			\hline
			$j,j,j-1 $ & $-j+1$ & $(1^{2j+1});(j)$ & $j$ & $(1,2j+1,2j+1,1)$ \\
			\hline
			$j,j, 0 $ & $0$ & $(j,j,1^{2j-1});(j,j-1)$
			& $2j-1$ & $(1,2j+1,2j+2,2)$ \\
			\hline
			$j,j-1, 0 $ & $1$ & $(j+1,j,1^{2j});(j,j)$ & $2j$ & $(1,2j+2,2j+3,2)$ \\
			\hline
			$j, 0, 0 $ & $j$ & $(2j,j,j,1^{2j-1});(j,j,j-1)$ & $3j-1$ & $(1,2j+2,2j+4,3)$ \\
			\hline
			$j-1, 0, 0 $ & $j+1$ & $(2j,j+1,j+1,1^{2j-1});(j^3)$ & $3j$ & $(1,2j+2,2j+4,3)$ \\
			\hline
			$0, 0, 0 $ & $2j$ &  $(2j,2j,2j,1^{2j-1});(j^3,j-1)$ &  $4j-1$ & $(1,2j+2,2j+5,4)$ \\
			\hline
		\end{tabular}
		
		\bigskip
		
		\bf Hyperplane sections of perfect ideals of codimension 2: \rm perfect ideals of codimension 2 are all licci. It easily follows that their hyperplane sections are also licci. The pair of partitions associated to $\s_I$ is $\boldsymbol \lambda=(k,1^{k+1})$, $\boldsymbol \mu=(1^{k})$, for $k\geq 3$ (the cases $k=1,2$ are the complete intersection and almost complete intersection of type 2). Format of the resolution: $(1,k+2,2k+1,k)$. \\
		
		%$\s_I = S_{k, 1^{k+1}} F_1 \otimes \bigwedge^{k} F_3^*$.  \\ Format $(1,k+2,2k+1,k)$ with $\kappa(I)= k\geq 3$.   \\
		
		%\label{gorenstein}
		\begin{tabular} {|l|c|c|c|r|}
			\hline
			$\lambda_1', \lambda_2', \lambda_3' $ & $p$ & pair of part. $\s_J$ & $\kappa(J)$ & Format res. of $J$ \\
			\hline
			$k,1,1 $ & $-1$ & $(k-1,1^{k});(1^{k-1})$
			& $k-1$ & $(1,k+1,2k-1,k-1)$ \\
			\hline
			$k,1, 0 $ & $0$ & $(k,1^{k+1});(1^{k})$
			& $k$ & $(1,k+2,2k+1,k)$ \\
			\hline
			$k, 0,0 $ & $1$ & $(k+1,1^{k+2});(1^{k+1})$ & $k+1$ & $(1,k+3,2k+3,k+1)$ \\
			\hline
			$1,1,1 $ & $k-2$ & $((k-1)^3,1^{k});(k,1^{k-2})$
			& $2k-2$ & $(1,k+3,2k+2,k)$ \\
			\hline
			$1,1, 0 $ & $k-1$ & $(k,k,k-1,1^{k});(k,1^{k-1})$ & $2k-1$ & $(1,k+3,2k+2,k)$ \\
			\hline
			$1, 0, 0 $ & $k$ & $(k+1,k,k,1^{k});(k,1^{k})$ & $2k$ & $(1,k+3,2k+3,k+1)$ \\
			\hline
			$0,0, 0 $ & $k+1$ & $((k+1)^3,1^{k});(k,1^{k+1})$ & $2k+1$ & $(1,k+3,2k+4,k+2)$ \\
			\hline
		\end{tabular}
		
		\bigskip
		
		\bf Anne Brown's ideals with 5 generators and type 2: \rm 
		in \cite{annebrown}, Anne Brown studies the family of licci ideals of codimension 3 with 5 generators, type 2, and having a Koszul relation among minimal syzygies (equivalently the ones not defining a Golod ring).
		The pair of partitions associated to $\s_I$ is $\boldsymbol \lambda=(2,2,1,1,1)$, $\boldsymbol \mu=(2,1)$. Format of the resolution: $(1,5,6,2)$. \\
		%$\s_I = S_{2,2,1,1,1} F_1 \otimes S_{2,1} F_3^*$.  \\ Format $(1,5,6,2)$, $\kappa(I)= 3$.   \\
		
		\begin{tabular} {|l|c|c|c|c|}
			\hline
			$\lambda_1', \lambda_2', \lambda_3' $ & $p$ & pair of part. $\s_J$ & $\kappa(J)$ & Format res. of $J$ \\
			\hline
			$2,2,1 $ & $-1$ & $(2,1,1,1);(1,1)$
			& $2$ & $(1,4,5,2)$ \\
			\hline
			$2,1,1 $ & $0$ & 
			$(2,2,1,1,1);(2,1)$  & $3$ & $(1,5,6,2)$ \\
			\hline
			$1,1,1 $ & $1$ & $(2,2,2,2,1);(2,2)$
			& $4$ & $(1,5,6,2)$ \\
			\hline
			$2,2, 0 $ & $0$ & $(2,2,2,1);(1,1,1)$ & $3$ & $(1,4,6,3)$ \\
			\hline
			$2,1, 0 $ & $1$ & $(3,2,2,1,1);(2,1,1)$ & $4$ & $(1,5,7,3)$ \\
			\hline
			$1,1, 0 $ & $2$ & $(3,3,2,2,1);(2,2,1)$ & $5$ & $(1,5,7,3)$ \\
			\hline
			$2, 0, 0 $ & $2$ & $(4,2,2,2,1);(2,1,1,1)$ & $5$ & $(1,5,8,4)$ \\
			\hline
			$1, 0, 0 $ & $3$ & $(4,3,3,2,1);(2,2,1,1)$ & $6$ & $(1,5,8,4)$ \\
			\hline
			$0, 0, 0 $ & $4$ & $(4,4,4,2,1);(2,2,1,1,1)$ & $7$ & $(1,5,9,5)$ \\
			\hline
		\end{tabular} 
		
		\medskip
		
		As already known, the smallest minimal link of an ideal in this class is an almost complete intersections of type 2.
		
		\bigskip
		
		\bf Ideal of type $E_6$ of Celikbas, Kra\'{s}kiewicz, Laxmi, Weyman: \rm  in \cite{CKLW}, another model of ideals of codimension 3 with 5 generators and type 2 is constructed. This ideal defines a Golod ring and, as a consequence of the results in \cite{GNW2}, this and the Anne Brown's model represent the only two Herzog classes in the case $c=3$, $d=2$, $t=2$. Observing the previous table, we see that this Herzog class can be obtained by the one of Anne Brown's model performing a largest minimal link. An explicit computation of this link in the rigid case appears in \cite{kustin}. 
		The pair of partitions associated to $\s_I$ is $\boldsymbol \lambda=(2,2,2,2,1)$, $\boldsymbol \mu=(2,2)$. Format of the resolution: $(1,5,6,2)$. \\

		%$\s_I = S_{2,2,2,2,1} F_1 \otimes S_{2,2} F_3^*$.  \\ Format $(1,5,6,2)$, $\kappa(I)= 4$.   \\
		
		\begin{tabular} {|l|c|c|c|c|}
			\hline
			$\lambda_1', \lambda_2', \lambda_3' $ & $p$ & pair of part. $\s_J$ & $\kappa(J)$ & Format res. of $J$ \\
			\hline
			$2,2,2 $ & $-1$ & $(2,2,1,1,1);(2,1)$ & $3$ & $(1,5,6,2)$ \\
			\hline
			$2,2,1 $ & $0$ & $(2,2,2,2,1);(2,2)$ & $4$ & $(1,5,6,2)$ \\
			\hline
			$2,2, 0 $ & $1$ & $(3,3,2,2,1);(2,2,1)$ & $5$ & $(1,5,7,3)$ \\
			\hline
			$2, 1, 0 $ & $2$ & $(4,3,2,2,2);(2,2,2)$ & $6$ & $(1,5,7,3)$ \\
			\hline
			$2, 0, 0 $ & $3$ & $(5,3,3,2,2);(2,2,2,1)$ & $7$ & $(1,5,8,4)$ \\
			\hline
			$1, 0, 0 $ & $4$ & $(5,4,4,2,2);(2,2,2,2)$ & $8$ & $(1,5,8,4)$ \\
			\hline
			$0, 0, 0 $ & $5$ & $(5,5,5,2,2);(2,2,2,2,1)$ & $9$ & $(1,5,9,5)$ \\
			\hline
		\end{tabular}
		
		\bigskip
		
		\bf Ideals of Tor algebra class $G(k-1)$ and $\kappa(I)= k \geq 4$: \rm in Section \ref{sec:closetoGor} we study a family of Herzog classes formed by ideals that are the closest to the Gorestein ones in terms the combinatoric of the partitions. 
		The pair of partitions associated to $\s_I$ is $\boldsymbol \lambda=(2^{k-1},1,1,1)$, $\boldsymbol \mu=(k-1,1)$. Format of the resolution: $(1,k+2,k+3,2)$. \\
		
		% $\s_I = S_{2^{k-1},1,1,1} F_1 \otimes S_{k-1,1} F_3^*$.  \\ Format $(1,k+2,k+3,2)$.   \\
		
		\begin{tabular} {|l|c|c|c|c|}
			\hline
			$\lambda_1', \lambda_2', \lambda_3' $ & $p$ & pair of part. $\s_J$ & $\kappa(J)$ & Format res. of $J$ \\
			\hline
			$2,2,2 $ & $k-5$ & $(k-1,k-3,k-3,k-3, 1); (2^{k-4},1,1,1)$ & $2k-5$ & $(1,5,k+3,k-1)$ \\
			\hline
			$2,2,1 $ & $k-4$ & $(k-1,k-2, k-2, k-3, 1); (2^{k-3},1,1)$ & $2k-4$ & $(1,5,k+3,k-1)$ \\
			\hline
			$2,1,1 $ & $k-3$ & $(k-1,k-1,k-2,k-2,1); (2^{k-2},1)$ & $2k-3$ & $(1,5,k+3,k-1)$ \\
			\hline
			$1,1,1 $ & $k-2$ & $(k-1,k-1,k-1,k-1,1); (2^{k-1})$ & $2k-2$ & $(1,5,k+3,k-1)$ \\
			\hline
			$2,2, 0 $ & $k-3$ & $(k-1,k-1,k-1,k-3, 1); (2^{k-3},1,1,1)$ & $2k-3$ & $(1,5,k+4,k)$ \\
			\hline
			$2,1, 0 $ & $k-2$ & $(k,k-1,k-1,k-2,1); (2^{k-2},1,1)$ & $2k-2$ & $(1,5,k+4,k)$ \\
			\hline
			$1,1, 0 $ & $k-1$ & $(k,k,k-1,k-1,1); (2^{k-1},1)$ & $2k-1$ & $(1,5,k+4,k)$ \\
			\hline
			$2, 0, 0 $ & $k-1$ & $(k+1,k-1, k-1, k-1,1); (2^{k-2},1,1,1)$ & $2k-1$ & $(1,5,k+5,k+1)$ \\
			\hline
			$1, 0, 0 $ & $k$ & $(k+1,k,k,k-1,1); (2^{k-1},1,1)$ & $2k$ & $(1,5,k+5,k+1)$ \\
			\hline
			$0, 0, 0 $ & $k+1$ & $(k+1, k+1, k+1, k-1,1); (2^{k-1},1,1,1)$ & $2k+1$ & $(1,5,k+6,k+2)$ \\
			\hline
		\end{tabular}
		
		\subsection{Ideals of Dynkin type}
		\label{sec:dynkin}
		
		By inspection we can list all the decorations for small values of $k$. From the results in Section \ref{sec:admissible}, one can obtain a combinatorial criterion that shows that these are the only possible ones. Here, for stylistic reasons, we use for the decorations the notation $S_{\boldsymbol \lambda} F_1 \otimes S_{\boldsymbol \mu} F_3^*$. We have:
		%\begin{itemize}
		%\item $k=1$: $[(1,1,1);(1)]$.
		%\item $k=2$: $[(1^5);(2)]$, $[(2,1,1,1);(1,1)]$.
		%\item $k=3$: $[(1^7);(3)]$, $[(2,2,1,1,1);(2,1)]$, $[(3,1^4);(1^3)]$, $[(2,2,2,1);(1^3)]$.
		%$\bigwedge^7 F_1 \otimes S_3F_3^*$; $S_{2,2,1,1,1} F_1 \otimes S_{2,1} F_3^*$; $S_{3,1,1,1,1} F_1 \otimes \bigwedge^3 F_3^*$; $S_{2,2,2,1} F_1 \otimes \bigwedge^3 F_3^*$.
		%\item $k=4$: $[(1^9);(4)]$, $[(2,2,2,1,1,1);(3,1)]$, $[(3,2,1^4);(2,2)]$, $[(2^4,1);(2,2)]$, $[(3,2,2,1,1);(2,1,1)]$, $[(4,1^5);(1^4)$, $[(3,2,2,2);(1^4)]$.
		%$\bigwedge^9 F_1 \otimes S_4F_3^*$; $S_{2,2,2,1,1,1} F_1 \otimes S_{3,1} F_3^*$; $S_{3,2,1,1,1,1} F_1 \otimes S_{2,2} F_3^*$; $S_{2,2,2,2,1} F_1 \otimes S_{2,2} F_3^*$; $S_{3,2,2,1,1} F_1 \otimes S_{2,1,1} F_3^*$; $S_{4,1,1,1,1,1} F_1 \otimes \bigwedge^4 F_3^*$; $S_{3,2,2,2} F_1 \otimes \bigwedge^4 F_3^*$.
		%\end{itemize}
		
		\begin{itemize}
			\item $k=1$: $\bigwedge^3 F_1 \otimes F_3^*$.
			\item $k=2$: $\bigwedge^5 F_1 \otimes S_2F_3^*$, $S_{2,1,1,1} F_1 \otimes \bigwedge^2 F_3^*$.
			\item $k=3$: $\bigwedge^7 F_1 \otimes S_3F_3^*$, $S_{2,2,1,1,1} F_1 \otimes S_{2,1} F_3^*$, $S_{3,1,1,1,1} F_1 \otimes \bigwedge^3 F_3^*$, $S_{2,2,2,1} F_1 \otimes \bigwedge^3 F_3^*$.
			\item $k=4$: $\bigwedge^9 F_1 \otimes S_4F_3^*$, $S_{2,2,2,1,1,1} F_1 \otimes S_{3,1} F_3^*$, $S_{3,2,1,1,1,1} F_1 \otimes S_{2,2} F_3^*$, $S_{2,2,2,2,1} F_1 \otimes S_{2,2} F_3^*$, $S_{3,2,2,1,1} F_1 \otimes S_{2,1,1} F_3^*$, $S_{4,1,1,1,1,1} F_1 \otimes \bigwedge^4 F_3^*$, $S_{3,2,2,2} F_1 \otimes \bigwedge^4 F_3^*$.
		\end{itemize}
		
		We recall that the Dynkin formats of resolutions of length 3 of perfect ideals are: \begin{itemize}
			%\item Type $A_n$: $(1,3,n,n-2) $ for $n \geq 3$ (in the perfect case $n=3$ and the ideal is a complete intersection).
			\item Type $D_n$: $(1,n,n,1) $  for $n\geq 3$ (Gorestein ideals)
			and $ (1,4,n,n-3) $ for $n \geq 5$ (almost complete intersections).
			\item Type $E_6$: $ (1,5,6,2) $.
			\item Type $E_7$: $ (1,5,7,3) $ and $ (1,6,7,2) $.
			\item Type $E_8$: $ (1,5,8,4) $ and $ (1,7,8,2) $.
		\end{itemize}
		
		For each of these formats there are only finitely many Herzog classes. For more details on this see \cite{GNW2}, \cite{GNW3}, \cite{examples}. The Herzog classes of Gorenstein ideals, almost complete intersections and ideals of type $E_6$ have been listed in the previous section. We still want to provide a list of all the Herzog classes of the other Dynkin formats $E_7$ and $E_8$.
		For completeness we list all those such that $k \geq 5:$
		\begin{itemize}
			\item Format $(1,6,7,2)$: 
			$ S_{3,2^3,1^2}F_1 \otimes S_{3,2}F_3^*$, 
			$ S_{3^2,2^3,1}F_1 \otimes S_{3,3}F_3^*$, 
			$  S_{3,2^5}F_1 \otimes S_{4,2}F_3^*$,  
			$ S_{3^3,2^3}F_1\otimes S_{4,3}F_3^*$, 
			$ S_{3^5,2}F_1\otimes S_{4,4}F_3^*$.
			\item Format $(1,5,7,3)$: 
			$ S_{3^2,2^2,1}F_1\otimes S_{2,2,1}F_3^*$, $ S_{3,2^4}F_1\otimes S_{3,1,1}F_3^*$,
			$ S_{3^3,2^2}F_1\otimes S_{3,2,1}F_3^*$, 
			$ S_{4,3,2^3}F_1\otimes S_{2^3}F_3^*$, 
			$ S_{3^4,1}F_1\otimes S_{2^3}F_3^*$, 
			$ S_{4,3^3,2}F_1\otimes S_{3,2,2}F_3^*$,  $ S_{3^5}F_1\otimes S_{3^2,1}F_3^*$, 
			$ S_{4^2,3^3}F_1\otimes S_{3,3,2}F_3^*$, 
			$ S_{4^4,3}F_1\otimes S_{3,3,3}F_3^*$.
			\item Format $(1,7,8,2)$: for this format we exhibit all the decoration with $k \leq 10$. The remaining decorations can be obtained from all the decorations of the Herzog classes in the formats $\leq (1,7,8,2)$ with $k \leq 10$ (including the decoration of the unit ideal $S_1F_1 \otimes S_0 F_3^*$) by taking the dual with respect to the pair of partitions $(6^7);(10,10)$.
			We have: \\
			$ S_{2^4,1^3}F_1 \otimes S_{4,1}F_3^*$, 
			$ S_{3^2,1^5}F_1 \otimes S_{3,2}F_3^*$, 
			$ S_{3^2,2^2,1^3}F_1 \otimes S_{4,2}F_3^*$, \\
			$ S_{4,2^3,1^3}F_1 \otimes S_{3,3}F_3^*$,
			$ S_{4,3,2^3,1^2}F_1 \otimes S_{4,3}F_3^*$, 
			$ S_{3^2,2^4,1}F_1 \otimes S_{5,2}F_3^*$, \\
			$ S_{3^4,1^3}F_1 \otimes S_{4,3}F_3^*$, 
			$ S_{4,3^2,2^3,1}F_1 \otimes S_{5,3}F_3^*$,  
			$ S_{4,3^3,2,1^2}F_1 \otimes S_{4,4}F_3^*$, \\
			$ S_{4^2,2^4,1}F_1 \otimes S_{4,4}F_3^*$,  
			$ S_{5,2^6}F_1 \otimes S_{4,4}F_3^*$,  
			$ S_{4^2,3^2,2^2,1}F_1 \otimes S_{5,4}F_3^*$, \\
			$ S_{5, 3^2, 2^4}F_1 \otimes S_{5,4}F_3^*$, 
			$ S_{4,3^3,2^3}F_1 \otimes S_{6,3}F_3^*$,  
			$ S_{3^6,1}F_1 \otimes S_{6,3}F_3^*$, \\
			$ S_{5,4,3^2,2^3}F_1 \otimes S_{5,5}F_3^*$, 
			$ S_{4^3,3,2^3}F_1 \otimes S_{6,4}F_3^*$, 
			$ S_{4^3,3^2,2,1}F_1 \otimes S_{5,5}F_3^*$, \\
			$ S_{4^2,3^4,1}F_1 \otimes S_{6,4}F_3^*$, 
			$ S_{5,3^4,2^2}F_1 \otimes S_{6,4}F_3^*$.  \\
			The decoration with the largest value of $k$ appearing for the format $(1,7,8,2)$ is $ S_{6^6,1}F_1 \otimes S_{10,10}F_3^*$.
			\item Format $(1,5,8,4)$: for this format we exhibit all the decoration with $k \leq 12$. 
			The remaining decorations can be obtained from all the decorations of the Herzog classes in the formats $\leq (1,5,8,4)$ with $k \leq 12$ (including the decoration of the unit ideal $S_1F_1 \otimes S_0 F_3^*$) by taking the dual with respect to the pair of partitions $(10^5);(6^4)$.
			%All the remaining decorations can be obtained from the decorations of ideals in the formats $\leq (1,5,8,4)$ with $k \leq 12$ (including the decoration of the unit ideal $S_1F_1 \otimes S_0F_3^*$) by taking the dual with respect to the representation $S_{6^4}F_3^*\otimes S_{10^5}F_1$.
			We have: \\
			$ S_{4,2^3,1}F_1 \otimes S_{2,1^3}F_3^*$, 
			$ S_{3^3,1^2}F_1 \otimes S_{2,1^3}F_3^*$, 
			$ S_{4,3^2,2,1}F_1 \otimes S_{2^2,1^2}F_3^*$, \\
			$ S_{4,3,2^3}F_1 \otimes S_{3,1^3}F_3^*$,  
			$ S_{4^2,3,2^2}F_1 \otimes S_{3,2,1^2}F_3^*$, 
			$ S_{5,3^2,2^2}F_1 \otimes S_{2^3,1}F_3^*$,  \\
			$ S_{4^2,3^2,2}F_1 \otimes S_{2^3,1}F_3^*$, 
			$ S_{3^5}F_1 \otimes S_{4,1^3}F_3^*$, 
			$ S_{5,4,3^2,2}F_1 \otimes S_{3,2^2,1}F_3^*$, \\
			$ S_{4^2,3^3}F_1 \otimes S_{4,2,1^2}F_3^*$,  
			$ S_{4^3,3,2}F_1 \otimes S_{3^2,1^2}F_3^*$, 
			$ S_{5,4^2,2^2}F_1 \otimes S_{2^4}F_3^*$,  \\
			$ S_{4^4,1}F_1 \otimes S_{2^4}F_3^*$,  
			$ S_{5,4^2,3^2}F_1 \otimes S_{4,2^,1}F_3^*$  
			$ S_{5^2,4,3,2}F_1 \otimes S_{3,2^3}F_3^*$, \\ 
			$ S_{5,4^3,2}F_1 \otimes S_{3^2,2,1}F_3^*$,  
			$ S_{6,4,3^3}F_1 \otimes S_{3,2^3}F_3^*$,  
			$ S_{5^2,3^3}F_1 \otimes S_{3^2,2,1}F_3^*$, \\ 
			$ S_{4^4,3}F_1 \otimes S_{4,3,1^2}F_3^*$, 
			$ S_{5^2,4^2,3}F_1 \otimes S_{4,3,2,1}F_3^*$, 
			$ S_{6,5,4,3^2}F_1 \otimes S_{3^2,2^2}F_3^*$, \\
			$ S_{6,4^3,3}F_1 \otimes S_{3^3,1}F_3^*$, 
			$ S_{5^3,3^2}F_1 \otimes S_{4,2^3}F_3^*$, 
			$ S_{5^3,4,2}F_1 \otimes S_{3^2,2^2}F_3^*$, \\
			$ S_{6,4^3,3}F_1 \otimes S_{4,2^3}F_3^*$, 
			$ S_{5^3,4^2}F_1 \otimes S_{5,2^3}F_3^*$, 
			$ S_{6,5^2,4,3}F_1 \otimes S_{4,3,2^2}F_3^*$, \\
			$ S_{6,5,4^3}F_1 \otimes S_{4,3^2,1}F_3^*$, 
			$ S_{7,4^4}F_1 \otimes S_{3^3,2}F_3^*$,  
			$ S_{5^3,4^2}F_1 \otimes S_{4^2,2,1}F_3^*$, \\
			$ S_{6^2,4^2,3}F_1 \otimes S_{3^3,2}F_3^*$,
			$ S_{5^4,3}F_1 \otimes S_{4,3^2,1}F_3^*$,   
			$ S_{6,5^3,4}F_1 \otimes S_{5,3,2^2}F_3^*$, 
			$ S_{7,5^2,4^2}F_1 \otimes S_{4,3^2,2}F_3^*$, \\
			$ S_{6^2,5^2,3}F_1 \otimes S_{4,3^2,2}F_3^*$,
			$ S_{6^2,5,4^2}F_1 \otimes S_{4^2,2^2}F_3^*$,
			$ S_{6^3,4,3}F_1 \otimes S_{3^4}F_3^*$, \\
			$ S_{6,5^3,4}F_1 \otimes S_{4^2,3,1}F_3^*$,  
			$ S_{7,6,4^3}F_1 \otimes S_{3^4}F_3^*$. \\
			The decoration with the largest value of $k$ appearing for the format $(1,5,8,4)$ is $ S_{10^4,9}F_1 \otimes S_{6^4}F_3^*\otimes$.
		\end{itemize}
		
	In total, we have only $2$ Herzog classes for the format $(1,5,6,2)$, $7$ Herzog classes for the format $(1,6,7,2)$, $11$ Herzog classes for the format $(1,5,7,3)$, $53$ Herzog classes for the format $(1,7,8,2)$, and $95$ Herzog classes for the format $(1,5,8,4)$.

		\subsection{Infinite families of Herzog classes of ideals with resolution of format $(1,6,8,3)$}
		\label{sec:1683}
		
		%\subsection{Infinite families with format $(1,6,8,3)$}
		
		The format of resolution $(1,6,8,3)$ is the smallest non-Dynkin format. The Kac-Moody Lie algebra associated to this format is not finitely generated and therefore there are infinitely many  higher structure maps that are possibly nonzero. This means that we expect to see infinitely many Herzog classes of ideals having minimal free resolution of this format. Computational experiments show that there are so many Herzog classes that a complete classification seems to be hard. However we can provide examples of a few nice families. \\
		
		For $n \geq 2$, set $a= \frac{n(3n+1)}{2}$. Define 
		$$ \s_{I_n}= S_{a+2,a,a,a-2n+1, a-2n+1, a-2n+1} F_1 \otimes S_{a-n+1,a-n+1,a-n} F_3^*, $$
		$$ \s_{H_n}= S_{a-n+1, a-n+1, a-n, a-3n+3,a-3n+1,a-3n+1} F_1 \otimes S_{a-2n+1,a-2n+1,a-2n+1} F_3^*, $$
		$$ \s_{J_n}= S_{a-2n+1, a-2n+1, a-2n+1, a-4n+3,a-4n+3,a-4n+2} F_1 \otimes S_{a-3n+3,a-3n+1,a-3n+1} F_3^*. $$
		
		It is not hard to check using the formulas in Examples \ref{minimalminimal} and \ref{maximalminimal} that, if $\rightsquigarrow$ denotes the smallest minimal link, we have $$ I_n \rightsquigarrow H_n \rightsquigarrow J_n \rightsquigarrow I_{n-1}, $$ and, if $\leftsquigarrow$ denotes the largest minimal link, we have $$ I_n \leftsquigarrow H_n \leftsquigarrow J_n \leftsquigarrow I_{n-1}. $$
		
		%In the opposite direction all links are maximal minimal, which means that one has to choose a regular sequence among minimal generators which does not contain any of the elements of a regular sequence defining the minimal minimal link.
		For $n= 1$, we have $\s_{I_1}= S_{4,2,2,1, 1, 1} F_1 \otimes S_{2,2,1} F_3^*$. The smallest minimal link of $\s_{I_1}$ produces the decoration of the almost complete intersection of format $(1,4,6,3)$. By induction this shows that for every $n$, $\s_{I_n}$, $\s_{J_n}$ and $\s_{H_n}$ are decorations and they are clearly of ideals with resolution of format $(1,6,8,3)$. Other infinite families of decorations with the same formats can be obtained with analogous methods. For instance, starting by $\s_{I_1'}= S_{3,3,2,1, 1, 1} F_1 \otimes S_{3,1,1} F_3^*$ and linking up using an infinite sequence of largest minimal links will produce three other infinite families of decorations with a similar behavior.
		
		%\bigskip 

		%For $n \geq 2$, set $a= \frac{n(3n+1)}{2} +1 $. Define 
		%$$ \s_{I'_n}= S_{a,a,a-1,a-2n, a-2n, a-2n} F_1 \otimes S_{a-n+1,a-n-1,a-n-1} F_3^*, $$
		%$$ \s_{H_n}= S_{a-n+1, a-n+1, a-n, a-3n+3,a-3n+1,a-3n+1} F_1 \otimes S_{a-2n+1,a-2n+1,a-2n+1} F_3^*, $$
		%$$ \s_{J_n}= S_{a-2n+1, a-2n+1, a-2n+1, a-4n+3,a-4n+3,a-4n+2} F_1 \otimes S_{a-3n+3,a-3n+1,a-3n+1} F_3^*. $$
		%Then, if $\rightsquigarrow$ denotes a minimal minimal link, we have $ I'_n \rightsquigarrow H'_n \rightsquigarrow J'_n \rightsquigarrow I'_{n-1}. $ In the opposite direction all links are maximal minimal, which means that one has to choose a regular sequence among minimal generators which does not contain any of the elements of a regular sequence defining the minimal minimal link.
		%For $n= 1$, we have $\s_{I_1}= S_{3,3,2,1, 1, 1} F_1 \otimes S_{3,1,1} F_3^*$. A minimal minimal link of $\s_{I_1}$ produces the functor of the hyperplane section of format $(1,5,7,3)$. 

		\section{Classifying the admissible decorations}
		\label{sec:admissible}
		
		In this section we consider the problem of classifying which pairs of partitions (or Schur functors) are indeed decorations of vertices of the licci graph. Let us consider a pair of Schur functors of the form $$ \s= S_{\lambda_1, \ldots, \lambda_{d+3}} F_1 \otimes S_{\mu_1, \ldots, \mu_{t}} F_3^* $$ where $\lambda_i, \mu_j$ are positive integers (listed in non-increasing order) such that
		$\sum_{i\geq 1} \lambda_i = 2k+1$ and $\sum_{j \geq 1} \mu_j = k$.  
		
		We know that $\s$ is a decoration if and only if after applying a finite number of times the operation corresponding to the smallest minimal linkage (see Example \ref{minimalminimal}) we obtain the decoration $\bigwedge^3 F_1 \otimes F_3^*$ corresponding to complete intersections. Since we will now work using smallest minimal links, if not otherwise specified, we set $p= p(\s)= \sum_{i\geq 4} \lambda_i - k = k+1 - \lambda_1 - \lambda_2 - \lambda_3.$ 
		
		We can immediately find several restrictions that prevent many pairs of Schur functors of the above form from being decorations:
		
		\begin{prop}
			\label{restrictions}
			Let $\s = \s_I$ be a decoration of a vertex of the licci graph $\mathrm{Licci}_3$. Then:
			\begin{enumerate}
				\item[(1)] $\lambda_1 + \lambda_2 \leq k+1$.
				\item[(2)] $\mu_1 + \lambda_1 + \lambda_2 > k+1$.
				\item[(3)] $\sum_{i \geq 1} \lambda_i^2 + \sum_{j \geq 1} \mu_j^2 = (k+1)^2.$
				\item[(4)] $\lambda_1+\mu_1 \leq k+1$.
			\end{enumerate}
		\end{prop}
		
		\begin{proof}
			Let $\s'$ be the pair obtained from $\s$ by smallest minimal link. \\
			(1)  If $\lambda_1 + \lambda_2 > k+1$, by definition of $p$, we would get $\lambda_3 + p < 0$. This is a contradiction since $\s'$ cannot have negative indices in any of the partitions. \\
			(2) We can assume that $\s$ is not the decoration of a complete intersection, since otherwise it would satisfy condition (2).
			If $\mu_1 + \lambda_1 + \lambda_2 \leq k+1$, we would have $\mu_1 \leq p+\lambda_3$. Hence, $$ s'= S_{\lambda_1+p, \lambda_2+p, \lambda_3+p, \mu_1, \ldots, \mu_t}F'_1 \otimes S_{\lambda_4, \ldots, \lambda_{d+3}}F_3'^*, $$ and $p(\s')= \sum_{j \geq 1} \mu_j - (k+p)= -p.$ If $\s''$ denotes the smallest minimal link of $\s'$ we get now $\s''=\s.$ It follows that $\s$ cannot be a decoration since it cannot be linked to the decoration of a complete intersection by any sequence of smallest minimal links. \\
			(3) We prove a more general version of this result in arbitrary codimension in Theorem \ref{squareformulas}. This case follows as a specialization of that theorem to the case $c=3$. \\
			%We use the formula for the graded minimal free resolution of $I$ given in (\ref{genericcomplex}). Let $e_1, \ldots, e_{d+3}$ and $g_1, \ldots, g_t$ denote basis for the first and third free module in the resolution.
			%Computing the degree of the entries of the higher structure map associated to $\s$ we find the formula
			%$$ 0 = \sum_{i=1}^{d+3} \lambda_i \deg(e_i) - \sum_{j=1}^{t} \mu_j \deg(g_j) = \sum_{i=1}^{d+3} \lambda_i (k+1-\lambda_i) - \sum_{j=1}^{t} \mu_j (k+1+\mu_j). $$ Easy calculations lead to the formula in (3). \\
            (4)  %Let $\s_J$ be the decoration obtained from $\s$ by taking the smallest minimal link. 
            Applying (1) to $\s'$, we obtain $\lambda_1+p + \mu_1 \leq \kappa(J)+1=k+p+1$. The thesis clearly follows.
		\end{proof}

		\begin{cor}
			\label{restrictions2}
			Let $\s = \s_I$ be a decoration of a vertex of the licci graph in codimension 3.
			The following assertions hold:
			\begin{enumerate}
				\item[(1)] $\mu_2=0$ if and only if $I$ is Gorenstein and if and only if $\lambda_1 = 1$.
				\item[(2)] $\lambda_5=0$ if and only if $I$ is an almost complete intersection.
				\item[(3)] $\lambda_1 = k$ if and only if $I$ is an hyperplane section of a perfect ideal of height 2.
			\end{enumerate}
		\end{cor}
		
		\begin{proof}
			(1) The first equivalence is straightforward. The condition $\mu_2=0$ is equivalent to $\mu_1=k$ which implies $\lambda_1=1$ by Proposition \ref{restrictions}(4).
            Similarly, if $\lambda_1 =1$, the condition $\mu_1=k$  follows by Proposition \ref{restrictions}(2). \\
            %Assuming the third condition, if by way of contradiction $\mu_2 \neq 0$, it follows that $\mu_1 < k$ and $\s$ does not satisfy condition (2) of Proposition \ref{restrictions}. \\
			(2) This is straightforward. \\
			(3) If $\lambda_1=k$, by Theorem \ref{genericcomplexthm} we can find an homogeneous ideal in the Herzog class of $I$ with a generator of degree 1.
			%the formulas in (\ref{genericcomplex}), 
			%we obtain that one generator of $I$ has degree one. Since $I$ is homogeneous, 
            Up to apply a linear change of coordinate we find that $I$ is in the same Herzog class of an hyperplane section. Therefore $I$ has to be an hyperplane section. % Any specialization of an hyperplane section is still an hyperplane section, so $\s$ corresponds to one of the Herzog classes of hyperplane sections. 
            Conversely, looking at the table in Section \ref{sec:tables}, one can easily describe the decorations of each hyperplane section starting from a complete intersection, using Algorithm \ref{algo:licci-examples}.
            % here we need again Theorem \ref{genericcomplexthm} \ec.
		\end{proof}

        The invariant $k=\kappa(I)$ is apparently not strictly related with the number of generators $b=d+3$ of a licci ideal $I$. We clearly have the inequality $b \leq 2k+1$ and, by Corollary \ref{restrictions2}, the equality $b = 2k+1$ holds if and only if $I$ is Gorenstein. We can prove that, in case $I$ is not Gorenstein, there is a more accurate bound for $b$ in function of $k$.

\begin{prop}
\label{boundfornumgens}
Let $\s_I$ be the decoration of a non-Gorenstein licci ideal $I$ of codimension 3, minimally generated by $b$ elements. 
Then $b \leq \kappa(I)+2$.
\end{prop}

\begin{proof}
We can work by induction on $k=\kappa(I)$, observing by inspection that the result is true for small values of $k$ (see Section 4). Consider the tight double link $\s_{I'}$ of $\s_I$ defined as in Example \ref{tightdouble}. Adopting the same notation as in that example, we have $\kappa(I')=k+p+q< k$ and by inductive hypothesis the number of generators $b(I')\leq \kappa(I') + 2 < k+2.$ Looking at the partitions associated to $\s_{I'}$, we also observe that $b-2 \leq b(I') \leq b$ and the second inequality is strict if and only if one of $\lambda_1+p+q$ and $\lambda_2+p+q$ is zero. If $b(I') > b-2$ we are done. If $b(I')=b-2$, then $\lambda_1+p+q= \lambda_2+p+q=0$. If by way of contradiction we suppose $b > k+2$, then $b-2 \leq k+p+q+2 < k+2$, hence $b=k+3$ and $p+q=-1$. This implies $\lambda_1 = 1$, which is impossible by Corollary \ref{restrictions2}, since $I$ is not Gorenstein.
\end{proof}

	%	We show now that the quantity $\lambda_1 + \lambda_2$ is maximized in the case the elements of the Herzog class are directly linked to ideals with strictly smaller total Betti number. We will investigate further properties of decorations satisfying this condition in Section \ref{sec:monomials}.
		
	%	\begin{lem}
	%		\label{koszulrelation}
	%		Let $\s_I$ be a decoration. 
	%		The following conditions are equivalent:
	%		\begin{enumerate}
		%		\item $I$ is directly linked to an ideal with smaller total Betti number.
			%	\item $\lambda_1 + \lambda_2 = k+1$.
			%\end{enumerate}
		%\end{lem}
		
		%\begin{proof}
	%		It is enough to consider the case when  $I$ defines a rigid algebra. \ec
%			If $I$ is directly linked to an ideal with smaller total Betti number, then $I$ must have a Koszul relation of its generators among the minimal sygyzies. Hence, by  Theorem \ref{genericcomplexthm}, \ec looking at the graded Betti numbers of the resolution (\ref{genericcomplex}) of $I$, there exists distinct $i,j$ such that $k+1-\lambda_i + k+1 - \lambda_j = k+1$. It follows that $\lambda_i + \lambda_j = k+1 $, and by item (1) of Proposition \ref{restrictions}, also $\lambda_1 + \lambda_2 = k+1$. Conversely, if $\lambda_1 + \lambda_2 = k+1$, setting $p= \sum_{i\geq 4} \lambda_i -k$, arguing as in the proof of item (1) of Proposition \ref{restrictions}, we obtain $p= -\lambda_3$ and therefore the smallest minimal link of $I$ has smaller total Betti number than $I$.
%		\end{proof}
		
		Next we show that the only decorations involving a pure exterior power of $F_3^*$ in the second partition (equivalently $\mu_1=1$) correspond either to almost complete intersections or to hyperplane sections. 
		
		\begin{prop}
			\label{exteriorF3}
			Let $\s = \s_I$ be a decoration of a vertex of the licci graph $\mathrm{Licci}_3$. Suppose that $\mu_1=1$. %(equivalently $\mu_j = 1$ for every $j=1, \ldots, t$). 
			Then $I$ is either an almost complete intersection or an hyperplane section. 
		\end{prop}
		
		\begin{proof}
			Using the restrictions obtained in Proposition \ref{restrictions} and looking at Section \ref{sec:dynkin}, we can easily check that the thesis is true for small values of $k$. Thus we can work by induction on $k$. By condition (1) and (2) of Proposition \ref{restrictions} we must have $\lambda_1 + \lambda_2 = k+1$ where $k= \kappa(I)$. By Corollary \ref{restrictions2}, we can also assume that $I$ is not Gorenstein and $\mu_2=1$.

Consider the minimal tight double link $\s_{I'}$ of $\s$ defined as in Example \ref{tightdouble}. 
Since $\lambda_1 + \lambda_2 = k+1$, then $p= k+1-\lambda_1-\lambda_2-\lambda_3 = -\lambda_3$ and 
 $$ \s_{I'}=  S_{\lambda_{1}+p+q, \lambda_{2}+p+q, \lambda_3, \lambda_4, \ldots, \lambda_{d+3}} G_1 \otimes S_{\mu_2, \ldots, \mu_t} G_3^*, $$ where $q = \lambda_3 -\mu_1 = \lambda_3-1$.
 Observe that $\kappa(I')= k+p+q= k-1 < k$. 
 By inductive hypothesis $I'$ is either an almost complete intersection or an hyperplane section. Observe that $\lambda_1+p+q= k - \lambda_2$ and $\lambda_2+p+q= k - \lambda_1$. By Corollary \ref{restrictions2}, one of the two is zero if and only if $I$ is an hyperplane section. If we assume that $I$ is not an hyperplane section, we must then have $\lambda_{1}+p+q, \lambda_{2}+p+q \neq 0$. Hence, if $I'$ is an almost complete intersection, we have $\lambda_5=0$, which implies that also $I$ is an almost complete intersection.
 
 If $I'$ is an hyperplane section, then $k-\lambda_1= \lambda_{2}+p+q= 1$, forcing $\lambda_1 = k-1$ and $\lambda_2=2$ since $p+q=-1$. But $k-\lambda_2= \lambda_{1}+p+q$ is equal either to $\kappa(I')=k-1$ or to $1$. The first case is impossible because $\lambda_2=2$, while the second case forces $\lambda_1=\lambda_2=2$ and $k=3$, which again forces $I$ to be an almost complete intersection by the tables in Section \ref{sec:dynkin}.  
		\end{proof}

		\begin{cor}
			\label{lambda4}
			Let $\s = \s_I$ be a decoration. Suppose that $\lambda_4=1$. Then $I$ is directly linked to either an almost complete intersection or an hyperplane section. 
		\end{cor}
		
		\begin{proof}
			Apply the smallest minimal link and use Proposition \ref{exteriorF3}.
		\end{proof}

			Thanks to Corollary \ref{lambda4}, the tables in Section \ref{sec:tables} allow to classify all the decorations such that $\lambda_4=1$.

        \section{Pairs of partitions and multiplicative structure in Tor algebra}
		\label{sec:toralgebra}

        We already recalled in Section \ref{sec:background-herzog} that the free resolution $\FF$ of a perfect ideal of codimension 3 admits an associative multiplicative structure consisting of %\bf cite \ec, which is defined by a lifting cycles procedure using the comparison maps between $\FF$ and the Koszul complex on the generators of $I$. The three maps defining this structure 
three linear maps $\bigwedge^2 F_1 \to F_2$, $F_1 \otimes F_2 \to F_3$, $\bigwedge^3 F_1 \to F_3$. In a local setting, the possible multiplicative structures can be classified modulo the maximal ideal. This correspond to study the multiplication in the Tor algebra Tor$(R/I,R/\m)$. For details about this classification see \cite{Weyman89}, \cite{AKM88}, \cite{CVWlinkage}, \cite{CV} and related papers. 

The configuration of unit elements in higher structure maps is an invariant of the Herzog classes. In particular so is the structure of multiplication in the Tor algebra. We show how, in the licci case, this structure can be detected simply by looking at the pair of partitions associated to the Herzog class of the ideal.

Let $I$ be a perfect ideal of codimension 3 in a polynomial ring $R$ and denote a given basis of the free modules $F_1, F_2, F_3$ in the minimal free resolution of $I$ by 
$\lbrace e_1, \ldots, e_{d+3}\rbrace$, $\lbrace f_1, \ldots, f_{d+t+2} \rbrace$, $\lbrace g_1, \ldots, g_{t}\rbrace$. We have:

\begin{thm}
\label{multiprelation}
Let $I$ be a licci ideal of codimension 3, with $\s_I = S_{\boldsymbol \lambda} F_1 \otimes S_{\boldsymbol \mu} F_3^*$. 
Then there exists a choice of bases for the free modules $F_i$ in a minimal free resolution of $R/I$ such that:
\begin{enumerate}
\item[(1)] The product $e_i^.e_j$ is nonzero modulo $\m$ if and only if $\lambda_i + \lambda_j = k+1$.
\item[(2)] We have $e_i^.f = g_j$ modulo $\m$ for some $f \in F_2$ if and only if $\lambda_i + \mu_j = k+1$.
\end{enumerate}
\end{thm}

\begin{proof}
Since the Tor algebra structure is an invariant of the Herzog class, it suffices to prove the claim for a single ideal in the same class. To this end, we assume $I$ is a homogeneous ideal with minimal graded free resolution of the form described in
Theorem \ref{genericcomplexthm}, and the basis of $F_1$ maps to a SGS of $I$ (see Definition \ref{def:SGS}).
The degree of the coefficient of $e_i^.e_j$ with respect to a basis element $f \in F_2$ is given by $$\deg(e_i)+\deg(e_j)-\deg(f)= k+1-\lambda_i+k+1-\lambda_j -k-1= k+1 -\lambda_i-\lambda_j.$$ Similarly, the degree of the coefficient of $e_i^.f$ with respect to $g_j$ is $$\deg(e_i)+\deg(f)-\deg(g_j)= k+1-\lambda_i+k+1-\mu_j -k-1= k+1 -\lambda_i-\mu_j.$$ %By Proposition \ref{restrictions}, both these degrees are non-negative and 
These degrees are equal to zero if and only if the corresponding equalities $\lambda_i + \lambda_j = k+1$ and $\lambda_i + \mu_j = k+1$ hold. It follows that if one of those equalities does not hold, the corresponding product is zero modulo $\m$. 

Conversely, suppose $\lambda_i + \lambda_j = k+1$ and without loss of generality say that $\lambda_i \geq \lambda_j$. Define the minimal link of $\s_I$ with respect to the partition $\boldsymbol \lambda'= \lambda_i, \lambda_j, \lambda_l$ for some $l \neq i,j$. The corresponding $p$ is $k+1-\lambda_i-\lambda_j-\lambda_l= -\lambda_l$. Thus $\lambda_l+p =0$ and, by the formula in Remark \ref{re2}, the ideal $J$ linked to $I$ with respect to a regular sequence associated to $\lambda_i, \lambda_j, \lambda_l$ has less than $t+3$ minimal generators. Therefore at least one of the products $ e_i^.e_j, e_i^.e_l, e_j^.e_l $ is nonzero modulo $\m$. 
If $\lambda_l < \lambda_j$, then by Proposition \ref{restrictions}(1), $\lambda_i + \lambda_l, \lambda_j + \lambda_l < k+1$ and the products $ e_i^.e_l, e_j^.e_l $ vanish modulo $\m$. It follows that $ e_i^.e_j $ is nonzero modulo $\m$. If $\lambda_l \geq \lambda_i$, then $\lambda_i + \lambda_l= \lambda_j + \lambda_l = k+1$. In this case $J$ has only $t$ generators and all three products $ e_i^.e_j, e_i^.e_l, e_j^.e_l $ do not vanish modulo $\m$. Similarly, if $ \lambda_j \leq \lambda_l < \lambda_i $, then $J$ has $t+1$ generators and the products $ e_i^.e_j, e_i^.e_l$ do not vanish modulo $\m$.
In any case this proves (1).

Finally, suppose $\lambda_i + \mu_j = k+1$. Consider the minimally linked decoration with respect to the choice $\boldsymbol \lambda'= \lambda_i, \lambda_r, \lambda_s$ for some $r,s \neq i$. Let $J$ be the associated linked ideal, and observe that there are two minimal generators of $J$ corresponding to the partition elements $\lambda_i + p$ and $\mu_j$ in $\s_J$. The first of these two generators is in $I \cap J$, while the second one is not in $I$. By assumption we have $\lambda_i+p+\mu_j=k+1+p= \kappa(J)+1$. By the statement (1) of this proposition, applied to $J$, we obtain that the product corresponding to these two generators is nonzero modulo $\m$. It follows that $e_i^.f = g_j$ modulo $\m$ for some $f$ as a consequence of \cite[Proposition 3.5 (3.7)]{GNW1}. This proves (2).
%By Proposition \ref{restrictions}(1), this forces $\lambda_1 + \lambda_2 = k+1$, since by definition of $\boldsymbol \lambda$, $\lambda_1+ \lambda_2 \geq \lambda_i+\lambda_j$.
%By applying the smallest minimal link as in Example \ref{minimalminimal}, we obtain $\lambda_3+p =0$. Therefore the Herzog class of $I$ is directly linked to an Herzog class whose ideals have strictly smaller total Betti number.%It is enough to consider the case when $I$ defines a rigid algebra.
%If $I$ is directly linked to an ideal with smaller total Betti number, then $I$ must have a Koszul relation of its generators among the minimal sygyzies. Hence, by Theorem \ref{genericcomplexthm}, looking at the graded Betti numbers of the resolution (\ref{genericcomplex}) of $I$, there exists distinct $i,j$ such that $k+1-\lambda_i + k+1 - \lambda_j = k+1$. It follows that $\lambda_i + \lambda_j = k+1 $, and by item (1) of Proposition \ref{restrictions}, also $\lambda_1 + \lambda_2 = k+1$. Conversely, if $\lambda_1 + \lambda_2 = k+1$, setting $p= \sum_{i\geq 4} \lambda_i -k$, arguing as in the proof of item (1) of Proposition \ref{restrictions}, we obtain $p= -\lambda_3$ and therefore the smallest minimal link of $I$ has smaller total Betti number than $I$.
\end{proof}

In the remaining of this section we investigate the cases where the map $\bigwedge^2 F_1 \to F_2$ does not vanish modulo $\m$ and the case where there is a partial perfect pairing on the resolution induced by the map $F_1 \otimes F_2 \to F_3$.

\subsection{Licci ideals with Koszul relations among minimal syzygies}
		\label{sec:monomials}

A perfect ideal of codimension three has Koszul relations among minimal syzygies if and only if the multiplication map $\bigwedge^2 F_1 \to F_2$ does not vanish modulo the maximal (or homogeneous maximal) ideal. This condition is equivalent to the property of being linked to some ideal having strictly smaller total Betti number.
In Theorem \ref{multiprelation}, we characterize the decorations corresponding to licci ideals with Koszul relations among minimal syzygies as the ones such that $\lambda_i + \lambda_j = k+1$ for some $i,j$. In particular, by Proposition \ref{restrictions}(1) this implies necessarily that $\lambda_1 + \lambda_2 = k+1$.
%hence such that the multiplication map $\bigwedge^2 F_1 \to F_2$ does not vanish modulo the maximal (or homogeneous maximal) ideal of $R$.  
For this reason, in this section we study the family of decorations, characterized by the condition $\lambda_1 + \lambda_2 = k+1$.
		
	%	In Lemma \ref{koszulrelation}, we characterize the decoration corresponding to ideals with Koszul relations among minimal syzygies, hence such that the multiplication map $\bigwedge^2 F_1 \to F_2$ does not vanish modulo the maximal (or homogeneous maximal) ideal of $R$. This condition is equivalent to the property of being linked to an Herzog class of ideals having strictly smaller total Betti number. In this section we study this family of decoration, which is characterized by the condition $\lambda_1 + \lambda_2 = k+1$.
		
		In \cite{hu-ul4} it is shown that any licci $\m$-primary monomial ideal is directly linked by the monomial regular sequence of smallest degree to a monomial ideal with smaller total Betti number. Therefore the Herzog classes of licci monomial ideals of codimension 3 in $K[x,y,z]$ are part of this family.
		
		We prove that for any fixed type $t$ or number of generators $b= d+3$, the pairs of partitions corresponding to decorations such that $\lambda_1 + \lambda_2 = k+1$ are bounded, hence they are finitely many. \\

		For any decoration $\s$ such that $\lambda_1 + \lambda_2 = k+1$ we can define a special linked decoration
		$$ p_0(\s)=  S_{\lambda_{1}, \lambda_{2}, \mu_1, \ldots, \mu_t}F_1 \otimes S_{\lambda_3, \ldots, \lambda_{d+3}}F_3^*. $$ The link from $\s$ to $p_0(\s)$ is defined choosing $\lambda_1', \lambda_2', \lambda_3' = \lambda_1, \lambda_2, 0$ and observing that the corresponding value of $p$ is zero. It is possible in some cases to have $p_0(\s )= \s$.
		
		We define now for $t \geq 1$, a family of decorations whose partitions are expressed in terms of consecutive powers of $2$: $$ \mathcal{H}_t = S_{2^{t-1},2^{t-1},2^{t-2},2^{t-2}, \ldots, 2,2,1,1,1}F_1 \otimes S_{2^{t-1},2^{t-2}, \ldots, 2,1}F_3^*. $$ Observe that $\mathcal{H}_1$ is the decoration of the complete intersections and next we have $\mathcal{H}_2= S_{2,2,1,1,1}F_1 \otimes S_{2,1}F_3^*$, $\mathcal{H}_3 = S_{4,4,2,2,1,1,1}F_1 \otimes S_{4, 2,1}F_3^*$, and so on.
		Recall that by $\gen(\s)$, we mean the generic link defined in Example \ref{genericlink}.
		We prove a couple of lemmas about this kind of decorations.
		
		\begin{lem}
			\label{hunekesequence}
			For $t \geq 1$, we have $ \mathcal{H}_{t+1} = p_0(\gen(\mathcal{H}_{t})). $ In particular each $\mathcal{H}_t$ is a decoration such that $\lambda_1 + \lambda_2 = k+1$.
		\end{lem}
		
		\begin{proof}
			Observe that $\kappa(\mathcal{H}_t)= \sum_{j=0}^{t-1} 2^{j} = 2^t-1.$ Hence, if we prove that $\mathcal{H}_t$ is a decoration, the condition $\lambda_1 + \lambda_2 = k+1$ will follow immediately.
			By induction on $t\geq 1$, it is equivalent to assume that $\mathcal{H}_{t}$ is a decoration and prove that $\mathcal{H}_{t+1}$ is also a decoration. 
			%If $\s$ is a monomial decoration, so it is also $\gen(\s)$ since a generic link of a monomial ideal can be always obtained linking by a monomial regular sequence. Also $p_0(\s)$ is a monomial decoration as a consequence of Lemma \ref{mon1}. It follows that $p_0(\gen(\s))$ is a monomial decoration and we are only required to prove the equality $ \mathcal{H}_t = p_0(\gen(\mathcal{H}_{t-1})). $ 
			Applying the definitions we get 
			$$ \gen( \mathcal{H}_{t}) = S_{2^t, 2^t, 2^t, 2^{t-1},2^{t-2}, \ldots, 2,1}F_1 \otimes S_{2^{t-1},2^{t-1},2^{t-2},2^{t-2}, \ldots, 2,2,1,1,1}F_3^*$$ and 
			$$ p_0(\gen( \mathcal{H}_{t})) = S_{2^t, 2^t, 2^{t-1},2^{t-1},2^{t-2},2^{t-2}, \ldots, 2,2,1,1,1}F_1 \otimes S_{2^t, 2^{t-1},2^{t-2}, \ldots, 2,1}F_3^*= \mathcal{H}_{t+1}. $$
			Therefore $\mathcal{H}_{t+1}$ is a decoration.
			%The condition $\lambda_1 + \lambda_2 = k+1$ is clear since $2^{t-1}+2^{t-1} = 2^t= \sum_{i=0}^{t-1}2^i+1$.
		\end{proof} 
		
		Given two arbitrary decorations $\s_1= S_{\lambda_1, \ldots, \lambda_{d_1+3}} F_1 \otimes S_{\mu_1, \ldots, \mu_{t_1}} F_3^*$ and $\s_2= S_{\iota_1, \ldots, \iota_{d_2+3}} F_1 \otimes S_{\upsilon_1, \ldots, \upsilon_{t_2}} F_3^*$, we say that $\s_1 \leq \s_2$ if for every $i$, $\lambda_i \leq \iota_i$ and $\mu_i \leq \upsilon_i$.
		
		\begin{lem}
			\label{mon2}
			The following assertions hold:
			\begin{enumerate}
				\item If $\s_1 \leq \s_2$ then $\gen(\s_1) \leq \gen(\s_2)$.
				\item If $\s_1 \leq \s_2$ are such that $p_0(\s_1)$ and $p_0(\s_2)$ are defined, then $p_0(\s_1) \leq p_0(\s_2)$.
				\item If $\s'$ is any link of $\s$, then $\s' \leq \gen(\s)$.
				\item If $\s$ is a decoration such that $\lambda_1 + \lambda_2 = k+1 $, then $p_0(p_0(\s))= \s$.
			\end{enumerate}
		\end{lem}
		
		\begin{proof}
			Items 1 and 2 follow immediately by the definitions of generic link and of $p_0$. 
			
			For item 3, set $$\s'= S_{\lambda'_1+p, \lambda'_2+p, \lambda'_3+p, \mu_1, \ldots, \mu_t} G_1 \otimes S_{\lambda_1'', \ldots, \lambda''_{d+3}} G_3^*$$ for $p=k+1-\lambda_1'-\lambda_2'-\lambda_3'$, where $k=\kappa(\s)=\sum_{i=1}^t \mu_i$. By definition of $\lambda_i''$ (as in Remark \ref{re2}), we clearly have $\lambda_i'' \leq \lambda_i$. Looking at the definition of $\gen(\s)$ and at the partition on the side of $G_1$, after reordering, it is enough to observe that, essentially by definition, $ \lambda_1' + p \leq k+1$ and $\mu_1 < k+1$. %If $\lambda'_1=0$, the inequality is clearly true. Thus suppose $\lambda'_1= \lambda_j$ for some $j$. It follows that $p \leq (\sum_{i=1}^{d+3} \lambda_i) - \lambda_j - k = k+1 - \lambda_j$ and the desired inequality holds also in this case.
			
			For item 4, by definition of $p_0(\s)$, we only need to prove that $\mu_1 \leq \lambda_2$. But if we had $\mu_1 > \lambda_2$, we would have $\lambda_1 + \mu_1 > \lambda_1 +\lambda_2 = k+1 = \kappa(p_0(\s)) + 1$. This would contradict item 1 of Proposition \ref{restrictions} applied to $p_0(\s)$.
		\end{proof}
		
		%$p= \sum_{i=1}^{d+3} \lambda_i''$

		\begin{thm}
			\label{dominatedmonomials}
			Let $\s= S_{\lambda_1, \ldots, \lambda_{d+3}} F_1 \otimes S_{\mu_1, \ldots, \mu_{t}} F_3^*$ be a decoration such that $\lambda_1 + \lambda_2 = k+1$. Then $\s \leq \mathcal{H}_t$.
		\end{thm}
		
		\begin{proof}
			We work by induction on $t$. If $t=1$, then $\s$ corresponds to the Herzog class of a Gorenstein ideal and by Corollary \ref{restrictions2}, $\lambda_1 + \lambda_2 = 1+1=2$. %By Lemma \ref{mon1}, the only possible monomial decoration with type 1 has to be the complete intersection, that is $ \mathcal{H}_{1}. $ 
			This forces $k=1$ and therefore $ \s= \mathcal{H}_{1} $ is the decoration of the complete intersections.
			Hence, we can assume $t > 1$ and suppose the result true for each decoration of ideals with type smaller than $t$. By hypothesis on $\s$, the decoration $p_0(\s)$ is well-defined. Let $\s'$ be the smallest minimal link of $p_0(\s)$. 
			%Since $\s$ is a monomial decoration, also $p_0(\s)$ is a monomial decoration and 
			It is easy to observe that the type of $\s'$ is strictly smaller than $t$. By inductive hypothesis $\s' \leq \mathcal{H}_{t-1}$. Since $p_0(\s)$ is a link of $\s'$,
			we can use items 3 and 1 of Lemma \ref{mon2} to say that $p_0(\s) \leq \gen(\s') \leq \gen(\mathcal{H}_{t-1})$. Then by items 4 and 2 of Lemma \ref{mon2} and by Lemma \ref{hunekesequence}, we obtain $\s = p_0(p_0(\s)) \leq p_0( \gen(\mathcal{H}_{t-1})) = \mathcal{H}_t$.
		\end{proof}

		\begin{cor}
			\label{bounds}
			%Let $\s_I$ be a decoration of a
            Let $I$ be a licci ideal $I$ of codimension 3 with $b=d+3$ generators, type $t$, and having Koszul relations among minimal syzygies. 
			Then %$t \leq \binom{b}{2} - b +1 $ and 44
			$b \leq 2t+1 $ and $t \leq 2b-5$.
		\end{cor}
		
		\begin{proof}
			The first inequality follows directly from Theorem \ref{dominatedmonomials}. For the second, observe that $p_0(\s_I)$ is the decoration of some ideal with $t+2$ generators and type $b-2$, and also having Koszul relations among minimal syzygies. Applying the first inequality to $p_0(\s_I)$ we get $t+2 \leq 2(b-2)+1$, yielding $t \leq 2b-5$.
		\end{proof}

		\begin{example}
			Using linkage theory, it is possible to construct, for every $t \geq 1$, a licci monomial ideals of codimension 3 in $K[x,y,z]$ whose Herzog class is decorated by $\mathcal{H}_t. $ A possible family of examples is given by 
			$$ J_t = (x^t, y^t, z^t)+ \sum_{i=1}^{t-1} z^{t-i} (x^i, y^i). $$
			Here we have $\s_{J_t}= \mathcal{H}_t$.
		\end{example}

        \subsection{Ideals with Tor algebra class $G$}
            \label{sec:classG}
        In this section we study the decorations corresponding to licci ideals with Tor algebra structure of class $G$. We recall the definition.
%Looking at the classification of the multiplicative structure of the Tor algebra of $R/I$, we say that 
An ideal $I$ with free resolution $\FF$ of length 3
is of class $G(r)$ if, up to change basis in $F_1, F_2, F_3$, one has $e_i ^. f_i = g_1$ for $i=1, \ldots, r$ and all the other products $e_i^.e_j$ and $e_i^.f_j$ vanish modulo the maximal ideal of $R$. It is well-know that Gorenstein ideals are the only ideals of class $G(r)$ where $r$ is the number of generators. 

We use the theory of pair of partitions to answer positively to a conjecture posed in \cite[7.4]{CVWlinkage}.
By Theorem \ref{multiprelation}, combined with the inequality from Proposition \ref{restrictions}(4), a licci ideal $I$ is of class $G(r)$ if in the corresponding decoration $\s_I$, we have $\lambda_i+\mu_1=k+1$ for $i=1, \ldots, r$, $\lambda_i+\mu_1<k+1$ for $i> r$, and $\lambda_1+\lambda_2 < k+1$. We have:

\begin{lem}
\label{Grinduction}
Let $\s_I$ be the decoration of a licci ideal $I$ and let $\s_{I'}$ be the decoration of the tight double link of $\s_I$ defined in Example \ref{tightdouble}. Let $\lambda_1, \ldots, \lambda_b$ and $\mu_1, \ldots, \mu_t$, and $l_1, \ldots, l_{b'}$ and $m_1, \ldots, m_{t'}$ be the pairs of partitions associated respectively to $I$ and to $I'$. \\
Suppose that for some $r \geq 2$, $\lambda_i + \mu_1 = \kappa(I)+1$ for every $i=1, \ldots, r$. Then $b'=b-2$ and $l_i + m_1 = \kappa(I')+1$ for every $i=1, \ldots, r-2$. 
%Moreover, if $I$ is of Tor algebra class $G(r)$, then $t'=t$ and if also $\lambda_3 + \lambda_4 < \kappa(I'')+1$, then $I''$ is of class $G(r-2)$. 
\end{lem}

\begin{proof}
Looking at Example \ref{tightdouble}, we first observe that $\lambda_1 + p+q= \lambda_2 + p+q=0$ since $\lambda_1+\mu_1=\lambda_2+\mu_1= k+1$. Hence, $b'=b-2$ (indeed $\mu_1+q = \lambda_3 > 0$).
It follows that $l_1, \ldots, l_{b'}= \lambda_3,\ldots, \lambda_b$. For $i=3, \ldots, r$ we have 
$\lambda_i +\lambda_3 + p = \kappa(I)+1-\mu_1 +\lambda_3 + p= \kappa(I)+p+q+1= \kappa(I')+1$. 
\end{proof}

We can use this result to study the ideals of class $G(r)$, comparing $r$ with the minimal number of generators.

 \begin{lem}
\label{forcetobelicci}
Let $I$ be a perfect ideal of codimension 3 minimally generated by $b$ elements and of Tor algebra class $G(r)$ with $r \geq b-4$. Then $I$ is licci.
\end{lem}

\begin{proof}
This is a special case of \cite[Corollary 5.15(2)]{xianglong}.
\end{proof}

\begin{thm}
\label{conjectureLars1}
Let $I$ be a perfect ideal of codimension 3 minimally generated by $b$ elements. The following conditions are equivalent:
\begin{enumerate}
\item $I$ is of Tor algebra class $G(r)$ with $r \geq b-2$.
\item $r=b$ and $I$ is Gorenstein.
\end{enumerate}
\end{thm}

\begin{proof}
By Lemma \ref{forcetobelicci}, we can restrict to consider only licci ideals.
We can work by induction on $k=\kappa(I)$ and, looking at the lists of decorations in Section 4 and applying Theorem \ref{multiprelation}, %we know that the result is true for small values of $k$. Additionally, 
we observe that for small values of $k$, excluding the Gorenstein cases, there are no decorations such that $\lambda_i + \mu_1 = \kappa(I)+1$ for every $i=1, \ldots, r$, if $r=b-2, b-1$. % or $r=b-4$ and $t=2$. 
%Also we can exclude all the decorations corresponding to Dynkin formats.
Suppose that $I$ is not Gorenstein of Tor algebra class $G(r)$ with %either 
$r \geq b-2$. %or $r=b-4$ and type $t=2$. 
Let $\s_I$ be the decoration of the Herzog class of $I$ and consider a tight doubly linked decoration $\s_{I'}$ as in Lemma \ref{Grinduction}. Keeping the same notation as in Lemma \ref{Grinduction}, it follows that $I'$ is an ideal with $\beta= b-2$ generators and its decoration has the property that $l_i + m_1 = \kappa(I'')+1$ for every $i=1, \ldots, u$ with $u \geq r-2$. Moreover, $\kappa(I')< \kappa(I)$. %If $r \geq b-2$, 
Hence, either 
$I'$ is Gorenstein or this contradicts the inductive hypothesis. If $I'$ is Gorenstein, then $I$ is directly linked to an almost complete intersection. It is sufficient to use the tables in Section 4 to exclude this case.
\end{proof}

\begin{cor}
\label{conjectureLars2}
Let $I$ be a perfect ideal of codimension 3 and type 2 minimally generated by $b$ elements. Then $I$ cannot be of class $G(b-4)$.
\end{cor}

\begin{proof}
Again by Lemma \ref{forcetobelicci}, we can restrict to consider licci ideals. We then proceed exactly as in the proof of Theorem \ref{conjectureLars1} assuming $b \geq 7$, otherwise $I$ is of Dynkin format. We look at the tight double link $I'$ as previously and we
recall that $I'$ has $\beta=b-2$ generators and the type of $I'$ cannot be larger than the type of $I$.  Again $\s_{I'}$ has the property that $l_i + m_1 = \kappa(I')+1$ for every $i=1, \ldots, u$ with $u \geq r-2$.
Hence, either $u=r-2=\beta -4$ and $I'$ has type 2, which is impossible by inductive hypothesis, or $u > r-2$.
If $u \geq r= \beta -2$ we find a contradiction using Theorem \ref{conjectureLars1}. If $u=r-1= \beta -3$, we combine Lemma \ref{Sknuovo} and Theorem \ref{partialperfectpairing} to show that in fact $I$ is of class $G(b-3)$ which is again a contradiction.
\end{proof}

The perfect ideals of codimension 3 with $b$ generators and Tor algebra class $G(b-3)$ will be described in the next section. With similar methods we expect that one can find decorations corresponding to ideals with $b$ generators, Tor algebra class $G(b-i)$ for every $4 \leq i\leq b-2$, and any possible type. 
		
		\section{The family of ideals of type 2 closest to the Gorenstein ones}
		\label{sec:closetoGor}
		
		In Corollary \ref{restrictions2}, it is shown that decorations of Gorenstein ideals of codimension 3 are characterized by the condition $\boldsymbol \mu = (k)$. In terms of partitions one can ask which are the decorations closest to the Gorenstein ones, in the sense that $\boldsymbol \mu = (k-1,1)$. In this section, we investigate this family of decorations.
		
		For $k \geq 2$, consider the pairs of Schur functors $\s_{I_k}= S_{2^{k-1},1,1,1} F_1 \otimes S_{k-1,1} F_3^*$. We show that these Schur functors define two families of decorations (one with $k$ odd and one with $k$ even). The ideals in the Herzog classes defined by these families are the closest to the Gorenstein ones.
		
		\begin{prop}
			\label{family,k-1,1}
			For every $k \geq 2,$ the pair of Schur functors $\s_{I_k}= S_{2^{k-1},1,1,1} F_1 \otimes S_{k-1,1} F_3^*$ is a decoration.
		\end{prop}
		
		\begin{proof}
			If $k=2,3$ we obtain the almost complete intersection of type 2 and the Anne Brown's class of format $(1,5,6,2)$ (see the tables in Section \ref{sec:tables}). Thus we can assume $k \geq 4$.
			We start from $\s_{I_k}$ and apply two consecutive smallest minimal links (in this case this is equivalent to apply the minimal tight double link defined in Example \ref{tightdouble}) .
            %we apply two consecutive smallest minimal links. 
            First we have $p= k+1-6 = k-5$ and the first smallest minimal link is 
			$$ \s_{J_k} = S_{k-1,k-3,k-3,k-3,1} G_1 \otimes S_{2^{k-4},1,1,1} G_3^*. $$
			Then we have $p'= k-3+1 - (2(k-4)+3) = -k+3$. Thus the second smallest minimal link of $\s_{I_k}$ is 
			$$  S_{2, 2^{k-4},1,1,1} F_1' \otimes S_{k-3,1} F_3'^* = \s_{I_{k-2}}. $$ By induction, this shows that $\s_{I_k}$ is a decoration for every $k \geq 1$.
		\end{proof}
		
		%Notice that the two families of ideals $J_k$ (for $k$ even or odd) are linked to the families of $\s_{I_k}$.
		
		We show now that if a decoration has $\boldsymbol \mu = (k-1,1)$, then it must be one of the decorations $\s_{I_k}$.
		
		\begin{prop}
			\label{functor,k-1,1}
			Let $\s= \s_I$ be a decoration, let $k= \kappa(\s)$ and suppose $\mu_1= k-1$. Then $\s = \s_{I_k}$.
		\end{prop}
		
		\begin{proof}
			The assumption $\mu_1=k-1$ implies $\mu_2=1$ and $\mu_3=0$. %By inspection (see Section \ref{sec:dynkin}) we know that the result is true for $k \leq 4$, thus we can assume $k \geq 5$. %and work by induction on $k$. 
            By Corollary \ref{restrictions2}, $I$ is not Gorenstein and $\lambda_1 \geq 2$. But, by Proposition \ref{restrictions}(4), we must have $\lambda_1+\mu_1 \leq k+1$ and therefore $\lambda_1 = 2$.
             %  Applying a first smallest minimal link to $\s$ (with $p= \sum_{i \geq 4} \lambda_i - k$), we get $$\s' = S_{\lambda_1+p,\lambda_2+p, \lambda_3+p, k-1,1} G_1 \otimes S_{\lambda_4, \ldots, \lambda_{d+3}} G_3^*.$$ To reorder in non-increasing order the indices of the partition in $G_1$, we observe that $k-1 > \lambda_3+p$ by condition (2) of Proposition \ref{restrictions} applied to $\s$. Moreover, we can assume $\lambda_3 + p >0,$ otherwise $\s'$ would correspond to an almost complete intersection and, in this case, looking at the tables in Section \ref{sec:tables} describing the links of the almost complete intersections, we get $k=3$ and $\s = \s_{I_3}$. Hence, we have $p(\s')= \lambda_3 +p+1- \sum_{i \geq 4} \lambda_i = \lambda_3 + 1 -k$ and the second smallest minimal link of $\s$ is
	%		$$ \s'' = S_{\lambda_1+p+\lambda_3 + 1 -k,\lambda_2+p+\lambda_3 + 1 -k, \lambda_3, \lambda_4, \ldots, \lambda_{d+3}} F_1' \otimes S_{\lambda_3+p,1} F_3'^*. $$ Since $\s''$ is a decoration, we must have $\lambda_2+p+\lambda_3 + 1 -k \geq 0$. %$\lambda_3 + p < k-1$, we can apply the inductive hypothesis on $\s''$ to get that $\lambda_2+p+\lambda_3 + 1 -k $
	%		Using the definition of $p$ and the fact that $\sum_{i \geq 1} \lambda_i = 2k+1$ we obtain $- \lambda_1 + 2 \geq 0$. By condition (1) of Corollary \ref{restrictions2} we must have $\lambda_1 > 1$, hence $\lambda_1= 2$. Now we know that 
            In particular, all $\lambda_i $ are equal to $2$ or $1$. Let  $a$ and $b$ be the number of $\lambda_i$'s equal respectively to $2$ or to $1$.
			We know that $2a+b= 2k+1$, and by Proposition \ref{restrictions}(3) we get also the relation $4a+b= (k+1)^2 - (k-1)^2-1 = 4k-1$. An easy calculation yields $a = k-1$, $b=3$.
		\end{proof}
		
		\begin{remark}
			\label{linkingup,k-1,1}
			%For the purpose of constructing ideals associated to the functors $\s_{I_k}$, 
			We provide also the formulas describing the links from $\s_{I_k}$ to $\s_{I_{k+2}}$. %The representation of $\sigma \in S_{2^{k-1},1,1,1} F_1$ giving the minimal minimal link to $\s_{J_{k}}$ is $$ s_1^2s_2^2s_3^2\gamma_1^2 \cdots \gamma^2_{d} \gamma_{d+1} \gamma_{d+2} \gamma_{d+3} $$ (where $d= k-1$). To link up to the functor $\s_{J_{k+2}}$ we need to use the element 
			%$$ s_1^2\gamma^2\gamma^2\gamma_1^2 \cdots \gamma^2_{d} \gamma_{d+1} \gamma_{d+2} \gamma_{d+3}.  $$ This link is described by the choices $\lambda_1' = \lambda_1$, $\lambda_2', \lambda_3' = 0$, $\lambda_i'' = \lambda_{i+1}$ for $i = 1,\ldots, d+2$, and $p = \sum_{i \geq 2} \lambda_i - k= (2k+1-2)-k= k-1$. The formula of Theorem \ref{linkageformula} gives now 
			We first link choosing $\lambda_1' = 2$, $\lambda_2'= \lambda_3' = 0$. The link with respect to this choice is
			$$  S_{k+1, k-1, k-1, k-1,1} G_1 \otimes S_{2^{k-2}, 1,1,1} G_3^* = \s_{J_{k+2}}. $$
			To successive link from $\s_{J_{k+2}}$ to $\s_{I_{k+2}}$ is now a minimal link obtained from the choice $\lambda_1', \lambda_2', \lambda_3' = (k-1, k-1, k-1) $.
			%$$ \s'= S_{\lambda_1+p, p, p, k-1,1} G_1 \otimes S_{\lambda_2, \ldots, \lambda_{d+3}} G_3^* = S_{k+1, k-1, k-1, k-1,1} G_1 \otimes S_{2^{k-2}, 1,1,1} G_3^* = \s_{J_{k+2}}. $$
			%Working now with $\s'$, the representation of $\sigma \in S_{k+1, k-1, k-1, k-1,1} G_1$ giving back the minimal minimal link to $\s_{I_k}$ is
			%$$ s_1^{k+1}s_2^{k-1}s_3^{k-1} \gamma_1^{k-1} \gamma_2. $$ To link up to the functor $\s_{I_{k+2}}$ we need to use the element 
			%$$ \gamma_1^{k+1}s_2^{k-1}s_3^{k-1} s_1^{k-1} \gamma_2.  $$ This link is described by the choices $\lambda_1', \lambda_2', \lambda_3' = \lambda_4, \lambda_2, \lambda_3 = k-1 $, $\lambda_1'' = \lambda_{1} = k+1$, $\lambda_2''= \lambda_5=1$, and $p = k+1+1-(2(k-2)+3)= -k+3$.
			%In terms of regular sequences if $I_k = (x_1, \ldots, x_{k+2})$ and $x_1, x_2, x_3$ is the regular sequence defining the link $I_k \sim J_k$, then the link $I_k \sim J_{k+2}$ is defined by a regular sequence of the form $(x_1, y_1, y_2)$ with $y_1, y_2 \in \m I_k$. The ideal $J_{k+2}$ has now the form $(x_1, y_1, y_2, z_1, z_2)$ with $ \deg z_1 < \deg z_2$. The next link $J_{k+2} \sim I_{k+2}$ is defined by a regular sequence of the form $(y_1+a_1, y_2+a_2, z_1+a_3)$ with $a_1, a_2, a_3 \in \m J_{k+2}.$
		\end{remark}

        Our next result shows that the family of licci ideals with decorations $\s_{I_k}$ corresponds exactly to the family of perfect ideals of codimension 3 having the largest perfect pairing on the free resolution after the Gorenstein ideals. To do this we use the results obtained in Section \ref{sec:classG} describing the decorations of ideals of Tor algebra class $G$.

\begin{lem}
\label{Sknuovo}
Let $\s_I$ be the decoration of a licci ideal and let $\s_{I'}$ be the decoration of its tight double link defined in Example \ref{tightdouble}. Suppose that $\lambda_i + \mu_i = \kappa(I)+1$ for $i\leq 3$ and that $I'= I_j$ where $j=\kappa(I')$. Then $I=I_k$ where $k=\kappa(I)$.
\end{lem}

\begin{proof}
Arguing as in the proof of Lemma \ref{Grinduction}, we use that $\lambda_i + \mu_i = k+1$ for $i=1,2$ to get
$$  \s_{I'}= S_{\lambda_3, \ldots, \lambda_b}F_1 \otimes S_{\lambda_3+p, \mu_2, \ldots, \mu_t}F_3^* = S_{I_j}= S_{2^{j-1},1,1,1}F_1 \otimes S_{j-1,1}F_3^*. $$  
From Example \ref{tightdouble}, we know that $ \lambda_3 = \mu_1+q $. Hence, $\lambda_3 + \lambda_3+p= \lambda_3 + \mu_1+q+p =k+1+q+p=\kappa(I')+1.  $ By Proposition \ref{restrictions}(4), $\lambda_3+p \geq \mu_2$ and therefore $\lambda_3+p = \kappa(I')-1 = k+p+q-1$. Combining these facts we get $ \lambda_3 = \mu_1+q = k+q-1 $. Hence $\mu_1 = k-1$ and the desired result follows by Proposition \ref{functor,k-1,1}. 
\end{proof}

\begin{thm}
\label{partialperfectpairing}
Le $I$ be a perfect ideal of codimension 3 minimally generated by $b$ elements. 
%and let $ \s_I$ be the corresponding decoration. 
Suppose 
%$\kappa(I)= k \geq 4$ and 
that $I$ is not Gorenstein. 
The following conditions are equivalent:
\begin{enumerate}
\item[(1)] $I$ is of class $G(r)$ with $r =b-3 $.
\item[(2)] $I$ is licci of class $G(r)$ with $r =k-1 $.
\item[(3)] $I$ is licci of class $G(r)$ with $r \geq k-1$. %with $k=d+1$.
\item[(4)] $I$ is licci and $\s_I = \s_{I_k}$.
\end{enumerate}
\end{thm}

\begin{proof}
The implication (2) $\Rightarrow$ (3) is obvious, while
the implications (4) $\Rightarrow$ (1) and (4) $\Rightarrow$ (2) follow immediately by Theorem \ref{multiprelation}.

We show the implication (3) $\Rightarrow$ (4).
%If $I$ is Gorenstein we know that $I$ is of class $G(r)$ with $r= 2k+1$. 
Suppose $I$ to be not Gorenstein and of class $G(r)$ with $r \geq k-1$. %$r \geq k-1$. 
It is sufficient to show $\mu_1=k-1$ and apply Proposition \ref{functor,k-1,1}.
By way of contradiction say that $\mu_1 \leq k-2$. 
The assumption of $G(r)$ implies, by Theorem \ref{multiprelation}, that $\lambda_i + \mu_1 = k+1$ for every $i=1, \ldots, r$. Hence, 
$$ 2k+1 \geq \lambda_{1} + \ldots + \lambda_{r} = r(k+1 -\mu_{1}) \geq (k-1)(k+1 -(k-2))= 3k-3. $$ It follows that $k \leq 4$ and %the result for these small values of $k$ 
it can be checked by inspection that, if $k\leq4$, the only Herzog class of Tor algebra class $G(r)$ (with $r \geq 3$) are the Gorenstein class and the one of $\s_{I_4}$. 

Finally, we prove (1) $\Rightarrow$ (4).
By Lemma \ref{forcetobelicci}, if $I$ satisfies condition (1), then it is licci. Looking at the tables in Section 4, we observe that the result is true both for small values of $k$ and if $I$ is directly linked to an almost complete intersection. Also we can check that $I$ is not an almost complete intersection.
Thus, we use a method similar to the one used in the proof of Theorem \ref{conjectureLars1}. 
By the assumption of class $G(b-3)$, %and $k\geq 4$, 
$\s_I$ satisfies the assumption of Lemma \ref{Grinduction}.
We work by induction on $k$. Using Lemma \ref{Grinduction}, %Theorem \ref{conjectureLars1} 
and the fact that (2),(3),(4) are equivalent, we can assume that the tight double link $\s_{I'}$ of $\s_I$ is either Gorenstein or of the form $\s_{I_j}$ with $j=\kappa(I')=k+p+q$ (see Example \ref{tightdouble} and observe that $\kappa(I')< k$). But if $I'$ is Gorenstein, then $I$ is directly linked to an almost complete intersection and the thesis follows. 

Otherwise, since $I'$ has $b-2$ generators, we can assume $b \geq 6$, $\lambda_3 + \mu_1 =k+1$ and
 $\s_{I'}= \s_{I_j}$. The thesis follows from Lemma \ref{Sknuovo}. 
\end{proof}

		In the next subsections we describe explicit free resolutions of ideals in each Herzog class $\s_{I_k}$. In an upcoming paper we will show that the ideals that we construct here define rigid algebras. % add reference for the upcoming paper on rigidity?. \ec
		
		\subsection{Construction of the model $I_k$ for $k$ even}
		
		Let $k \geq 2 $ be an even number. Let $X$ be a $3 \times k$ matrix with generic entries $x_{ij}$. Denote by $M_{ij}^{\hat{l}}$ the $2 \times 2$ minor of $X$ obtained by removing the $l$-th row and taking the $i$-th and $j$-th columns. Similarly denote by $M_{ijl}$ the $3 \times 3$ minor of $X$ on the columns $i,j,l$.
		
		Consider another set of variables $z_{ij}$ for $2 \leq i < j \leq k$ and let $Z= \lbrace Z_{ij} \rbrace$ be the skew-symmetric matrix such that $Z_{ij} = (-1)^{i+j} z_{i+1, j+1}$ if $j \neq i$ and $Z_{ii}=0$. We adopt the notation $P_{\hat{h_1}, \ldots, \hat{h_t}}$ to denote the pfaffian obtained from $Z$ removing the rows and columns $h_1 -1, \ldots, h_t-1$. By convention the pfaffian of an empty matrix is equal to $1$. %and the Pfaffians of a matrix of size $< 1$ are zero.
		
		Define a polynomial ring $R$ over a field $K$ with all the variables $x_{ij}, z_{ij}$ and a variable $y$. We fix the grading on $R$ setting $\deg(x_{ij})= \deg(y)= 1$ and $\deg(z_{ij})=2$. \\
		
		Define the ideal $I_k$ generated by the elements $ \alpha_2, \ldots, \alpha_{k}, \beta_1, \beta_2, \beta_3$, where
		$$ \alpha_r = yP_{\hat{r}} + \sum_{i,j \neq 1,r}^k M_{1ij} P_{\hat{r}, \hat{i}, \hat{j}} (-1)^{\sigma
			(i,j,r)},  $$
		$$ \beta_h=  \sum_{i=2}^k M_{1i}^{\hat{h}} P_{\hat{i}} (-1)^{i}. $$
		%Denote by $X'$ the matrix obtained from $X$ by removing the first column and changing the sign on all the entries of the second row.
		The second differential of the free resolution of $I_k$ is
		$$ d_2= \bmatrix 
		-M_{1,2}^{\hat{1}} & -M_{1,2}^{\hat{2}} & -M_{1,2}^{\hat{3}} &  0 &  0 & -z_{23} & \ldots & z_{2k} \\  
		-M_{1,3}^{\hat{1}} & -M_{1,3}^{\hat{2}} & -M_{1,3}^{\hat{3}} &  0 & z_{23} & 0 & \ldots & -z_{3k} \\
		\vdots & \vdots & \vdots & \vdots & \vdots & \vdots & \ddots & \vdots \\
		-M_{1,k}^{\hat{1}} & -M_{1,k}^{\hat{2}} & -M_{1,k}^{\hat{3}} & 0 & -z_{2k} & z_{3k} & \ldots & 0 \\
		y & 0 & 0 & x_{11} & -x_{12} & x_{13} & \ldots & -x_{1k} \\ 
		0 & y & 0 & -x_{21} & x_{22} & -x_{23} & \ldots & x_{12k} \\
		0 & 0 & y & x_{31} & -x_{32} & x_{33} & \ldots & -x_{3k} \\
		\endbmatrix. $$ %where $A$ is the matrix with entries
		%$$ A_{ij}= (-1)^{i+j} M_{1,i+1}^{\hat{j}}. $$ 
		The third differential of the free resolution of $I_k$ is
		$$ d_3= \bmatrix 
		x_{11} &  \sum_{h=2}^k x_{1h} P_{\hat{h}} (-1)^{h+1} \\  
		-x_{21} &  \sum_{h=2}^k x_{2h} P_{\hat{h}} (-1)^{h} \\ 
		x_{31} &  \sum_{h=2}^k x_{3h} P_{\hat{h}} (-1)^{h+1} \\ 
		-y &  \sum_{i,j,r \geq 2}^k M_{i,j,r} P_{\hat{i}, \hat{j}, \hat{r}} (-1)^{\sigma(i,j,r)} \\ 
		0 & \alpha_2 \\ 
		\vdots & \vdots \\ 
		0 & \alpha_k \\ 
		\endbmatrix. $$
		
		\begin{example}
			In the case $k=2$, we recover the resolution of the ideal generated by $\alpha_2 = y$ and by the three maximal minors of a $3 \times 2$ matrix in the variables $x_{ij}$. Therefore $I_2$ is a generic almost complete intersection of type 2.
		\end{example}
		
		\begin{example}
			For $k=4$, we obtain that $I_4$ is generated by $$yz_{34} + M_{134}, \quad yz_{24} - M_{124}, \quad yz_{23} + M_{123}, $$  $$M_{12}^{\hat{1}} z_{34}- M_{13}^{\hat{1}} z_{24} + M_{14}^{\hat{1}} z_{23}, \quad M_{12}^{\hat{2}} z_{34}- M_{13}^{\hat{2}} z_{24} + M_{14}^{\hat{2}} z_{23}, \quad M_{12}^{\hat{3}} z_{34}- M_{13}^{\hat{3}} z_{24} + M_{14}^{\hat{3}} z_{23}. $$
			The other differentials of the free resolution of $I_4$ are given by:
			$$ d_2= \bmatrix 
			-M_{1,2}^{\hat{1}} & M_{1,2}^{\hat{2}} & -M_{1,2}^{\hat{3}} & 0 & 0 & -z_{23} & z_{24}\\ 
			M_{1,3}^{\hat{1}} & -M_{1,3}^{\hat{2}} & M_{1,3}^{\hat{3}} & 0 & z_{23} & 0 & -z_{34}\\ 
			-M_{1,4}^{\hat{1}} & M_{1,4}^{\hat{2}} & -M_{1,4}^{\hat{3}} & 0 & -z_{24} & z_{34} & 0\\ 
			y & 0 & 0 & x_{11} & x_{12} & x_{13} & x_{14} \\ 
			0 & y & 0 & -x_{21} & -x_{22} & -x_{23} & -x_{24} \\
			0 & 0 & y & x_{31} & x_{32} & x_{33} & x_{34} \\
			\endbmatrix, $$
			$$ d_3= \bmatrix 
			x_{11} &   - x_{12}z_{34} + x_{13}z_{24} - x_{14}z_{23}  \\  
			-x_{21} &  - x_{22}z_{34} + x_{23}z_{24} - x_{24}z_{23} \\ 
			x_{31} &  - x_{32}z_{34} + x_{33}z_{24} - x_{34}z_{23} \\ 
			-y &  M_{234}  \\ 
			0 & yz_{34} + M_{134} \\ 
			0 & yz_{24} - M_{124} \\ 
			0 & yz_{23} + M_{123} \\ 
			\endbmatrix. $$
		\end{example}
		%\medskip 

		% prove that these links preserve rigidity, done by Xianglong! \ec
		
		%\medskip
		
		% Using induction on $k$, we need to do: \ec \\
		% Option 1: First prove that the above differentials define an exact complex $\FF$. Compute structure maps $w^{(3)}_1$ and $w^{(1)}_1$ on $\FF$. Use them to compute generators and $w^{(1)}_1$ of $J_{k+2}$ and then only the generators of $I_{k+2}$.  \\
		% Option 2: If we cannot prove exactness, we can compute twice the whole mapping cone and find the resolutions of $J_{k+2}$ and $I_{k+2}$. \\
		% also to do: each ideal is reached by the larger ones by non-homogeneous specialization of the new variables \ec
		
		\begin{remark}
			\label{exactF_k,even}
			Setting $d_1$ to be the row matrix with entries $(\alpha_2, \ldots, \alpha_{k}, \beta_1, \beta_2, \beta_3)$ and 
			using standard formulas involving relations of pfaffians and minors one can see that $d_1, d_2, d_3$ form the differentials of a complex $\FF_k$. To show that this complex is acyclic, one can use the Acyclicity Lemma, localize to any prime ideal of codimension $2$ and without loss of generality assume that the variable $z_{k-1,k}$ is a unit. Reducing then the complex to a minimal complex and specializing mapping to zero all the variables of the form $x_{i,k-1}, x_{ik}, z_{i,k-1}, z_{ik}$ with $(i,k) \neq (k-1, i)$, one obtains the complex $\FF_{k-2}$. By induction this shows that $\FF_k$ is acyclic.
		\end{remark}

		\begin{remark}
			\label{linkF_k,even}
			The complex $\FF_k$ from can be constructed inductively from $\FF_{k-2}$ using a double mapping cone construction. 
			One can define the double link of ideals in the following way:
			the link from $I_k $ to $ J_{k+2}$ is defined, introducing the opportune new variables, by the regular sequence $\alpha_1, \theta_1, \theta_2 $ where
			$$ \theta_i = \sum_{j=2}^{k-1} z_{j,k+i} \alpha_j + \sum_{h=1}^{3} x_{h,k+i} \beta_h. $$
			The ideal $J_{k+2}$ is then generated by $\alpha_1, \theta_1, \theta_2, \theta_3, \eta$. The regular sequence defining the link from $J_{k+2}$ to $ I_{k+2}$ is given by  
			$$ \theta_1 + \alpha_1 z_{k, k+1}, \quad \theta_2 + \alpha_1 z_{k, k+2}, \quad \theta_3 + \alpha_1 z_{k+1, k+2}. $$
            Using inductively Lemma \ref{lem:special-graded-free-res-link} as described in Algorithm \ref{algo:licci-examples}, we obtain that the ideal $ I_{k+2}$ is in the Herzog class corresponding to $\s_{I_k}$.
		\end{remark}

		\subsection{Construction of the model $I_k$ for $k$ odd}

		Let $k \geq 3 $ be an odd number. Let $X$ be a $3 \times k-1$ matrix with generic entries $x_{ij}$. Denote by $M_{ij}^{\hat{l}}$ the $2 \times 2$ minor of $X$ obtained by removing the $l$-th row and taking the $i$-th and $j$-th columns. Similarly denote by $M_{ijl}$ the $3 \times 3$ minor of $X$ on the columns $i,j,l$.
		
		Consider another set of variables $w_{ij}$ for $1 \leq i < j \leq k-1$ and let $W= \lbrace W_{ij} \rbrace$ be the skew-symmetric matrix such that $W_{ij} = (-1)^{i+j} w_{i+1, j+1}$ if $j \neq i$ and $W_{ii}=0$. We adopt the notation $P_{\hat{h_1}, \ldots, \hat{h_t}}$ to denote the pfaffian obtained from $W$ removing the rows and columns $h_1 -1, \ldots, h_t-1$. The pfaffian of $W$ is denoted by $P$. Again, by convention the pfaffian of an empty matrix is equal to $1$. %and the Pfaffians of a matrix of size $< 1$ are zero.
		
		Define a polynomial ring $R$ over a field $K$ with all the variables $x_{ij}, z_{ij}$ and two variables $y$ and $z$. We fix the grading on $R$ setting $\deg(x_{ij})= \deg(y)= \deg(z)= 1$ and $\deg(w_{ij})=2$. \\
		
		Define the ideal $I_k$ generated by $\alpha_1, \ldots, \alpha_{k-1}, \beta_1, \beta_2, \beta_3$, where
		$$  \alpha_r = z \sum_{i,j,l  \neq r} M_{ijl} P_{\hat{r}, \hat{i}, \hat{j}, \hat{l}} (-1)^{\sigma
			(i,j,l,r)} + 
		\sum_{i \neq r} P_{ \hat{i}, \hat{r}} (y_1 x_{1i} + y_2 x_{2i} + y_3 x_{3i}) (-1)^{\sigma
			(i,r)},  $$
		$$ \beta_h=  Py_h + z \sum_{i,j} M_{ij}^{\hat{h}} P_{\hat{i}, \hat{j}} (-1)^{i+j}. $$

		%Denote by $X'$ the matrix obtained from $X$ by changing the sign on all the entries of  the odd columns \ec and by $X''$ the transpose of $X$ after changing the sign to all the entries in odd columns \ec.
		The second differential of the free resolution of $I_k$ is
		$$  d_2= \bmatrix 
		-y_1 x_{11} - y_2 x_{21} - y_3 x_{31} &  0  & -w_{12} & \ldots & -w_{1k-1} & zx_{11} & zx_{21} & zx_{31} \\  
		y_1 x_{12} +y_2 x_{22} + y_3 x_{32} &  w_{12}  & 0 & \ldots & w_{2k-1} & -zx_{12} & -zx_{22} & -zx_{32} \\
		\vdots & \vdots & \vdots &  \ddots & \vdots & \vdots & \vdots & \vdots \\
		y_1 x_{1k-1} +y_2 x_{2k-1} + y_3 x_{3k-1}  &   w_{1k-1}  & -w_{2k-1} & \ldots & 0 & -zx_{1k-1} & -zx_{2k-1} & -zx_{3k-1} \\
		0 &  -x_{11} & x_{12}  & \ldots &  x_{1k-1} & 0 & y_3  & -y_2 \\ 
		0 &   -x_{21} & x_{22}& \ldots &  x_{2k-1}  & -y_3 & 0 & y_1\\
		0 &  -x_{31} & x_{32} & \ldots &  x_{3k-1}  &  y_2 & -y_1 & 0\\
		\endbmatrix. $$ %where $Y$ is the skew-symmetric matrix with entries $Y_{ij}= (-1)^{i+j} y_h$ with $h \in \lbrace 1,2,3 \rbrace \setminus \lbrace i,j \rbrace$ and $Y_{ii}=0$. 

		The third differential of the free resolution of $I_k$ is
		$$ d_3= \bmatrix 
		z &  P \\  
		0 &  \alpha_1 \\ 
		\vdots &  \vdots \\ 
		0 &  \alpha_{k-1} \\ 
		y_1 & \sum_{i,j} M_{ij}^{\hat{1}} P_{\hat{i}, \hat{j}} (-1)^{i+j} \\ 
		y_2 & \sum_{i,j} M_{ij}^{\hat{2}} P_{\hat{i}, \hat{j}} (-1)^{i+j} \\ 
		y_3 & \sum_{i,j} M_{ij}^{\hat{3}} P_{\hat{i}, \hat{j}} (-1)^{i+j} \\ 
		\endbmatrix. $$
		
		\begin{example}
			In the case $k=3$ we recover Anne Brown's model of format $(1,5,6,2)$ generated by $  y_1 x_{12} + y_2 x_{22} + y_3 x_{32}, \, y_1 x_{11} + y_2 x_{21} + y_3 x_{31}, \, w_{12} y_1 + z M_{12}^{\hat{1}}, \, w_{12} y_2 - z M_{12}^{\hat{2}}, \, w_{12} y_3 + z M_{12}^{\hat{3}}.  $
			The other differentials of the resolutions are given by: 
			$$ d_2= \bmatrix 
			y_1 x_{11} + y_2 x_{21} + y_3 x_{31} & 0 & -w_{12} &  -zx_{11} & -zx_{21} & -zx_{31} \\  
			-y_1 x_{12} - y_2 x_{22} - y_3 x_{32} & w_{12} & 0 &  zx_{12} & zx_{22} & zx_{32} \\ 
			0 & -x_{11} & x_{12} & 0  & -y_3 & y_2 \\ 
			0 & -x_{21} & x_{22} &  y_3 & 0 & -y_1 \\ 
			0 & -x_{31} & x_{32} &  -y_2 & y_1 & 0 \\ 
			\endbmatrix, $$
			$$ d_3= \bmatrix 
			z &  w_{12} \\  
			0 &   y_1 x_{12} + y_2 x_{22} + y_3 x_{32} \\ 
			0 &  y_1 x_{11} + y_2 x_{21} + y_3 x_{31} \\ 
			y_1 & M_{12}^{\hat{1}} \\ 
			y_2 & -M_{12}^{\hat{2}} \\ 
			y_3 & M_{12}^{\hat{3}} \\ 
			\endbmatrix. $$
		\end{example}
		
		\begin{example}
			For $k=5$,  setting $\delta_i= y_1 x_{1i} + y_2 x_{2i} + y_3 x_{3i}$, we have 
			%$$\alpha_1 = z M_{234} +  z_{34}\delta_2 - z_{24}\delta_3 + z_{23}\delta_4,$$ $$\beta_1 = Py_1 + z ( M_{12}^{\hat{1}} z_{34} - M_{13}^{\hat{1}} z_{24} + M_{14}^{\hat{1}} z_{23} + M_{23}^{\hat{1}} z_{14} - M_{24}^{\hat{1}} z_{13} + M_{34}^{\hat{1}} z_{12}),$$ where $P= z_{12} z_{34} - z_{13} z_{24}+ z_{14} z_{23}.$
			% $I_5$ is generated by $ z M_{234} +  z_{34}\delta_2 - z_{24}\delta_3 + z_{23}\delta_4,  \, z M_{134} +  z_{34}\delta_1 + z_{14}\delta_3 - z_{13}\delta_4, \,  z M_{124} + z_{24}\delta_1 - z_{14}\delta_2 - z_{12}\delta_4, \, z M_{123} + z_{23}\delta_1 - z_{13}\delta_2 + z_{12}\delta_1, \,  y_1 x_{11} + y_2 x_{21} + y_3 x_{31}, \, w_{12} y_1 + z M_{12}^{\hat{1}}, \, w_{12} y_2 - z M_{12}^{\hat{2}}, \, w_{12} y_3 + z M_{12}^{\hat{3}}. $
			%The differentials of the resolutions are given by:
			%The 
			that the second differential of the resolution is given by 
			$$ d_2= \bmatrix 
			\delta_1 & 0 & -w_{12} & w_{13} & -w_{14} &  -zx_{11} & -zx_{21} & -zx_{31} \\  
			-\delta_2 & w_{12} & 0 & -w_{23} & w_{24} &  zx_{12} & zx_{22} & zx_{32} \\ 
			\delta_3 & -w_{13} & w_{23} & 0 & -w_{34} & -zx_{13} & -zx_{23} & -zx_{34} \\ 
			-\delta_2 & w_{14} & -w_{24} & w_{34} & 0 & zx_{14} & zx_{24} & zx_{34} \\ 
			0 & -x_{11} & x_{12} & -x_{13} & x_{14} & 0  & -y_3 & y_2 \\ 
			0 & -x_{21} & x_{22} & -x_{23} & x_{24} &  y_3 & 0 & -y_1 \\ 
			0 & -x_{31} & x_{32} & -x_{33} & x_{34} &  -y_2 & y_1 & 0 \\ 
			\endbmatrix. $$
			The other two differentials can be computed according to the general formulas.
		\end{example}
		% $$ d_3= \bmatrix 
		% z &  P \\  
		% 0 &  \alpha_1 \\ 
		% \vdots &  \vdots \\ 
		% 0 &  \alpha_{k-1} \\ 
		% y_1 & \sum_{i,j} M_{ij}^{\hat{1}} P_{\hat{i}, \hat{j}} (-1)^{i+j} \\ 
		% y_2 & \sum_{i,j} M_{ij}^{\hat{2}} P_{\hat{i}, \hat{j}} (-1)^{i+j} \\ 
		% y_3 & \sum_{i,j} M_{ij}^{\hat{3}} P_{\hat{i}, \hat{j}} (-1)^{i+j} \\ 
		% \endbmatrix. $$

		% prove the same things as in the even case \ec

		\begin{remark}
			\label{exactF_k,odd}
			Set $d_1$ to be the row matrix with entries $(\alpha_1, \ldots, \alpha_{k-1}, \beta_1, \beta_2, \beta_3)$
			Similarly to the case where $k$ is even, one can show that the maps $d_1, d_2, d_3$ form the differentials of an exact complex $\FF_k$. In this case one localize to invert the variable $w_{k-1, k-2}$ and specialize to zero all the variables $x_{i,k-2}, x_{i,k-1}, w_{i,k-2}, w_{i,k-1}$ with $(i,k-1) \neq (k-2, k-1)$.
			%Setting $d_1$ to be the row matrix with entries $(\alpha_2, \ldots, \alpha_{k}, \beta_1, \beta_2, \beta_3)$ and using standard formulas involving relations of pfaffians and minors one can see that $d_1, d_2, d_3$ form the differentials of a complex $\FF_k$. To show that this complex is acyclic, one can use the Acyclicity Lemma, localize to any prime ideal of codimension $2$ and without loss of generality assume that the variable $z_{k-1,k}$ is a unit. Reducing then the complex to a minimal complex and specializing mapping to zero all the variables of the form $x_{i,k-1}, x_{ik}, z_{i,k-1}, z_{i,k}$ with $(i,k) \neq (k-1, i)$, one obtains the complex $\FF_{k-2}$. By induction this shows that $\FF_k$ is acyclic.
		\end{remark}
		
		\begin{remark}
			\label{linkF_k,odd}
			In this case the link from $I_k $ to $ J_{k+2}$ is defined, introducing the opportune new variables, by the regular sequence $\alpha_1, \theta_1, \theta_2 $ where
			$$ \theta_i = \sum_{j=2}^{k-1} w_{j,k-1+i} \alpha_j + \sum_{h=1}^{3} x_{h,k-1+i} \beta_h. $$
			The ideal $J_{k+2}$ is then generated by $\alpha_1, \theta_1, \theta_2, \theta_3, \eta$. The regular sequence defining the link from $J_{k+2} $ to $ I_{k+2}$ is given by  
			$$ \theta_1 + \alpha_1 w_{1, k}, \theta_2 + \alpha_1 w_{1, k+1}, \theta_3 + \alpha_1 w_{k, k+1}. $$
			Also in this case, using inductively Lemma \ref{lem:special-graded-free-res-link} as described in Algorithm \ref{algo:licci-examples}, we obtain that the ideal $ I_{k+2}$ is in the Herzog class corresponding to $\s_{I_k}$.
		\end{remark}
		
		% Should we say more about this? The details would require many pages of computations \ec

		%\section{Decorations of zero dimensional monomial ideals}

\section{The case of arbitrary codimension}
\label{sec:arbitrary}
			%\begingroup \color{magenta}
			There is evidence to suggest that the framework of Section \ref{sec:preliminaries} generalizes to arbitrary codimension $c \geq 2$; see for instance \cite[\S2.3]{xianglong}. Specifically, we change the left arm of the graph $T$ to have length $c-2$, and we consider the set $\leftindex^{z_1}{W(d,t)}^{x_{c-2}}$ of minimal length representatives for their respective double cosets. We analogously define the limiting graph $\leftindex^{z_1}{\overline{W}}^{x_{c-2}}$ and the ``combinatorial licci graphs'' $\mathrm{Licci}^{\mathrm{bi}}_c$ and $\mathrm{Licci}_c$.
			
			The vertices of $\mathrm{Licci}_c$ are given by $\leftindex^{z_1}{\overline{W}}^{x_{c-2}}$ by definition, and they may be identified with certain pairs of partitions $(\boldsymbol \lambda,\boldsymbol \mu)$ using the same formula \eqref{eq:weights-to-partitions}. One now has that
			\[
				\sum \lambda_i = 1+(c-1)\sum \mu_i.
			\]
			As before, we will typically use $k$ to denote $\sum \mu_i$. We will often write these pairs $(\boldsymbol \lambda,\boldsymbol \mu)$ in the notation $S_{\boldsymbol \lambda} F_1 \otimes S_{\boldsymbol \mu} F_c^*$ in view of the representation theory perspective given in Section \ref{sec:background-reps}.
			
			One has the following analogue of Theorem~\ref{thm:graph-edges-partitions}. The proof is almost verbatim the same, so we omit it.
			\begin{thm}\label{thm:graph-edges-partitions-codim-c}
				Let $\sigma \in \leftindex^{z_1}{\overline{W}}^{x_{c-2}}$ correspond to the pair of partitions $(\boldsymbol\lambda,\boldsymbol\mu)$, where $\boldsymbol\lambda = (\lambda_1,\lambda_2,\ldots)$ and $\boldsymbol\mu = (\mu_1,\mu_2,\ldots)$ are viewed as infinite lists with trailing zeros. Let $\lambda'_1 \geq \ldots \geq \lambda'_c$ be parts of $\boldsymbol\lambda$ (possibly equal to zero) and let $\boldsymbol\mu^\mathrm{link}$ be the partition obtained from $\boldsymbol\lambda$ after removing these parts. Define
				\[
				\boldsymbol\lambda^\mathrm{link} = \operatorname{rsort}(\lambda'_1 + p, \ldots, \lambda'_c + p,\mu_1,\mu_2,\mu_3,\ldots).
				\]
				where $\operatorname{rsort}$ means to sort in non-increasing order and
				\[
				p = \sum_{j=1}^\infty \mu^\mathrm{link}_j - \sum_{j=1}^\infty \mu_j.
				\]
				Then the pair $(\boldsymbol\lambda^\mathrm{link},\boldsymbol\mu^\mathrm{link})$ corresponds to some $\sigma^\mathrm{link} \in \leftindex^{z_1}{\overline{W}}^{x_{c-2}}$ adjacent to $\sigma$ on $\mathrm{Licci}_c$, and all neighbors of $\sigma$ are obtained in this fashion for some choice of parts $\lambda'_1 \geq \ldots \geq \lambda'_c$.
			\end{thm}
			
			Our hope is that the theory generalizes in the following way. We include in parentheses the results that each of the following points is intended to generalize.
			\begin{conjecture}\label{conj:general-c}
				Let $(R,\mf{m},\Bbbk)$ be a local Noetherian ring containing $\mb{Q}$.
				\begin{enumerate}
					\item There exists a notion of ``higher structure maps'' for codimension $c$ perfect ideals $I \subset R$. In particular, there are maps of the form $S_{\boldsymbol \lambda} F_1 \otimes S_{\boldsymbol \mu} F_c^* \to R$, where the $F_i$ denote the free modules in a free resolution of $R/I$. These maps are the components of a single map $w^{(1)} \colon V \otimes R \to R$, where $V = L(\omega_{x_{c-2}})^\vee$ is a representation of $\mf{g}(T)$. The bottom component of $w^{(1)}$ is the differential $d_1 \colon F_1 \to R$. If the ideal $I$ is homogeneous, then these maps may be chosen to be homogeneous as well. (\cite[\S5]{GNW3})
					
					\item The perfect ideal $I$ is licci if and only if $w^{(1)}$ is nonzero mod $\mf{m}$. The lowest representation $\s_I = S_{\boldsymbol \lambda} F_1 \otimes S_{\boldsymbol \mu} F_c^*$ on which it is nonzero mod $\mf{m}$ is necessarily extremal, and intrinsic to the ideal $I$. (\cite[Theorem 6.4]{GNW3}, paraphrased here as Theorem~\ref{licci1})
					
					\item The assignment $I \mapsto \s_I$ induces a graph isomorphism
					\[
						\left\{\begin{matrix}
							\text{Herzog classes of codimension $c$}\\
							\text{licci ideals and the unit ideal}
						\end{matrix}\right\} \xrightarrow{\sim} \mathrm{Licci}_c
					\]
					where edges on the left are given by direct linkage. (\cite[Theorem 7.3]{GNW3} and Theorem~\ref{thm:licci-graph})
					\item Using $w^{(1)}$, one can define special generating systems (SGS) in the same way as Definition~\ref{def:SGS}, and their application to linkage should be analogous to Remark~\ref{re2}.
				\end{enumerate}
			\end{conjecture}

			\subsection{The linkage formula in codimension $c$}
			
			Let $I$ be a licci ideal of codimension $c$, deviation $d$, and type $t$, or the unit ideal. Then, assuming Conjecture~\ref{conj:general-c}, there is an associated $\s_I = S_{\boldsymbol \lambda}F_1 \otimes S_{\boldsymbol \mu}F_c^*$ where $\boldsymbol \mu = (\mu_1, \ldots, \mu_{t})$ is a partition of $k$ and
			$\boldsymbol \lambda = (\lambda_1, \ldots, \lambda_{d+c})$ is a partition of $(c-1)k+1$. We define $\kappa(I) = k$. Again, these pairs of partitions/Schur functors associated to licci ideals will be called \it decorations. \rm Our work in this section is primarily focused on the combinatorial analysis of these decorations, which remains valid independently of Conjecture~\ref{conj:general-c}. However, we present all our results assuming the validity of the conjecture, since point (3) is the main reason for our interest in these decorations anyway.
			
			Assuming Conjecture~\ref{conj:general-c} (4), we have the following translation of Theorem~\ref{thm:graph-edges-partitions-codim-c}.   
			\begin{definition} \rm (Linkage formula for decorations) 
				\label{linkageformulahighcodim}
				Let $I$ be a codimension $c$ licci ideal (or the unit ideal) with $\s_I = S_{\boldsymbol \lambda}F_1 \otimes S_{\boldsymbol \mu}F_c^*$. Let $\boldsymbol \lambda' = (\lambda'_1,\ldots,\lambda'_c)$ be a sublist of $\boldsymbol \lambda$, where we view $\boldsymbol \lambda$ as having infinitely many trailing zeros, and let $\boldsymbol \lambda'' = \boldsymbol \lambda \backslash \boldsymbol \lambda'$. Denote by $G_1, \ldots, G_c$ the free modules in the minimal free resolution of $R/J$ where $J$ is linked to $I$ in the same manner as Remark~\ref{re2} for the chosen $\boldsymbol \lambda'$ (generalized in the evident way to arbitrary $c$).
				
				Then:
				$$ \s_J =  S_{\lambda'_1+p, \ldots, \lambda'_c+p, \mu_1, \ldots, \mu_t} G_1 \otimes S_{\boldsymbol \lambda''} G_3^*, $$
				after reordering the first of the above partitions in non-increasing order, where
				\begin{equation}
					\label{phighcodim}
					p = \kappa(J) - \kappa(I) = (c-2)k+1-\lambda_1' - \cdots - \lambda_c'.
				\end{equation}
			\end{definition}
			If $I$ is the unit ideal, then $\s_I = F_1$ by Conjecture~\ref{conj:general-c} (1), since the differential $d_1 \colon F_1 \to R$ would be nonzero mod $\mf{m}$. It follows that the decorations $\s$ are exactly the pairs of Schur functors/partitions that can be obtained via finitely many applications of this formula, starting from the decoration $F_1$. For example, applying this formula to the unit ideal with $\boldsymbol \lambda' = (0,\ldots,0)$ results in a complete intersection $J$ with $\s_J = \bigwedge^c G_1 \otimes G_c^*$. Conjecture~\ref{conj:general-c} (1) then suggests that complete intersections can be characterized by the nonvanishing of some map $\bigwedge^c G_1 \to G_c$ modulo $\mf{m}$. This makes sense, since, indeed, the comparison map from the Koszul complex on the generators of $J$ to the resolution $\mb{G}$ of $R/J$ yields such a map.
			
			\begin{remark}\label{rem:recursive-s_I-codim-c}
				Assuming Conjecture~\ref{conj:general-c} (2), specifically that the assignment $I \mapsto \s_I$ is well-defined for licci ideals, this gives a way of recursively computing $\s_I$ for a licci ideal $I$. However, it cannot be used as a standalone definition, because \emph{a priori} it depends on the choice of links from $I$ to a complete intersection.
			\end{remark}
			
			%\endgroup
			
			Combinatorially, the smallest and largest minimal links, the generic link, and the minimal tight double link can be defined analogously (in terms of $c$) as for the codimension 3 case; see Examples~\ref{minimalminimal}-\ref{tightdouble}.
			
			The next lemma provides a first restrictions on the elements of the partitions of a decoration.	
		
		\begin{lem}
			\label{mu1lambdac}
			Let $\s= S_{\boldsymbol \lambda}F_1 \otimes S_{\boldsymbol \mu}F_c^*$ be a decoration of a licci ideal of codimension $c$. Let $p = (c-2)k+1-\lambda_1-\ldots-\lambda_c$ be the $p$ corresponding to the smallest minimal link. Then $0 \leq \lambda_c + p < \mu_1$.
		\end{lem}
		
		\begin{proof}
			Apply the smallest minimal link to $ \s $ to get $$ \s'= S_{\lambda_1+p, \ldots, \lambda_c+p, \mu_1, \ldots, \mu_t} F_1' \otimes S_{\lambda_{c+1}, \ldots, \lambda_{b}} (F_c')^{\ast}. $$ Since $\s'$ is also a decoration of some licci ideal, we must have $ \lambda_c + p \geq 0. $ Suppose $ \lambda_c + p \geq \mu_1. $ Then the smallest minimal link of $s'$ is
			$$ \s''= S_{\lambda_1+p+p', \ldots, \lambda_c+p+p',\lambda_{c+1}, \ldots, \lambda_{b}} F_1'' \otimes S_{\mu_1, \ldots, \mu_t} (F_c'')^{\ast}. $$ But $p' = (c-2)(k+p)+1-\lambda_1-\ldots-\lambda_c - cp = -p.$ Thus $\s'' = \s$ and this is a contradiction since a decoration can always be linked to a complete intersection by smallest minimal links. Hence, $ \lambda_c + p < \mu_1. $
		\end{proof}

        \begin{cor}
			\label{lambda<k}
			Let $\s = S_{\boldsymbol \lambda}F_1 \otimes S_{\boldsymbol \mu}F_c^*$ be a decoration and let $k= \sum_{i=1}^t \mu_i$. Then 
			$\lambda_i \leq k$ for every $i$.
		\end{cor}

        \begin{proof}
        If $k=1$, the equality $\sum \lambda_i = (c-1)k+1=c$ forces $\s=\bigwedge^c F_1 \otimes F_c^*$. The result is true in this case. We can then work by induction on $k$.
            Let $k>1$ and let $\s'$ be the minimal tight double link of $\s$. By Lemma \ref{mu1lambdac}, $\kappa(\s')=k+p+q<k$. By inductive hypothesis we get $\lambda_1+p+q \leq \kappa(\s')=k+p+q$. Hence, $\lambda_1 \leq k$ and the result follows.
        \end{proof}

		We prove a fundamental relation involving the sums of squares of the partition elements of a decoration. In the case of $c=3$, this can be used to prove the relation stated in item (3) of Proposition \ref{restrictions}.
		
		\begin{thm} \rm (Squares formula) \it
			\label{squareformulas}
			Let $\s = S_{\boldsymbol \lambda}F_1 \otimes S_{\boldsymbol \mu}F_c^*$ be a decoration. Then
			$$ \sum_{i=1}^{d+c} \lambda_i^2 + \sum_{i=1}^{t} \mu_i^2 = (c-2)k^2 + 2k+1. $$
		\end{thm}
		
		\begin{proof}
			% the proof depends only on the combinatoric of linkage and not on free resolutions! \ec \\
			If $\s$ is the decoration of a complete intersection, then $\s = \bigwedge^c F_1 \otimes F_c^*$, $k=1$ and the result follows.
			Then we prove that if $\s $ and $ \s'$ are linked and the above formula holds for $\s$ then it holds also for $\s'$. We use Definition \ref{linkageformulahighcodim} to define the link from $\s$ to $\s'$. We have to prove that 
			\begin{equation}
				\label{eqsquare}
				\sum_{i=1}^{c} (\lambda_i'+p)^2 + \sum_{i=1}^{t} \mu_i^2 + \sum_{i=1}^{d+c} (\lambda_i'')^2 = (c-2)(k+p)^2 + 2(k+p)+1.
			\end{equation}  
			We expand the left-hand side of (\ref{eqsquare}) as
			$$  \sum_{i=1}^{c} (\lambda_i')^2 + 2p \sum_{i=1}^{c} \lambda_i' + cp^2 + \sum_{i=1}^{t} \mu_i^2 + \sum_{i=1}^{d+c} (\lambda_i'')^2.  $$
			By definition of the partitions $\lambda'$ and $\lambda''$ and by assumption on $\s$, we obtain 
			$$ \sum_{i=1}^{c} (\lambda_i')^2 +  \sum_{i=1}^{d+c} (\lambda_i'')^2 =  \sum_{i=1}^{d+c} \lambda_i^2 = (c-2)k^2 + 2k+1 - \sum_{i=1}^{t} \mu_i^2. $$
			Combining the above formulas, we obtain that the left-hand side of (\ref{eqsquare}) is equal to
			$$  2p \sum_{i=1}^{c} \lambda_i' + cp^2 + (c-2)k^2 + 2k+1. $$ We use now the definition of $p$ (eq. \ref{phighcodim}) to write 
			$ \sum_{i=1}^{c} \lambda_i' = (c-2)k+1-p. $ An easy computation gives the desired result.
		\end{proof}

		\begin{remark}
			Theorem \ref{squareformulas}, in the case $c=2$, shows that the only decorations are of the form $\bigwedge^{k+1} F_1 \otimes \bigwedge^k F_2^*$ for any $k \geq 1$.
			Our linkage formula for the decorations recovers the well-known linkage theory of perfect ideals of codimension 2.
		\end{remark}
		
%\begingroup \color{magenta}
			
			In Section \ref{sec:freeres}, it was demonstrated how homogeneous  licci ideals of codimension $c=3$ with the modules $F_1$ and $F_c$ generated in particular degrees enjoy nice properties. The machinery used to establish these properties is unfortunately absent for $c\geq 4$, but assuming Conjecture~\ref{conj:general-c}, we can lay out the expectations.
			
			First, we observe a lemma that is true independently of Conjecture~\ref{conj:general-c}.
			
			\begin{lem}\label{lem:mapping-cone-decorations}
				Let $\s = S_{\boldsymbol \lambda}F_1 \otimes S_{\boldsymbol \mu}F_c^*$ be a decoration, and let $d,t,k$ be as usual. Suppose $I \subset R$ is a homogeneous  perfect ideal of codimension $c$ and
\begin{equation}\label{eq:special-graded-free-res-codim-c}
				\mb{F} \colon 0 \to \bigoplus_{i=1}^t R(-(k+c-2+\mu_i)) \to \cdots \to \bigoplus_{i=1}^{c+d} R(-(k+1-\lambda_i)) \to R
				\end{equation}
				is a minimal length graded free resolution of $R/I$, with first differential $\begin{bmatrix}
					x_1 & \cdots & x_{c+d}
				\end{bmatrix}$ where $\deg(x_i) = k+1-\lambda_i$. Let $\lambda'_1 \geq \ldots \geq \lambda'_c$ be a sublist of $(\lambda_1,\ldots,\lambda_{c+d},0,0,0,\ldots)$, and for $i = 1,\ldots,c$ take $\alpha_i$ as follows:
				\begin{itemize}
					\item If $\lambda_i' = \lambda_j > 0$, let $\alpha_i = x_j$.
					\item If $\lambda_i' = 0$, let $\alpha_i \in \mf{m}I$ have degree $k+1$.
				\end{itemize}
				Suppose that $\alpha_1,\ldots,\alpha_c$ is a regular sequence. Then $R/((\alpha_1,\ldots,\alpha_c):I)$ has a minimal length graded free resolution of the form
				\[
				0 \to \bigoplus_{i=1}^{t'} R(-((k+p)+c-2+\mu^\mathrm{link}_i)) \to \cdots \to \bigoplus_{i=1}^{c+t} R(-((k+p)+1-\lambda^\mathrm{link}_i)) \to R.
				\]
				where $t'$ is obtained as $d$ plus the number of $\lambda'_i$ equal to zero, $p = (c-2)k +1 - \sum_j \lambda'_j$, and $\boldsymbol \lambda^\mathrm{link}, \boldsymbol \mu^\mathrm{link}$ are as in Theorem~\ref{thm:graph-edges-partitions-codim-c}.
			\end{lem}
			\begin{remark}\label{rem:mapping-cone-decorations}
				As $k+p = \sum \mu^\mathrm{link}_i$, the resulting resolution looks very similar to the original \eqref{eq:special-graded-free-res-codim-c}, but with $\s^\mathrm{link}$ in place of $\s$. However, the new $\boldsymbol\lambda^\mathrm{link}$ may have fewer than $c+t$ nonzero parts.
				
				That being said, if the link going backwards from $\s^\mathrm{link}$ to $\s$ is minimal, then $\boldsymbol \lambda^\mathrm{link}$ will necessarily have exactly $c+t$ nonzero parts, so the output is exactly \eqref{eq:special-graded-free-res-codim-c} applied to $\s^\mathrm{link}$ in this case.
			\end{remark}
			\begin{proof}
				$R/((\alpha_1,\ldots,\alpha_c):I)$ can be resolved using the dual of the mapping cone of a comparison map from the Koszul complex on $\alpha_1,\ldots,\alpha_c$ to $\mb{F}$. Such a comparison map has the form
				\[
				\begin{tikzcd}
					0 \ar[r]& \bigoplus_{i=1}^t R(-(k+c-2+\mu_i)) \ar[r]& \cdots\ar[r] & \bigoplus_{i=1}^{c+d} R(-(k+1-\lambda_i)) \ar[r]& R\\
					0 \ar[r]& R(-(c(k+1) -\sum_i \lambda'_i)) \ar[r]\ar[u]& \cdots\ar[r]& \bigoplus_{i=1}^c R(-(k+1-\lambda'_i)) \ar[r]\ar[u]& R\ar[u,equals]
				\end{tikzcd}
				\]
				hence a resolution of the linked ideal is given by
				\begin{align*}
					0 \to& \begin{matrix}
						\bigoplus_{\lambda_i \in \lambda\backslash \lambda'} R(-(c(k+1)-\sum_j \lambda'_j - (k+1-\lambda_i))
					\end{matrix}
					\to\cdots\\[1em]
					&\cdots\to
					\begin{matrix}
						\bigoplus_{i=1}^c R(-(k+1-\lambda'_i))\\
						\oplus\\
						\bigoplus_{i=1}^t R(-(c(k+1)-\sum_j \lambda'_j - (k+c-2+\mu_i))
					\end{matrix} \to R.
				\end{align*}
				This simplifies to the desired form by noting that
				\begin{align*}
					c(k+1) - \sum_j \lambda'_j - (k+1-\lambda_i) &= k + p + c - 2 + \lambda_i\\
					k+1 - \lambda'_i &= k+p + 1 - (\lambda'_i + p)\\
					c(k+1)-\sum_j \lambda'_j - (k+c-2+\mu_i) &= k+p+1-\mu_i
				\end{align*}
				together with the definitions of $\boldsymbol \lambda^\mathrm{link}, \boldsymbol \mu^\mathrm{link}$.
			\end{proof}
			
			\begin{thm}
				\label{degreesgeneric}
				Let $\s = S_{\boldsymbol \lambda}F_1 \otimes S_{\boldsymbol \mu}F_c^*$ be a decoration other than $F_1$. There exists a codimension $c$ homogeneous ideal $I$ in the standard graded polynomial ring $R = K[x_1,\ldots,x_c]$ such that
				\begin{itemize}
					\item $I$ is homogeneously licci.
					\item $R/I$ has a graded free resolution of the form \eqref{eq:special-graded-free-res-codim-c}.
				\end{itemize}
			\end{thm}
			\begin{proof}
				We follow Algorithm~\ref{algo:licci-examples}. First, we repeatedly take smallest minimal links to combinatorially link to the decoration of a complete intersection:
				\[
				\s = \s_N \sim \cdots \sim \s_1 \sim \s_0 = \bigwedge^c F_1 \otimes F_c^*.
				\]
				We then inductively construct ideals $I_j$ which satisfy the conclusions of the theorem for $\s_j$, and which each contain a regular sequence of minimal degree.
				
				We start by taking $I_0 = (x_1,\ldots,x_c)$. We construct $I_{j+1}$ from $I_j$ by linking using a homogeneous regular sequence as in Lemma~\ref{lem:mapping-cone-decorations}, where $\lambda'$ corresponds to the link from $\s_j$ to $\s_{j+1}$. Since $I_j$ has a regular sequence of minimal degree, such a homogeneous regular sequence always exists. In view of Remark~\ref{rem:mapping-cone-decorations}, the resulting $I_{j+1}$ satisfies the desired conclusion, and furthermore $\alpha_1,\ldots,\alpha_c$ appear as minimal degree generators of $I_{j+1}$ because $\s_{j+1} \sim \s_{j}$ was a smallest minimal link. This completes the induction.
			\end{proof}
					
			\begin{lem} 
				Assume Conjecture~\ref{conj:general-c}. Then the analogue of Lemma~\ref{lem:special-graded-free-res} holds: let $\s$ be a decoration, $R$ be graded, and $I \subset R$ be a homogeneous codimension $c$ perfect ideal such that $R/I$ has a graded free resolution of the form \eqref{eq:special-graded-free-res-codim-c}. Let $\mf{p} \subset R$ be a prime ideal such that $R_i \subset \mf{p}$ for all $i \neq 0$. If $I_\mf{p}$ is a licci ideal in $R_\mf{p}$, then:
				\begin{enumerate}
					\item $\s_I = \s$. In particular, the ends of the resolution \eqref{eq:special-graded-free-res-codim-c} are minimal.
					\item Any minimal generating set for $I$, ordered so that the degrees are non-decreasing, forms an SGS for $I_\mf{p}$ after localization.
					\item All of the graded Betti numbers of $R/I$ are completely determined by the two ends in \eqref{eq:special-graded-free-res-codim-c}.
				\end{enumerate}
			\end{lem}
			
			\begin{proof}
				Assuming the conjecture, the proof of this statement is the same as Lemma~\ref{lem:special-graded-free-res}. The comment in (1) follows because $\s_I = \s$ implies that the ideal has deviation exactly $d$ and type exactly $t$. Point (3) is not as explicit here as in the $c=3$ case, because we do not know the graded minimal free resolution of the generic examples for arbitrary $c$.
			\end{proof}
			%\endgroup

		The next corollary generalizes item (3) of Corollary \ref{restrictions2} providing a conjectural description of the decorations of hyperplane sections in arbitrary codimension.

		\begin{cor}
			\label{hyperplanesection}
            Assume Conjecture~\ref{conj:general-c}.
            Let $\s_I= S_{\boldsymbol \lambda}F_1 \otimes S_{\boldsymbol \mu}F_c^*$ be the decoration of a licci ideal of codimension $c$. Let $k= \sum_{i=1}^t \mu_i$.
			The following conditions are equivalent: 
			\begin{enumerate}
				\item[(1)] $\lambda_1 = k$.
				\item[(2)] $I$ is an hyperplane section of a licci ideal of codimension $c-1$. 
			\end{enumerate}
            Furthermore, if $I'$ is an hyperplane section of $I$, then 
			$ \s_{I'}= S_{k,\boldsymbol \lambda}F_1 \otimes S_{\boldsymbol \mu}F_{c+1}^*.  $
		\end{cor}
		
		Adopting the notation of the previous corollary, we say that the decoration $\s_{I'}$ is an \it hyperplane section \rm of $\s_I$.
		We analyze now the decorations for small values of $k$.
		
		\begin{thm} 
			\label{smallk}
			Let $\s = \s_I = S_{\boldsymbol \lambda}F_1 \otimes S_{\boldsymbol \mu}F_c^*$ be a decoration and let $k = \sum_{i=1}^t \mu_i$. Then:
			\begin{enumerate}
				\item[(1)]  $k=1$ if and only if $I$ is a complete intersection. 
				\item[(2)]  $k=2$ if and only if $\s$ is obtained as iterated hyperplane section of one of the following decorations:
				\begin{itemize}
					\item $\bigwedge^3 F_1 \otimes \bigwedge^2 F_2^* $ \, (almost complete intersection of codimension 2).
					\item $\bigwedge^5 F_1 \otimes S_2 F_3^* $ \, (Gorenstein ideal of codimension 3 and deviation 2).
				\end{itemize}
				%either $\bigwedge^3 F_1 \otimes \bigwedge^2 F_2^* $ (almost complete intersection of codimension 2) or of $\bigwedge^5 F_1 \otimes S_2 F_3^* $ (Gorenstein ideal of codimension 3 and deviation 2).
				\item[(3)] $k=3$ if and only if $\s$ is obtained as iterated hyperplane section of one of the following decorations:
				\begin{itemize}
					\item $\bigwedge^4 F_1 \otimes \bigwedge^3 F_2^* $ \, (perfect ideal of codimension 2 and deviation 2).
					\item $S_{2,2,2,1} F_1 \otimes \bigwedge^3 F_3^*$ \, (almost complete intersection of codimension 3 and type 3).
					\item $S_{2,2,1,1,1} F_1 \otimes S_{2,1} F_3^*$ \, (Anne Brown's model of codimension 3, deviation 2, type 2).
					\item $\bigwedge^7 F_1 \otimes S_3 F_3^* $ \, (Gorenstein ideal of codimension 3 and deviation 4).
					\item $S_{2,2,2,2,2} F_1 \otimes S_{2,1} F_4^*$ \, (almost complete intersection of codimension 4 and type 2 that is not hyperplane section - see Section \ref{sec:gor4} for more details).
					\item $S_{2,2,2,1,1,1,1} F_1 \otimes S_{3} F_4^*$ \, (Kustin-Miller's model of Gorenstein ideals of codimension 4 and deviation 3 \cite{KM1},\cite{KM2}).
					\item $S_{2,2,2,2,2,2,1} F_1 \otimes S_{3} F_5^*$ \, (Huneke-Ulrich's model of Gorenstein ideals of codimension 5 and deviation 2 \cite{hu-ul5}).
				\end{itemize}
			\end{enumerate}
		\end{thm}
		
		\begin{proof}
			The case $k=1$ is straightforward (see the proof of Corollary \ref{lambda<k}). For $k=2$, if we assume $\s$ to be not an hyperplane section, we must have $\lambda_i=1$ for every $i$. We can use now the formula $2c-1= (c-1)k+1 =  \sum_{i=1}^{d+c} \lambda_i = d+c$ to get $d= c-1$. From the squares formula in Theorem \ref{squareformulas} we obtain $ 2c-1 +\sum_{i=1}^t \mu_i^2 =  4c-3. $ 
			Since $k=2$, we must have either $\boldsymbol \mu = (1,1)$ or $\boldsymbol \mu = (2)$. These two possibilities yield the desired classification.
			
			To study the case $k=3$ we observe that we have 3 cases: $\boldsymbol \mu = (1,1,1)$, $\boldsymbol \mu = (2,1)$, $\boldsymbol \mu = (3)$. Again assuming that $\s$ is not an hyperplane section, we must have $\lambda_i \leq 2$ for every $i$. For $j=1,2$, let $a_j$ be the number of $\lambda_i$'s equal to $j$. Hence we have 
			$$ 3c-2 = (c-1)k+1 =  \sum_{i=1}^{d+c} \lambda_i = 2a_2 + a_1, $$ and, by Theorem, \ref{squareformulas},
			$$   4a_2+a_1 +\sum_{i=1}^t \mu_i^2 =  9c-11.    $$ Subtracting the previous equations, we get $2a_2 = 6c -9 - \sum_{i=1}^t \mu_i^2$. The three possibilities for $ \boldsymbol \mu $ give %$a_2 \in \lbrace 3c - 9, 3c -7, 3c-6 \rbrace$. 
			$2a_2 \in \lbrace 6c - 18, 6c -14, 6c-12 \rbrace$ and therefore $a_1 \in \lbrace -3c +20, -3c +16, -3c+14 \rbrace$.
			% This allows to compute $a_1$ in terms of $c$. 
			The fact the $a_1 \geq 0$ implies that the previous equations can be satisfied only for finitely many values of $c$. Straightforward computations yield now the desired results.
		\end{proof}
		
		% if $k =4$, all but finitely many decorations are hyperplane sections, we can find the other ones by computer
		
		%if $k \geq 5$, I do not know whether this is still true, the main different with the previous cases is that $ \binom{i}{2} > i$ if $i \geq 4$. \ec \\
		
		In general for every fixed value of $k$, we expect that there exist only finitely many decorations at level $k$ that are not hyperplane sections.

		\subsection{The doubling of a decoration}
		
		In codimension 3 all Gorestein ideals are licci, but this is no more true in higher codimension. The ideal of $2 \times 2$ minors of a generic $3 \times 3$ matrix is a classical example of a non-licci Gorenstein ideal of codimension 4. An interesting problem is the classification of Herzog classes of Gorenstein licci ideals of arbitrary codimension.
		
		The goal of this section is to find a way to construct decorations of (licci) Gorenstein ideals of codimension $c+1$ from decorations of licci ideals of codimension $c$.
		
		The \it doubling \rm of an ideal is a construction which obtains a Gorenstein ideal of codimension $c+1$ from a perfect ideal of codimension $c$, under some extra assumptions \cite{MMVW}, \cite{DGA-algebra}. %Particular examples of doublings are hyperplane sections of Gorenstein ideals and Gorenstein ideals obtained as sums of geometrically linked ideals.
		
		Here we give a purely combinatoric definition of the doubling of a decoration. Conjecturally this corresponds to the decoration of the doubling in the ideal-theoretic sense, for the ideals for which the notion of doubling is well-defined. The main application of this construction will be the explicit description of an infinite family of Herzog classes of Gorenstein ideals of codimension 4 with 9 generators.
		
		\begin{definition}
			\label{defdoubling}
			Given a decoration $\s = S_{\boldsymbol \lambda}F_1 \otimes S_{\boldsymbol \mu}F_c^*$, define the \it doubling \rm of $\s$ as
			$$\mathcal{D}(\s) = S_{\boldsymbol \lambda, \boldsymbol \mu}G_1 \otimes S_{k}G_{c+1}^*, $$ after reordering the first partition in non-increasing order.
		\end{definition}
		
		\begin{thm} 
		\label{doubling}
			Let $\s= S_{\boldsymbol \lambda}F_1 \otimes S_{\boldsymbol \mu}F_c^*$ be a decoration of a licci ideal of codimension $c$. Denote by $\s'$ the minimal tight double link of $\s$. %and by $\mathcal{D}(\s)'$ the minimal tight double link of $\mathcal{D}(\s)$.
            Then, if $k=\kappa(\s)>1$, there exists a double link $\mathcal{L}'$ of $\mathcal{D}(\s)$ such that
            $ \mathcal{L}' =  \mathcal{D}(\s')$. %and $\kappa(\mathcal{L}')=\kappa(\s')< k.$
            %if $k > 1$, there exists a double link $\s''$ of $\s$ and a double link $\mathcal{L}''$ of $\mathcal{D}(\s)$ such that:
			%\begin{itemize}
				%\item[(1)] $\mathcal{L}''=\mathcal{D}(\s'')$.
				%\item[(2)] $\kappa(\s'')< \kappa(\s)=k.$
			%\end{itemize}
			In particular, $\mathcal{D}(\s)$ is a decoration of a Gorenstein licci ideal of codimension $c+1$.
		\end{thm}
		
		\begin{proof}
			First observe that if $\s$ is the decoration of a complete intersection, then $\mathcal{D}(\s)$ is the decoration of a complete intersection of codimension $c+1$. Thus we assume $k > 1$.

            Extending the formula from Example \ref{tightdouble} to codimension $c$, we obtain            		%	Let $\s'$ be the smallest minimal link of $\s$. Hence, we have $p = (c-2)k+1-\lambda_1-\ldots-\lambda_c,$ and
		%	$$ \s'= S_{\lambda_1+p, \ldots, \lambda_c+p, \mu_1, \ldots, \mu_t} F_1' \otimes S_{\lambda_{c+1}, \ldots, \lambda_{b}} (F_c')^{\ast}. $$ 
	%		We define now a minimal link $\s''$ of $\s'$ by choosing the subpartition $(\lambda_1+p, \ldots, \lambda_{c-1}+p, \mu_1 )$ (this link is well-defined but it may not be the smallest minimal).
	%		We obtain 
			$$ \s'= S_{\lambda_1+p+q, \ldots, \lambda_{c-1}+p+q, \mu_1+q, \lambda_{c+1}, \ldots, \lambda_b} F_1' \otimes S_{\lambda_{c}+p, \mu_2 \ldots, \mu_{t}} (F_c')^{\ast}, $$ where 
            $p = (c-2)k+1-\lambda_1-\ldots-\lambda_c,$ and
            $$q= (c-2)(k+p)+1-(c-1)p-\lambda_1-\ldots-\lambda_{c-1}-\mu_1.$$ By Lemma \ref{mu1lambdac}, $\lambda_{c}+p < \mu_1 $. This implies $\kappa(\s')< \kappa(\s)=k.$
			
	%We compute then the minimal tight double link to $\mathcal{D}(\s)$.
                Let us define now a first minimal link $\mathcal{L}$ of $\mathcal{D}(\s)$ with respect to the subpartition $(\lambda_1, \ldots, \lambda_{c}, \mu_1 )$. This gives
			$$ \mathcal{L}= S_{\lambda_1+p', \ldots, \lambda_c+p', \mu_1+p', k} G_1' \otimes S_{\lambda_{c+1}, \ldots, \lambda_{b}, \mu_2, \ldots, \mu_t} (G_{c+1}')^{\ast}, $$ where $p'= (c-1)k+1-\lambda_1-\ldots-\lambda_c-\mu_1 = \lambda_{c+1}+\ldots +\lambda_b - \mu_1.$ Since $\s$ does not correspond to a complete intersection, we have $b > c$ and $\mu_1+p' > 0.$ Moreover, $\lambda_c+p' = \lambda_c + p + k-\mu_1 \geq 0$, since $\lambda_c+p \geq 0$ by Lemma \ref{mu1lambdac}, and $k - \mu_1 \geq 0$ by definition of $k$. It follows that all the elements $\lambda_1+p', \ldots, \lambda_c+p', \mu_1+p'$ are non-negative and we can link again from $ \mathcal{L} $ choosing the subpartition $(\lambda_1+p', \ldots, \lambda_{c-1}+p', \mu_1+p', k)$. Therefore, we get
			$$ \mathcal{L}'= S_{\lambda_1+p'+q', \ldots, \lambda_{c-1}+p'+q', \mu_1+p'+q', k+q',\lambda_{c+1}, \ldots, \lambda_{b}, \mu_2, \ldots, \mu_t} G_1'' \otimes S_{\lambda_{c}+p'} (G_{c+1}'')^{\ast}, $$ where $q'= (c-1)(k+p')+1-cp' - \lambda_1-\ldots-\lambda_{c-1}-\mu_1 -k.$
			Observing that $p'=p+k-\mu_1$, we get 
			$$ \kappa(\mathcal{L}')= \lambda_{c}+p' = \lambda_{c}+p+k-\mu_1 = \lambda_{c}+p+\mu_2+ \ldots+ \mu_{t} = \kappa(\s'). $$
			Hence $k+p+q= \kappa(\s') = \kappa(\mathcal{L}') = k+p'+q' $, implying $p+q=p'+q'$. This obviously implies that $ \lambda_i  +p +q= \lambda_i +p'+q'$ for every $i=1, \ldots, c-1$.
			
			To prove $\mathcal{L}'=\mathcal{D}(\s')$ we still need to show that $\mu_1+ q = k+q'$ and $\lambda_c + p = \mu_1 +p'+ q'. $
			But, from the equality $q - q' = p-p' =  k-\mu_1$, we obtain $\mu_1+ q = k+q'$, while from the definitions of $q$ and $p$, we get $$ q + \mu_1 = (c-2)k+1-p -\lambda_1-\ldots-\lambda_{c-1} = \lambda_c. $$ From this we obtain
			$\mu_1 +p'+ q' = \mu_1 +p+ q = \lambda_c + p. $
			By definition of doubling, this shows that $\mathcal{L}'=\mathcal{D}(\s')$.
			
			Finally, by induction on $k$, this also proves that $\mathcal{D}(\s)$ is a decoration. Indeed, $\mathcal{D}(\s)$ is a double link of $\mathcal{D}(\s')$, which is a decoration by inductive hypothesis thanks to the fact that $\kappa(\mathcal{D}(\s'))= \kappa(\mathcal{L}')= \kappa(\s')< k.$
		\end{proof}
		
		\begin{example}
			\label{exdoubling}
			We list the following remarkable examples of doublings of decorations:
			\begin{itemize}
				\item Gorenstein ideals:  
				$ \mathcal{D}(S_{\boldsymbol \lambda} F_1 \otimes S_{k} F_c^*) = S_{k,\boldsymbol \lambda} G_1 \otimes S_{k} F_{c+1}^*. $ The doubling is the hyperplane section.
				\item Hyperplane sections:
				$ \mathcal{D}(S_{k, \boldsymbol \lambda} F_1 \otimes S_{\boldsymbol \mu} F_c^*) = S_{k,\boldsymbol \lambda, \boldsymbol \mu} G_1 \otimes S_{k} F_{c+1}^*. $ The doubling of the hyperplane section of $\s$ is the hyperplane section of the doubling of $\s$.
				\item $\mathcal{D}(S_{2,2,1,1,1}F_1 \otimes S_{2,1}F_3^*) = \mathcal{D}(S_{2,2,2,1}F_1 \otimes S_{1^3}F_3^*) =S_{2^3, 1^4} G_1 \otimes S_{3}G_4^*.$
				In codimension 3, the Anne Brown's model and the almost complete intersection of type 3 have the same doubling, that is the Kustin-Miller's model with 7 generators.
				\item $\mathcal{D}(S_{2,2,2,2,1}F_1 \otimes S_{2,2}F_3^*) = S_{2^6, 1} G_1 \otimes S_{4}G_4^*.$
				In codimension 3, the doubling of the Celikbas-Kra\'skiewicz-Laxmi-Weyman's model is the generic Gorenstein ideal of codimension 4 of format $E_7$ (see \cite{examples}, \cite[Section 5]{CGNW}).
				%\item In codimension 3, the doubling of the $G(b-3)$ family is
				%\item $\mathcal{D}(S_{2^{k-1},1^3}F_1 \otimes S_{k-1,1}F_3^*) = S_{k-1,2^{k-1},1^4} G_1 \otimes S_{k}G_4^*.$
				\item In Section \ref{sec:1683}, we observed that in codimension 3, we have infinite families of decorations of ideals with resolution of format $(1,6,8,3)$. For instance
				$$ \s_{I_n}= S_{a+2,a,a,a-2n+1, a-2n+1, a-2n+1} F_1 \otimes S_{a-n+1,a-n+1,a-n} F_3^*, $$ with $a= \frac{n(3n+1)}{2}$. The family
				$$ \mathcal{D}(\s_{I_n}) = S_{a+2,a,a, a-n+1,a-n+1,a-n, a-2n+1, a-2n+1, a-2n+1} G_1 \otimes S_{3a-3n+2} G_4^* $$ provides infinitely many decorations of licci Gorenstein ideals of codimension 4 with 9 generators. To link from $I_n$ to $I_{n-1}$ we need three smallest minimal links. Hence, to link from $ \mathcal{D}(\s_{I_n}) $ to $ \mathcal{D}(\s_{I_{n-1}}) $ we need three double links.
			\end{itemize}
		\end{example}
		
		In \cite{Ul2}, it is proved that the sum of two geometrically linked licci ideals of codimension $c$ is a Gorenstein ideal of codimension $c+1$. This construction is a particular example of a doubling. We conjecture the following result about the decorations of sums of linked ideals.
		
		\begin{conjecture}
			\label{conjectureonsums}
			Let $I$ be a licci ideal of codimension $c$. Let $\s_I = S_{\boldsymbol \lambda}F_1 \otimes S_{\boldsymbol \mu}F_c^*$.
			Suppose that $J$ is an ideal geometrically linked to $I$ such that $\kappa(I) \leq \kappa(J)$. Then, $\s_{I+J} = \mathcal{D}(\s_I). $
			%$$ \s_{I+J} = S_{\boldsymbol \lambda, \boldsymbol \mu}G_1 \otimes S_{k}G_{c+1}^*  $$ after reordering the partitions in non-increasing order. 
		\end{conjecture}

		\subsection{Gorenstein ideals of codimension 4}
		\label{sec:gor4}
		Structure theorems for Gorenstein ideals of codimension 4 have been studied in many papers, including \cite{DGA-algebra}, \cite{KM1}, \cite{KM2},\cite{vv}, \cite{reid}.
		
		Recently, a theory relating free resolutions of Gorenstein ideals of codimension 4, generic rings, and Kac-Moody Lie algebras, analogous to the one established in codimension 3, has been addressed in \cite{weyman-gorenstein}, \cite{CGNW}. The Kac-Moody Lie algebra associated to the format of resolution of a Gorenstein ideal of codimension 4 is of type $E_n$ where $n$ is the number of minimal generators of the ideal. It follows that the corresponding diagram is Dynkin only for $n \leq 8$ and therefore there are only finitely many Herzog classes of Gorenstein ideals of codimension 4 with $n\leq 8$ generators.
		Moreover, it is expected that all such Gorenstein ideals are licci. 
		
		We can explicitly list the pair of partitions associated to their Herzog classes. We can start from the case $n=6$, since, for $n=4$ the ideals are complete intersections, and there are no Gorenstein ideals of codimension 4 with 5 generators because of a theorem of Kunz \cite{kunz} stating that almost complete intersections are never Gorenstein.
		
		For $n=6$, it is proved in \cite{CGNW}, that there is only one Herzog class, that correspond to the hyperplane section of a Gorenstein ideal of codimension 3 with 5 generators and the associated partition is 
		$S_{2,1^5}F_1 \otimes S_{2}F_4^*$.
		
		For $n=7$, there are two Herzog classes: $S_{2^3, 1^4} F_1 \otimes S_{3}F_4^*$
		and $S_{2^6, 1} F_1 \otimes S_{4}F_4^*.$ As already pointed out in Example \ref{exdoubling}, they correspond to the Kustin-Miller's model from \cite{KM1},\cite{KM2} and to the generic model computed in \cite{examples}, and they both appear as doublings (in the sense of Definition \ref{defdoubling}).
		
		In the case $n=8$, there are 10 Herzog classes. The associated pairs of partitions are 
		$S_{3, 1^7} F_1 \otimes S_{3}F_4^*$, 
		$S_{3, 2^3, 1^4} F_1 \otimes S_{4}F_4^*$, 
		$S_{3^2, 2^4, 1^2} F_1 \otimes S_{5}F_4^*$, 
		$S_{4,3, 2^6} F_1 \otimes S_{6}F_4^*$, 
		$S_{3^4, 2^3, 1} F_1 \otimes S_{6}F_4^*$, 
		$S_{4, 3^4, 2^3} F_1 \otimes S_{7}F_4^*$, 
		$S_{3^7, 1} F_1 \otimes S_{7}F_4^*$, 
		$S_{4^2, 3^5, 2} F_1 \otimes S_{8}F_4^*$,
		$S_{4^4,3^4} F_1 \otimes S_{9}F_4^*$, 
		$S_{4^7,3} F_1 \otimes S_{10}F_4^*$.
		They all appear as doubling (in the sense of Definition \ref{defdoubling}) of perfect ideals of codimension 3 and format $E_7$, except the last one.
		
		In Example \ref{exdoubling}, it is described an infinite family of Herzog classes of licci Gorenstein ideals of codimension 4 with 9 generators.
		
		From linkage theory we obtain that there exist only finitely many Herzog classes of almost complete intersections of codimension 4 and type $\leq 4$. Those of type $2$ have associated pairs of partitions $S_{2,2,1,1,1} F_1 \otimes S_{1,1}F_4^*$
		and $S_{2^5} F_1 \otimes S_{2,1}F_4^*.$ Those of type $3$ and $4$ are many more, but can still be calculated algorithmically.

		\subsection{Gorenstein ideals of deviation 2}
        \label{sec:dev2}
		
		Gorenstein ideals of deviation 2 have been object of the investigation of \cite{herzog-miller}, \cite{hu-ul5}.
		
		We expect that the theory of Kac-Moody Lie algebras associated to formats of free resolutions can be described also in this case. %of Gorenstein ideals of deviation 2. In this case,
        There are finitely many Herzog classes of Gorenstein ideals of deviation 2 if the codimension $c$ of the ideals is $\leq 6$. We list the associated pairs partitions for $c=5,6$ (the case $c=3,4$ have been already covered previously). We have:
		$S_{2,2,1^5} F_1 \otimes S_{2}F_5^*, $
		$S_{2^6,1} F_1 \otimes S_{3}F_5^*$ ($c=5$, format $E_7$), and 
		$S_{2^3,1^5} F_1 \otimes S_{2}F_6^*,$
		$S_{3,2^6,1} F_1 \otimes S_{3}F_6^*,$
		$S_{3^5,2^3} F_1 \otimes S_{4}F_6^*,$
		$S_{4^2,3^6} F_1 \otimes S_{5}F_6^*,$
		$S_{4^7,3} F_1 \otimes S_{6}F_6^*$ ($c=6$, format $E_8$).
		
		We have been able to produce an infinite family of Herzog classes of Gorenstein ideals of deviation 2 and codimension 7. Let $a$ be a nonnegative integer. Define $\s_{a}$ to be 
        $$  S_{ \frac{(3a^2 + 7a + 4)}{2}, 
				\frac{(3a^2 + 7a + 4)}{2}, 
				\frac{(3a^2 + 7a + 4)}{2}, 
				\frac{(3a^2 + 6a + 4)}{2}, 
				\frac{(3a^2 + 6a + 2)}{2}, 
				\frac{(3a^2 + 6a + 2)}{2}, 
				\frac{(3a^2 + 5a + 2)}{2},
				\frac{(3a^2 + 5a + 2)}{2},
				\frac{(3a^2 + 5a + 2)}{2}  }
			F_1
			\otimes S_{ \frac{ (9a^2 + 18a + 8)}{4} }F_7^* $$ if $a$ is even, and 
            $$ S_{ \frac{(3a^2 + 7a + 4)}{2}, 
				\frac{(3a^2 + 7a + 4)}{2}, 
				\frac{(3a^2 + 7a + 4)}{2}, 
				\frac{(3a^2 + 6a + 5)}{2}, 
				\frac{(3a^2 + 6a + 3)}{2}, 
				\frac{(3a^2 + 6a + 3)}{2}, 
				\frac{(3a^2 + 5a + 2)}{2},
				\frac{(3a^2 + 5a + 2)}{2},
				\frac{(3a^2 + 5a + 2)}{2} }
			F_1
			\otimes S_{ \frac{(9a^2 + 18a + 9)}{4}  }F_7^*   $$ if $a$ is odd.

		For $a = 0$, this yields $S_{2^4,1^5}F_1 \otimes S_2F_7^*$, which corresponds to a quadruple hyperplane section of a Gorenstein ideal of codimension 3 and deviation 2. For every $a \geq 1$, $\s_{a}$ can be obtained from $\s_{a-1}$ by applying $6$ consecutive largest minimal links.

		%lambda is the partition with
		%3 parts given by (3a^2 + 7a + 4)/2
		%1 part given by (3a^2 + 6a + 4)/2 if a is even and (3a^2 + 6a + 5)/2 if a is odd
		%2 parts given by (3a^2 + 6a + 2)/2 if a is even and (3a^2 + 6a + 3)/2 if a is odd
		%3 parts given by (3a^2 + 5a + 2)/2
		%mu is the partition with a single part given by (9a^2 + 18a + 8)/4 if a is even and (9a^2 + 18a + 9)/4 if a is odd.
		
		%For a = 0, this yields (2,2,2,2,1,1,1,1,1), (2), which (if the theory works as we expect) should correspond to a quadruple hypersurface section of (1,5,5,1).
		
		%To go from S_a to S_(a+1), the sequence of links is
		%S_a ~ a.c.i. ~ gor ~ a.c.i. ~ gor ~ a.c.i. ~ S_(a+1)
		%where every link is largest minimal

		Conjecturally, all Gorenstein ideals of deviation 2 and codimension $c \leq 6$ are licci. We do not know whether there exist Gorenstein ideals of deviation 2 and codimension $c \geq 7$ that are not licci. We leave this as an open question.

\subsection{Licci conjecture for $(c,d,t)$ triplets}
        To conclude this paper, we summarize the expected results and the state of the art regarding the question of which perfect ideals of codimension $c$, deviation $d$, and type $t$ are licci.

        \begin{conjecture}
            \label{genlicciconj}
            Fix $c\geq 3$, $d\geq 1$, $t \geq 1$. Then:
            \begin{itemize}
                \item[(a)] If the diagram $T_{c-1,d+1,t+1}$ is Dynkin (of type ADE), then every perfect ideal of codimension $c$, deviation $d$, and type $t$ is licci.
                \item[(b)]  If the diagram $T_{c-1,d+1,t+1}$ is not Dynkin, then there exist perfect ideals of codimension $c$, deviation $d$, and type $t$ that are not licci.
            \end{itemize}
        \end{conjecture}

        In the case $c=3$, both parts of the conjecture have positive answer \cite{CVWdynkin}, \cite{GNW2} \cite{GNW3}. In an upcoming paper, we plan to present a positive answer also in the case $c=4$, $t=1$ (Gorenstein ideals of codimension 4). By taking hyperplane sections, it is easy to observe that if the triplet $(c,d,t)$ admits perfect non-licci ideals, so does the triplet $(c+1,d,t)$. Moreover, it is not hard to exhibit non-licci perfect ideals corresponding to triplets $(4,4,2),(4,2,4)$.
        Hence, to conclude this study in codimension 4, it will be required to analyze the triplets $(4,2,2), (4,3,2), (4,2,3)$, and discover if they admit any non-licci perfect ideal. As pointed out at the end of Section \ref{sec:dev2}, in higher codimension, also the existence of non-licci perfect ideals for the triplets $(c,2,1)$ is not known. 
		
		\section*{Acknowledgements}
		
		This material is partially based upon work supported by the National Science Foundation under Grant No. DMS-1928930 and by the Alfred P. Sloan Foundation under grant G-2021-16778, while the authors were in residence at the Simons Laufer Mathematical Sciences Institute (formerly MSRI) in Berkeley, California, during the Spring 2024 semester. 
		
		%T. Chmiel, 
		L. Guerrieri and J. Weyman are supported by the grants: \\ MAESTRO NCN-UMO-2019/34/A/ST1/00263 - Research in Commutative Algebra and Representation Theory, \\
		NAWA POWROTY- PPN/PPO/2018/1/00013/U/00001 - Applications of Lie algebras to Commutative Algebra,  \\
		OPUS grant National Science Centre, Poland grant UMO-2018/29/BST1/01290. 
		
		L. Guerrieri is also supported by the Miniatura grant 2023/07/X/ST1/01329 from NCN (Narodowe Centrum Nauki), which funded his visit to SLMath in April 2024.
		
		The authors would like to thank Ela Celikbas, Tymoteusz Chmiel, Lars Christensen, David Eisenbud, Sara Angela Filippini, Craig Huneke, Witold Kra\'skiewicz, Andrew Kustin, Jai Laxmi, Claudia Polini, Steven Sam, Jacinta Torres, Bernd Ulrich, and Oana Veliche for interesting discussions
		pertaining to this paper and related topics.

		%\end{thebibliography}

		%\bibitem{Br83} 
		%Bruns, W.
		%Divisors on varieties of complexes,
		%Math. Ann.
		% 264 (1983) 53-71

		%\bibitem{BE73} 
		%Buchsbaum, D., Eisenbud, D.
		%\it  What makes a complex exact, \rm
		%J. of Algebra.
		%25
		%(1973)
		%259-268

		%\bibitem{Gaeta}
		%Gaeta, F.
		%\it Quelques progr\`es r\'ecents dans la classification des vari\'et\'es alg\'ebriques d'un espace projectif, \rm
		%Deuxi\`eme Colloque de G\'eometrie Alg\'ebrique Li\`ege (1952), C.B.R.M.

		%\bibitem{homconj} 
		% Hochster, M.
		%\it Homological conjectures, old and new, \rm
		%Illinois J. Math. 51(1): 151-169 (Spring 2007).

		%\bibitem{JPW81}  
		% J\'ozefiak, T.,  Pragacz, P., Weyman, J.
		%Resolutions of determinantal varieties and tensor complexes associated with symmetric
		%and antisymmetric matrices,
		%Asterisque
		%87-88
		%(1981)
		%109-189

		%\bibitem{Macaulay}
		%Macaulay, F.S.
		%\it On the Resolution of a given Modular System into Primary Systems including some Properties of Hilbert Numbers, \rm
		%Mathematische Annalen (1913), Vol. 74, pp. 66-121.

		% \bibitem{Peskine-Szpiro-4pages}
		% Peskine, C., Szpiro, L.
		%\it Syzygies et multiplicit\'es, \rm
		%C. R. Acad. Sci. Paris S\'er. A Math. 278
		%(1974), 1421–1424.

		%\bibitem{PW90}  
		% Pragacz, P., Weyman, J.
		%\it  On the generic free resolutions, \rm
		%J. of Algebra 128, no.1 (1990) 1-44.

		%\bibitem{W03}   Weyman, J. 
		%\it Cohomology of vector bundles and syzygies, \rm
		% Cambridge University Press, Cambridge, UK 
		%(2003) 
		%Cambridge Tracts in Mathematics, vol. 149

		%\end{thebibliography}

	\end{document}